\documentclass[10pt]{amsart}

\usepackage{times,amsfonts,amsmath,amstext,amsbsy,amssymb,
  amsopn,amsthm,upref,eucal, amscd,color,url,caption}
\usepackage[T1]{fontenc}

\usepackage[pdftex]{graphicx}

\usepackage{tikz}
\usetikzlibrary{arrows}

\usepackage{longtable}
\usepackage{array}

\newtheorem{theorem}{Theorem}[section]
\newtheorem{lemma}[theorem]{Lemma}
\newtheorem{claim}[theorem]{Claim}
\newtheorem{corollary}[theorem]{Corollary}
\newtheorem{conjecture}[theorem]{Conjecture}
\newtheorem{proposition}[theorem]{Proposition}

\numberwithin{equation}{section}

\theoremstyle{definition}
\newtheorem{definition}[theorem]{Definition}
\newtheorem{example}[theorem]{Example}

\newtheorem{remark}[theorem]{Remark}
\newtheorem{question}[theorem]{Question}

\def\Cset{\mathbb{C}}
 \def\Hset{\mathbb{H}}
 \def\Nset{\mathbb{N}}
 \def\Qset{\mathbb{Q}}
 \def\Rset{\mathbb{R}}
 \def\Zset{\mathbb{Z}}
 \def\Tset{\mathbb{T}}

\def\Pset{\mathbb{P}}
\def\Gset{\mathbb{G}}
\def\Kset{\mathbb{K}}
\def\va{ \varepsilon}
\def\wt{\widetilde}
\def\wh{\widehat}
\def\leq{\leqslant }
\def\geq{\geqslant}

\def\la{\lambda}
\def\lai{\lambda^{-1}}
\def\ula{\underline \lambda}
\def\umu{\underline \mu}

\newcommand{\cH}{\mathcal H}

\newcommand{\SL}{{\mathrm{SL}}}
\newcommand{\GL}{{\mathrm{GL}}}
\newcommand{\Gal}{{\mathrm{Gal}}}
\newcommand{\Sp}{{\mathrm{Sp}}}

\newcommand{\red}{{\operatorname{red}}}

\newcommand{\diag}{{\mathrm{diag}}}
\newcommand{\Teichmuller}{{Teich\-m\"uller }}

\begin{document}
\title{A criterion for the simplicity of the Lyapunov spectrum of square-tiled surfaces}

\author{Carlos Matheus}
\address{Carlos Matheus: Universit\'e Paris 13, Sorbonne Paris Cit\'e, LAGA, CNRS (UMR 7539), F-93439, Villetaneuse, France.}
\email{matheus@impa.br.}
\urladdr{http://www.impa.br/$\sim$cmateus}

\author{Martin M\"oller}
\address{Martin M\"oller: Institut f\"ur Mathematik, Goethe-Universit\"at Frankfurt, Robert-Mayer-Str. 6-8, 
60325 Frankfurt am Main, Frankfurt, Germany.}
\email{moeller@math.uni-frankfurt.de.}
\urladdr{https://www.uni-frankfurt.de/fb/fb12/mathematik/ag/personen/moeller}

\author{Jean-Christophe Yoccoz}
\address{Jean-Christophe Yoccoz: Coll\`ege de France (PSL), 3, Rue d'Ulm, 75005 Paris, France.}
\email{jean-c.yoccoz@college-de-france.fr.}
\urladdr{http://www.college-de-france.fr/site/jean-christophe-yoccoz}

\date{\today}

\begin{abstract}
We present a Galois-theoretical criterion for the simplicity of the Lyapunov spectrum of the Kontsevich-Zorich cocycle over the \Teichmuller flow on the $\SL_2(\Rset)$-orbit of a square-tiled surface. The simplicity of the Lyapunov spectrum has been proved by A. Avila and M.Viana with respect to the so-called Masur-Veech measures associated to connected components of  moduli spaces of translation surfaces, but is not always true for square-tiled surfaces of genus $\geq 3$. We apply our criterion to square-tiled surfaces of genus $3$ with one single zero. Conditionally to a conjecture of Delecroix and Leli\`evre, we prove with the aid of Siegel's theorem (on integral points on algebraic curves of genus $>0$) that all but finitely many such square-tiled surfaces have simple Lyapunov spectrum.
\end{abstract}
\maketitle

\setcounter{tocdepth}{1}
\tableofcontents

\section{Introduction}\label{s.introduction}

\subsection{The KZ cocycle} \label{cocycle}

The \Teichmuller flow on the moduli spaces of translation surfaces acts as renormalization dynamics for 
interval exchange transformations and translation flows. The Kontsevich-Zorich cocycle describes how the 
homology of the surface evolves along the orbits of the \Teichmuller flow. The seminal works of A. Zorich \cite{Zo1}, \cite{Zo2}, \cite{Zo3}, \cite{Zo4}, \cite{Zo5} and G. Forni \cite{Fo1} have explicited an intimate relation between  the  deviations of Birkhoff sums from ergodic averages for interval exchange maps and translation flows and the 
Lyapunov spectrum of the Kontsevich-Zorich (KZ for short) cocycle.

\par


 The connected components of  moduli spaces of unit area translation surfaces support natural probability 
 measures invariant under the \Teichmuller flow, the so-called Masur-Veech measures $\mu_{MV}$.
 Based on numerical experiments, M. Kontsevich and A. Zorich \cite{K} conjectured that Lyapunov spectra of KZ cocycle with respect to these measures are always simple, i.e. all Lyapunov exponents have multiplicity 1.
 
 \par
G. Forni obtained an important partial result \cite{Fo1} in this direction: he proved that the KZ-cocycle is 
non-uniformly hyperbolic w.r.t. $\mu_{MV}$, i.e the Lyapunov exponents are different from $0$. 
Then the conjecture of M. Kontsevich and A. Zorich was fully proved in the celebrated work of A. Avila and M. Viana \cite{AVKZ}.

\par


On the other hand, G. Forni and his coauthors (see \cite{Fo2} and \cite{FMZ}) constructed two examples of 
translation surfaces with the following property: their orbits under the natural $\SL_2(\Rset)$-action on  
moduli space are closed and support $\SL_2(\Rset)$-invariant probability measures with respect to  which all 
nontrivial Lyapunov exponents of the KZ cocycle vanish. In particular, these Lyapunov spectra  are far from simple. 

\par

Partly motivated by this discussion, G. Forni \cite{Fo3} recently provided a criterion for the non-uniform hyperbolicity of the Lyapunov spectrum of the
KZ cocycle with respect to a given $\SL_2(\Rset)$-invariant ergodic probability measure $\mu$ essentially based on the geometry of the horizontal foliation of the translation surfaces
in the support of $\mu$. However, as it is pointed out in \cite{Fo3} by means of concrete examples, this geometric criterion is not sufficient to ensure the simplicity of the Lyapunov spectrum in general.

\subsection{ Main results} \label{origamis}

As a matter of fact, the examples of non-simple Lyapunov spectrum in \cite{Fo2}, \cite{FMZ} and \cite{Fo3}  
come from a  class of translation surfaces $(M,\omega)$ called \emph{square-tiled surfaces} 
or {\it origamis}.
\par
Our main result, presented below, is a Galois-theoretical criterion for the simplicity of the Lyapunov spectrum of the KZ-cocycle, with respect to the natural $\SL_2(\Rset)$-invariant probability measure associated to an origami.

\par
A translation surface $(M,\omega)$ is square-tiled if its group of relative periods is contained in $\Zset \oplus i \Zset$. Equivalently, there is a ramified covering $\pi: M \to \Tset^2\equiv\mathbb{C}/(\mathbb{Z}\oplus i\mathbb{Z})$, unramified over $\Tset^2 - \{0\}$, such that 
$\omega$ is the pull back of the standard form $dz$ on $\Tset^2$. The square-tiled surface $(M,\omega)$ is
{\it reduced} if its group of relative periods is equal to $\Zset \oplus i \Zset$.

\par
Let $(M,\omega)$ be a translation surface. An  orientation-preserving homeomorphism $A$ of $M$ is
 {\it affine} if it is given locally by affine maps in the charts provided by local primitives of $\omega$. 
 Then, it has a {\it linear part} or {\it derivative}  $DA \in \SL_2(\Rset)$. The derivatives of affine
  homeomorphisms form a subgroup of $\SL_2(\Rset)$ called the {\it Veech group} of $(M,\omega)$. 
  When $(M,\omega)$ is  reduced square-tiled, the Veech group is a subgroup of finite index in 
  $\SL_2(\Zset)$. An affine homeomorphism is an {\it automorphism} of 
$(M,\omega)$ if its derivative is the identity. 
\par
Let $(M,\omega)$ be a reduced square-tiled surface. There is a canonical splitting 

$$H_1(M,\mathbb{Q})=H_1^{st}(M,\mathbb{Q})\oplus H_1^{(0)}(M,\mathbb{Q}),$$ 
which is invariant under the action of any affine homeomorphism. Here $H_1^{(0)}(M,\mathbb{Q})$ is the 
$2$-codimensional kernel of the homomorphism $\pi_*: H_1(M,\mathbb{Q}) \to H_1(\Tset^2,\mathbb{Q})$ 
and the summands are orthogonal for the symplectic intersection form. The action of an affine 
homeomorphism $A$ on $H_1^{st}(M,\mathbb{Q}) \simeq H_1(\Tset^2,\mathbb{Q}) \simeq \Qset^2 $ is through the standard action of $DA \in \SL_2(\Zset)$ and correspond to the tautological extremal Lyapunov exponents of the KZ-cocycle.

\par
The action of $A$ on $H_1^{(0)}(M,\mathbb{Q})$ preserves the symplectic intersection form. It is given by 
a symplectic matrix with integer coefficients\footnote{See however Remark \ref{r.Sympl}.}. Its characteristic polynomial 
is reciprocal of degree $2g-2$ (as usual, $g$ denotes the genus of $M$). One says that $A$ is {\it Galois-pinching} if the characteristic polynomial 
is irreducible over $\Qset$, has only real roots, and its Galois group is largest possible, with order 
$2^{g-1}(g-1)!$ (being then isomorphic to the semi-direct product 
$S_{g-1}\rtimes(\mathbb{Z}/2\mathbb{Z})^{g-1}$, acting on the set of roots as the centralizer of the involution $\lambda \to \lambda^{-1}$).

\par
We can now formulate a first version of our simplicity criterion.

\begin{theorem}\label{thm:mainsimpcrit}
Let $(M,\omega)$ be a reduced square-tiled surface having no nontrivial automorphism. Assume that there
 exist two affine homeomorphisms $A,B$ of $(M,\omega)$ with the following properties:
 \begin{itemize}
 \item[i)] $A$ is Galois-pinching and $DA$ has trace $>2$ ;
 \item[ii)] $B$ acts on $H_1^{(0)}(M,\mathbb{Q})$ through an unipotent endomorphism distinct from the identity,
 such that the image of $B- \textrm{id}$ is not a lagrangian subspace of $H_1^{(0)}(M,\mathbb{Q})$.
 \end{itemize}
Then, the Lyapunov spectrum of the KZ-cocycle, relative to the $\SL_2(\Rset)$-invariant probability measure supported by the $\SL_2(\Rset)$-orbit of $(M,\omega)$ in moduli space, is simple.
\end{theorem}

Let $(M,\omega)$ be an origami.  The union of the zeros of $\omega$ and the saddle-connections in a 
given rational direction disconnect $M$ into a finite number of cylinders. The dimension of the (isotropic) 
subspace of  $H_1(M,\mathbb{Q})$ spanned by the classes of the waist curves of these cylinders is the {\it homological dimension} of the given rational direction. It takes values in $\{1,\ldots,g\}$.

\begin{corollary}\label{cor:mainsimpcrit}
Let $(M,\omega)$ be a reduced square-tiled surface. Assume that
 \begin{itemize}
 \item [i)] $(M,\omega)$ has no nontrivial automorphisms;
 \item[ii)] there exists a Galois-pinching affine homeomorphism $A$  with 
 $\textrm{tr} (DA) > 2$; 
  \item[iii)] there exists a rational direction with homological dimension  $\ne 1,g$.
   \end{itemize}
  Then, the Lyapunov spectrum of the KZ-cocycle, relative to the $\SL_2(\Rset)$-invariant probability measure supported by the $\SL_2(\Rset)$-orbit of $(M,\omega)$ in moduli space, is simple.
\end{corollary}

In Section \ref {sec:redtoAV}, we give another version of the criterion which involves two affine homeomorphism
with hyperbolic linear part (Theorem \ref{thm:variant}). 

\medskip

In the second part of the paper, we give some application of our simplicity criterion.

\par

For a square-tiled surface $(M,\omega)$ of genus $2$, the Lyapunov exponents of the KZ-cocycle (w.r.t. 
 to the $\SL_2(\Rset)$-invariant probability measure supported by the $\SL_2(\Rset)$-orbit of $(M,\omega)$)
 are known (cf. \cite{Bainbridge}): the nontrivial exponents are 
 $\pm \frac 13$ if $\omega$ has a double zero, $\pm \frac 12$ if $\omega$ has two simple zeros. In other 
 words, they are equal to the Lyapunov exponents relative to the Masur-Veech measure of the 
 stratum containing $(M,\omega)$.
 
 \par
In this paper, we consider the simplest case where the Lyapunov exponents are not known: origamis 
$(M,\omega)$ of genus $3$ such that $\omega$ has a single zero (of order $4$). The corresponding  moduli space is denoted by $\mathcal H(4)$. Such origamis only have the identity as an automorphism 
(Proposition \ref {no_automorphism}). Reduced origamis in $\mathcal H(4)$ fall in three different classes, 
 related to the existence and the properties of an affine homeomorphism whose linear part is 
 $-\textrm{id}$, called anti-automorphism  below; because of Proposition~\ref{no_automorphism}, there is at most one such anti-automorphism.
\begin{itemize}
\item Origamis with no anti-automorphism are said to be of odd type\footnote{Odd here denotes the parity of a spin structure which is crucial in the Kontsevich-Zorich  classification \cite {KoZo03} of connected components of strata.}.
\item An anti-automorphism has either $4$ or $8$ fixed points. The origami is said to be of Prym type\footnote{See Subsection \ref {ccH(4)}.} in the first case, of hyperelliptic type in the second case.
\end{itemize}
 
 \par
 The non trivial Lyapunov  exponents for an origami of Prym type are known to be equal to
  $\pm \frac 15, \pm \frac 25$ (see Subsection \ref {ss:Prym}), from   \cite{chenmoeller} and  \cite{ekz}. The 
  crucial fact, that allows to determine exactly the exponents, is that, for an origami of Prym type, there is a splitting of $H_1^{(0)}$ into two $2$-dimensional summands which are invariant under any affine homeomorphism. 
  \par
 The Lyapunov exponents do not change when one replaces an origami by another in the same 
 $\SL_2(\Zset)$-orbit. To classify $\SL_2(\Zset)$-orbits of origamis, several invariants have been introduced:
 \begin{itemize}
 \item A trivial invariant is the number of squares, i.e the degree of the ramified covering $\pi: M \to \Tset^2$.
 \item When there exists an anti-automorphism, its fixed points project to points of order $2$ in $\Tset^2$. 
 The distribution of the projections of fixed points was first considered by E. Kani \cite{kani} and 
 P. Hubert-S. Leli\`evre \cite{hubertlelievre}) in genus $2$. This {\it HLK-invariant} for $\mathcal H(4)$ is 
 described more precisely in Subsection \ref{HLK}.
 \item In all cases, following D. Zmiaikou (\cite {zmia}), it is possible to associate to an $N$-square origami a
  subgroup of $S_N$ called the {\it monodromy group} (see Subsection \ref{ss.sts}). Actually, for a large number of squares ($ N \geq 7$ 
 for  $ \mathcal H(4)$), the monodromy group is either the full symmetric group $S_N$ or the alternating group
  $A_N$ (\cite[Theorem~3.12]{zmia}). 
\end{itemize}

Supported by some numerical experiments with Sage, V. Delecroix and S. Leli\`evre have conjectured that the
 invariants above are sufficient to classify $\SL_2(\Zset)$-orbits in the odd and hyperelliptic cases 
 (the Prym case had been settled earlier by E. Lanneau and D.-M. Nguyen~\cite{manhlann}). More precisely,
  they expect that, for $N >8$
  \begin{itemize}
\item There are two orbits of $N$-square reduced origamis of odd type, associated to the two possibilities for the monodromy group.
\item There are four  (for odd $N$) or three (for even $N$) orbits of $N$-square reduced origamis of hyperelliptic type, associated to the possible values of the HLK-invari\-ant.
  \end{itemize}
  
  The complete statement of the conjecture is given in Subsection \ref{DLconj}. Our result for $\mathcal H(4)$ 
  is as follows.
  
  \begin{theorem}\label{thm:introH4}
  For any large enough integer $N$, there exist two  $N$-square reduced origamis of odd type in $\mathcal H(4)$ 
  with simple Lyapunov spectra whose monodromy groups are respectively the full symmetric group $S_N$ 
  and the alternating group $A_N$.
  
  For any large enough integer $N$, and any realizable value of the HLK-invariant, there exists 
  a  $N$-square reduced origami of hyperelliptic type  in $\mathcal H(4)$ with simple Lyapunov spectrum having the prescribed
  HLK-invariant. 
  \end{theorem}

 \begin{corollary} \label{cor:introH4}
If the Delecroix-Leli\`evre conjecture holds, then the Lyapunov spectrum of the KZ-cocycle for all but 
(possibly) finitely many reduced square-tiled surfaces in $\mathcal{H}(4)$ is simple.
\end{corollary} 

\subsection{Questions and comments}  \label{comments}

\begin{remark} \label{rem0}
A variant of Theorem \ref{thm:mainsimpcrit} (stated as Theorem \ref{thm:variant} below) was used by V. Delecroix and the first author \cite{DeMa} to show that there is no general converse to G. Forni's geometrical criterion for non-uniform hyperbolicity \cite{Fo3}.
\end{remark}

\begin{remark}\label{rem1}
The exact value of the Lyapunov exponents of the KZ cocycle are not known for a primitive square-tiled surface
 in $\mathcal H(4)$ except in the Prym case. On the other hand, Chen and M\"oller \cite{chenmoeller} have 
 shown that the {\it sum}  of the nonnegative exponents  depends only of the connected component of the 
 moduli space which contains the surface, and is equal to the sum of the nonnegative exponents for the 
 Masur-Veech measure of this component. The sum of the nonnegative exponents for  Masur-Veech 
 measures can be computed explicitly from Siegel-Veech constants \cite {ekz}. The sum of the nontrivial nonnegative exponents is therefore equal to $\frac 35$ for square-tiled surfaces of odd or Prym type, and to $\frac 45$ for origamis of hyperelliptic type.
\end{remark}

\begin{remark}\label{rem2}
While numerical methods to estimate the values of the Lyapunov exponents of the KZ-cocycle are quite
 effective for the Masur-Veech measures associated to components of the moduli space, they are much less so
  for  the natural measures associated to individual square-tiled surfaces. For origamis in $\mathcal H(4)$
   (not of Prym type), only the first two decimal places can be guaranteed with some degree of confidence. 
  Within these limitations, no variation can be detected numerically for the Lyapunov exponents of origamis 
  of the same type in  $\mathcal H(4)$. 
   \end{remark}

\begin{remark}\label{rem3}
How effective are Theorem \ref{thm:introH4} and Corollary \ref{cor:introH4}? The short answer is that they are
 effective in principle but not in a practical way. The construction in Sections \ref{s.odd} and \ref{s.hyperelliptic} of 
 one-parameter families of origamis with some of the required properties 
 (prescribed monodromy group or HLK-invariant, rational direction of homological dimension $2$)
  is valid as soon as the number of squares (which is an affine function of the parameter) is not too small. 
 These origamis are equipped with an affine homeomorphism $A$ satisfying $\textrm{tr} (A) >2$. 
 In order to apply Corollary  \ref{cor:mainsimpcrit}, we have to prove that $A$ is Galois-pinching. 
 After some elementary Galois theory, this is equivalent to show that three quantities, which are 
 explicit polynomials with integer coefficients in the parameter, are not squares. This can be done quite 
 explicitly for the first quantity. However, to deal with the other two quantities, we have to appeal to 
 Siegel's theorem (see, e.g., \cite{HiSi}) on the finiteness of integral points on algebraic curves of genus $>0$. 
 \par
 Siegel's theorem admits effective versions (see for instance \cite{Bilu}) but the bounds on the height of the 
 integral points are currently, as far as the authors know, doubly exponential in the size of the coefficients of the 
 polynomials involved. Thus, even assuming the conjecture of Delecroix and Leli\`evre, a proof of the simplicity of the Lyapunov spectrum for {\it all} origamis in $\mathcal H(4)$ along these lines (completing Corollary 
 \ref{cor:introH4} by a numerical investigation of the finitely many remaining origamis) is hopeless.
 
\end{remark}

The invariance of the sum of nonnegative Lyapunov exponents for origamis in the same connected component of moduli space, as exemplified by the results of Bainbridge in genus $2$ and of Chen-M\"oller (Remark \ref{rem1}) for $\mathcal H(4)$, does not
 extend to all moduli spaces. For instance, the moduli space of genus $3$ translation surfaces with $4$ simple 
 zeroes contains the example cited earlier \cite {Fo2} with totally degenerate  Lyapunov spectrum.
 
 \par
 In genus $\geq 3$, the relation between the Lyapunov spectra of the KZ-cocycle w.r.t.  the probability measure
  associated to a square-tiled surface $(M,\omega)$ and w.r.t.  the Masur-Veech measure associated with the
   connected component of moduli space containing $(M,\omega)$ is poorly understood. The only general result
   is due to Eskin: in the appendix of \cite{chencovers}, he proves that, as $N$ goes to $+\infty$, the \emph{average}  of the {\em sum} of  nonnegative Lyapunov
exponents   over all $N$-square origamis in a component of some moduli space converges
to the \emph{sum} of Lyapunov exponents w.r.t. the Masur-Veech measure of this component.

In connection with the results mentioned above and the recent works  \cite{em} and \cite{emm}, the following
 question looks natural. Consider a sequence of square-tiled surfaces $(M_n, \omega_n)$ in some connected 
 component of some moduli space of translation surfaces. Denote by $\mu_n$ the natural invariant probability
  measure on the $\SL_2(\Rset)$-orbit of $(M_n, \omega_n)$ in moduli space. Assume that the sequence 
 $\mu_n$ converges in the weak topology to some probability measure $\mu$. Then, by Theorem 2.3 of 
 \cite{emm}, the measure $\mu$  is ergodic  and the support of $\mu$ contains $(M_n,\omega_n)$ for all large $n$.
 
\begin{question} Does  {\it each} Lyapunov exponent w.r.t. $\mu_n$ converge towards the corresponding exponent
w.r.t. $\mu$?
\end{question}

\begin{remark} In the context of the question, it was pointed to us by A. Eskin that the sum of the non-negative Lyapunov exponents w.r.t. $\mu_n$ converges towards the sum of the non-negative Lyapunov exponents w.r.t. $\mu$. The argument runs along the following lines. The Kontsevich-Forni formula (cf. \cite{ekz} and  \cite{Fo1}) expresses the sum of the Lyapunov exponents w.r.t. an ergodic $\SL(2,\mathbb{R})$-invariant probability measure $\nu$ supported on some connected component $\mathcal{C}$ of moduli space as the $\nu$-integral of a function $\Lambda$ which is bounded and continuous on $\mathcal{C}$. Moreover, from the work of A. Eskin and H. Masur \cite{EsMas}, there exists for each $\varepsilon>0$ a compact set $K_{\varepsilon}\subset\mathcal{C}$ such that $\nu(K_{\varepsilon})>1-\varepsilon$ for any such probability measure $\nu$. The asserted convergence follows.
\end{remark}

\subsection {Outline of the paper}\label{outline}

 We recall in Section~\ref{sec:background} the definition and some elementary properties of 
translation surfaces, the KZ-cocycle, and square-tiled surfaces. Then we state as Theorem \ref{thm:AVcrit} a variant 
of the simplicity criterion of A. Avila and M. Viana \cite{AVKZ} for locally constant integrable 
cocycles over full shifts on a countable alphabet, equipped with measures 
with bounded distortion. The simplicity criterion says that such cocycles have simple Lyapunov spectrum 
whenever they verify two conditions called pinching and twisting. The version of the simplicity 
criterion used here differs from the original one 
by A. Avila and M. Viana in a few details, such as the precise statement of the twisting property. 
For the convenience of the reader, we discuss in Appendices \ref{a.AVcriterion} and \ref{a.twisting}
 the straightforward modifications one must perform in the original 
argument of A. Avila and M. Viana in order to get Theorem \ref{thm:AVcrit}.
\par

In order to apply  Theorem \ref{thm:AVcrit} to the KZ-cocycle over the $\SL_2(\Rset)$-orbit of a square-tiled 
surface, we explain in Section~\ref{sec:cfalgo} how to view the \Teichmuller flow on such an orbit as 
the suspension of a full shift over a countable alphabet. This is derived from the classical relation between 
the continued fraction algorithm and the geodesic flow on the modular surface. The invariant Haar measure
on the orbit has bounded distortion. The KZ-cocycle corresponds in this setting to a locally constant integrable
 cocycle given by symplectic matrices with integer coefficients.
 \par
Section~\ref{sec:twistingcrit} provides the main step in the proof of Theorem \ref{thm:mainsimpcrit}. 
The pinching condition of Theorem \ref{thm:AVcrit} is replaced by the stronger hypothesis of Galois-pinching,
which only makes sense for matrices with integer coefficients. For Galois-pinching matrices, we are able to 
replace the twisting condition of Theorem  \ref{thm:AVcrit} by a weaker hypothesis, which is easily 
checked in the setting of Theorem \ref{thm:mainsimpcrit}. The proof of this fact (Theorem \ref {thm:maintwisting})
is quite involved.

\par
In Section~\ref{sec:redtoAV}, we  complete the proof of Theorem~\ref{thm:mainsimpcrit} and
Corollary~\ref{cor:mainsimpcrit}. We also present and prove a variant of Theorem~\ref{thm:mainsimpcrit} .

\par
In the last three sections of the paper, we explain how to apply Corollary \ref{cor:mainsimpcrit} to prove 
Theorem \ref {thm:introH4}.
\par 
In Section \ref{origamisH(4)}, we give some general background on origamis in $\mathcal H(4)$ which will 
be needed later. We extract from S. Leli\`evre's classification of saddle configurations in Appendix \ref{a.Samuel} a characterization of 
$2$-cylinder directions. We define  the HLK-invariant for origamis of odd or Prym type, and formulate 
precisely the conjecture of Delecroix and Leli\`evre. We recall what is known about origamis of Prym type.
Finally, we present some elementary Galois theory of reciprocal polynomials of degree $4$ which is instrumental in checking that an affine homeomorphism is Galois-pinching.
\par
The proof of Theorem \ref {thm:introH4} is given in Section  \ref{s.odd} for origamis of odd type, in Section 
\ref{s.hyperelliptic}  for origamis of hyperelliptic type. The method of the proof is the same in both cases.
We define in each case a model geometry, consisting of origamis with three horizontal cylinders and three 
vertical cylinders with a very simple intersection pattern. These two families of  origamis are parametrized 
by six integers, the heights of the horizontal and vertical cylinders. It is very easy to determine from 
the parameters the monodromy group (in the odd case) or the HLK-invariant (in the hyperelliptic case).
The model geometry allows to construct explicitly, for any values of the six parameters, an affine homeomorphism $A$ with $\textrm{tr}(A) >2$. 

\par
We consider a finite number of one-parameter subfamilies, where only one of the six parameters is allowed to vary along an arithmetic progression. We get in this way enough origamis to get one with prescribed invariant
(monodromy group or HLK-invariant) for any large number of squares. Also, in each of these families, a well-chosen rational direction has homological dimension $2$. Finally, we apply the elementary Galois theory of Subsection \ref{Galoistheory} to show that $A$ is Galois-pinching when the number of squares is large enough. Then Theorem \ref {thm:introH4} is a consequence from Corollary \ref{cor:mainsimpcrit}.

\section*{Acknowledgments} We are thankful to Artur Avila, Vincent Delecroix, Alex Eskin, Giovanni Forni, Pascal Hubert, Samuel Leli\`evre, Bertrand Patureau and Gabriela Schmith\"usen for several discussions related to the main results of this paper. 

Moreover, we are indebted to the two anonymous referees whose useful suggestions helped us to improve this paper at several places (including the statement and proof of Corollary \ref{cor:mainsimpcrit}).

Also, we are grateful to Coll\`ege de France (Paris-France), Hausdorff Research Institute for Mathematics (Bonn-Germany), and Mittag-Leffler Institute and Royal Institute of Technology (Stockholm-Sweden) for providing excellent research conditions during the preparation of this manuscript.

Last, but not least, the first and third authors were partially supported by the French ANR grant \emph{GeoDyM}~(ANR-11-BS01-0004) and by the Balzan Research Project of J. Palis and the second author was partially supported by \emph{ERC-StG}~{257137}.

\section{Background and Notations} \label{sec:background}

The basic references for the next 4 subsections are \cite{Zo6} and \cite{Ve89} (see also Section 1 of \cite{MY}), and for the last subsection is \cite{AVKZ}.

\subsection{Translation surfaces}  

Let  
\begin{itemize}
\item $M$ be a compact oriented topological surface of genus $g \geq 1$;
\item $\Sigma:= \{O_1,\ldots, O_{\sigma} \}$ be a non-empty finite subset of $M$; 
\item $\kappa = ( k_1, \ldots, k_{\sigma})$ be a non-increasing sequence of non-negative integers satisfying $\sum k_i = 2g-2$.
\end{itemize}

A structure of \emph{translation surface} on $(M ,\Sigma,\kappa)$ is a structure of Riemann surface on $M$, 
together with   a non identically zero holomorphic $1$-form $\omega$ which has at $O_i$ a zero of order $k_i$. 
Observe that, by the Riemann-Roch theorem, all zeroes of $\omega$ belong to $\Sigma$.
For example, taking $M = \Rset^2/ \Zset^2$, $\Sigma = \{0\}$, $k_1 = 0$, a structure of translation surface 
is defined  by the complex structure inherited from $\Cset \equiv \Rset^2$ and the holomorphic 
$1$-form  $\omega_0$ induced by $dz$. 


The reason for the nomenclature ``translation surface'' comes from the fact that local primitives of 
$\omega$ on $M-\Sigma$ provide an atlas whose changes of charts are given by translations of the plane. Below, the charts of this atlas will be called \emph{translation charts}.

\smallskip

Sometimes we will slightly abuse notation by denoting a translation surface by $(M, \omega)$ when the structure of Riemann surface is unambiguous. 

\smallskip

Denote by $\textrm{Diff}^+(M,\Sigma,\kappa)$ the group of orientation-preserving homeomorphisms $f$ of 
$M$ which preserve  $\Sigma$ and $\kappa$ (i.e if $f(O_i) = O_j$, then $k_i = k_j$). Denote by  
$\textrm{Diff}_0^+(M,\Sigma,\kappa)$ the identity component of $\textrm{Diff}^+(M,\Sigma,\kappa)$.
 The quotient 
 $$\Gamma(M,\Sigma,\kappa):= \textrm{Diff}^+(M,\Sigma,\kappa) / \textrm{Diff}_0^+(M,\Sigma,\kappa)$$  
 is the so-called \emph{mapping-class group}. 
 
 \smallskip

The \emph{\Teichmuller space} $\mathcal{T}(M,\Sigma,\kappa)$ is the quotient of the set of structure of 
translation surfaces on $(M,\Sigma,\kappa)$ by  the natural action of the group
 $\textrm{Diff}_0^+(M,\Sigma,\kappa)$. Similarly, the \emph{moduli space} 
$\mathcal{H}(M,\Sigma,\kappa)$  is the quotient of the same set by the natural action of the larger group 
$\textrm{Diff}^+(M,\Sigma,\kappa)$.
Thus the mapping class group acts on \Teichmuller space and the quotient space is the moduli space.

\smallskip


The group $\GL^+_2(\Rset)$ naturally acts on the set of structures of translation surfaces on 
$(M,\Sigma,\kappa)$  by post-composition with the translation charts. This action commutes with the action of
$\textrm{Diff}^+(M,\Sigma,\kappa)$. Therefore it induces an action of  $\GL^+_2(\Rset)$ on 
both $\mathcal{T}(M,\Sigma,\kappa)$ and $\mathcal{H}(M,\Sigma,\kappa)$. 

\smallskip

 The action of the $1$-parameter diagonal subgroup $g_t:=\textrm{diag}(e^t,e^{-t})$ 
 is the so-called \emph{\Teichmuller  flow}.


\smallskip

Let $\omega$ define a structure of translation surface on $(M,\Sigma,\kappa)$.
The {\it automorphism group} ${\rm Aut} (M,\omega)$, resp. the 
{\it affine  group} ${\rm Aff} (M,\omega)$, of the translation surface $(M,\omega)$ is the group of 
orientation-preserving homeomorphisms of $M$ which preserve $\Sigma$ and whose restrictions to 
$M-\Sigma$ read as  translations, resp. affine maps, in the translation charts of $(M,\omega)$. The 
embedding of ${\rm Aut} (M, \omega)$ into ${\rm Aff} (M, \omega)$ is completed into an 
exact sequence
$$1 \longrightarrow {\rm Aut} (M,\omega) \longrightarrow {\rm Aff} (M,\omega) 
\longrightarrow \SL (M,\omega) \longrightarrow 1 ,$$
where $\SL(M,\omega) \subset \SL(2,\Rset)$ is the {\it Veech group} of $(M,\omega)$: it is the stabilizer of the point in moduli space represented by $(M,\omega)$, for the  
action of $\GL^+_2( \Rset)$. The map from ${\rm Aff} (M,\omega)$ onto the Veech group is defined by  associating 
to each homeomorphism $\phi\in {\rm Aff} (M,\omega)$ its derivative (linear part) $D\phi\in\SL(M,\omega)$ in translation charts.

\bigskip

\begin{definition} Let $(M,\omega)$ be a translation surface. An {\it anti-automorphism} (or {\it central symmetry}) is an element of the affine group $ {\rm Aff} (M,\omega)$ whose derivative is $-\textrm{Id}$.
\end{definition}

The square of an anti-automorphism is an automorphism.

\subsection{\Teichmuller flow and Kontsevich-Zorich cocycle}

The total area $A(M,\omega)$ of a translation surface $(M,\omega)$ (namely $A(M,\omega):=(i/2)\int_M\omega\wedge\overline{\omega}$) is invariant under the commuting actions of 
$\SL_2(\Rset)$ and $\textrm{Diff}^+(M,\Sigma,\kappa)$. Thus it makes sense to consider the restriction of the action of $\SL_2(\Rset)$  to 
the \Teichmuller space $\mathcal{T}^{(1)}(M,\Sigma,\kappa)$ of unit area translation surfaces, and to  the moduli space $\mathcal{H}^{(1)}(M,\Sigma,\kappa)$ of unit area translation surfaces. 

The dynamical features of the \Teichmuller flow $g_t$ on the moduli spaces $\mathcal{H}^{(1)}(M,\Sigma,\kappa)$ have important consequences in the study of interval exchange transformations, translation flows and billiards. For instance, H. Masur \cite{Ma82} and W. Veech \cite{V82}, \cite{V86} constructed, on each connected component of  the normalized moduli space $\mathcal{H}^{(1)}(M,\Sigma,\kappa)$, a  $\SL_2(\Rset)$-invariant probability measure, nowadays called Masur-Veech measure.   They used the recurrence properties of the \Teichmuller flow with respect to this probability measure to confirm M. Keane's conjecture on the unique ergodicity of ``typical'' interval exchange transformations.  We will denote the Masur-Veech measures by $\mu_{MV}$ in what follows. 

After that, A. Zorich and M. Kontsevich (see \cite{Zo1}, \cite{Zo2}, \cite{Zo3} and \cite{K}) introduced the so-called  {\em Kontsevich-Zorich cocycle} (KZ cocycle for short) partly motivated by the study of deviations of ergodic averages of interval exchange transformations. 

Roughly speaking, the KZ cocycle $G_t^{KZ}$ is obtained from the quotient of the trivial cocycle

\begin{eqnarray*}
\widehat{G}_t^{KZ}:\mathcal{T}^{(1)}(M,\Sigma,\kappa)\times H_1(M,\mathbb{R})&\to& \mathcal{T}^{(1)}(M,\Sigma,\kappa)\times H_1(M,\mathbb{R}), \\
 \widehat{G}_t^{KZ}(\omega,[c])&:=& (g_t(\omega),[c]),
\end{eqnarray*}
by the (diagonal) action of the mapping-class group $\Gamma(M,\Sigma,\kappa)$.
In other words, the KZ cocycle $G_t^{KZ}$ acts on $H_1^{\Rset}:=(\mathcal{T}^{(1)}(M,\Sigma,\kappa)\times H_1(M,\mathbb{R}))/\Gamma(M,\Sigma,\kappa)$, a non-trivial bundle over $\mathcal{M}^{(1)}(M,\Sigma,\kappa)$ called the \emph{real Hodge bundle} in the literature.\footnote{Some translation surfaces $(M,\omega)$ have a (finite but) non-trivial group of automorphisms, so that this definition of the KZ cocycle may not lead to a well-defined linear dynamical cocycle  over certain regions of the moduli space. Over such regions, we only have a ``cocycle up to a finite group'', but this is not troublesome as far as Lyapunov exponents are concerned: in a nutshell, by taking finite covers, we can solve the ambiguity coming from automorphisms groups in the definition of KZ cocycle without altering Lyapunov exponents. See Subsection \ref{ss.sts} below and/or \cite{MYZ} for more discussion on this in the setting of square-tiled surfaces.} 

\bigskip

The action of $\Gamma(M,\Sigma,\kappa)$  on the $2g$-dimensional real vector space $H_1(M,\mathbb{R})$  preserves the natural symplectic intersection form on homology. Hence, $G_t^{KZ}$ is a symplectic cocycle. 
In particular, the Lyapunov exponents of the KZ cocycle with respect to an ergodic $g_t$-invariant probability measure $\mu$ on $\mathcal{M}^{(1)}(M,\Sigma,\kappa)$ satisfy $\lambda_{g+i}^{\mu}=-\lambda_{g-i+1}^{\mu}$ for each $i=1,\dots, g$.


Moreover, it is possible to show that the top Lyapunov exponent $\lambda_1^{\mu}$ is  equal to $1$, and is simple, i.e., $1=\lambda_1^{\mu}>\lambda_2^{\mu}$ (see e.g. \cite{Fo1}). In summary, the Lyapunov spectrum 
of the KZ cocycle with respect to any ergodic $g_t$-invariant probability measure $\mu$ has the form:
$$1=\lambda_1^{\mu}>\lambda_2^{\mu}\geq\dots\geq\lambda_g^{\mu}\geq-\lambda_g^{\mu}\geq\dots\geq-\lambda_2^{\mu}>-\lambda_1^{\mu}=-1$$

 The relevance of the KZ cocycle to the study of deviations of ergodic averages of interval exchange transformations resides on the fact that, roughly speaking, the Lyapunov exponents of the KZ cocycle ``control'' the deviation of ergodic averages of interval exchange transformations. See \cite{Zo1}, \cite{Zo2}, \cite{Zo3} and \cite{Fo1} for more details. In particular, this fact is one of the motivation for the study of Lyapunov exponents of KZ cocycle.

Based on numerical experiments, M. Kontsevich and A. Zorich \cite{K} conjectured that the Lyapunov spectrum of the KZ cocycle with respect to any Masur-Veech measure $\mu_{MV}$ is \emph{simple}, that is, the multiplicity of all exponents $\lambda_i^{\mu_{MV}}$ is equal  $1$, and  \emph{a fortiori} the exponents are all non-zero (i.e., $\lambda_g^{\mu_{MV}}>0$). This conjecture is nowadays known to be true after the celebrated works of G. Forni \cite{Fo1}, who proved that $\lambda_g^{\mu_{MV}}>0$, and A. Avila and M. Viana \cite{AVKZ}, who established the full conjecture. 

On the other hand, G. Forni and his coauthors (see e.g. \cite{Fo2} and \cite{FMZ}) constructed two instances of $\SL_2(\Rset)$-invariant ergodic probabilities $\mu_{EW}$ and 
$\mu_O$ whose Lyapunov spectra with respect to the KZ cocycle are totally degenerate in the sense that $\lambda_2^{\mu_{EW}}=\lambda_2^{\mu_O}=0$, that is, all Lyapunov exponents vanish except for the extreme (``tautological'') exponents $\lambda_1=1$, $\lambda_{2g}=-1$. 


\bigskip

 The main goal of this article is the proof of the simplicity  of the Lyapunov spectrum of the KZ cocycle with respect to certain ergodic $\SL_2(\Rset)$-invariant probability measures supported on the $\SL_2(\Rset)$-orbits of a special kind of  translation surfaces called \emph{square-tiled surfaces} or \emph{origamis}. For the reader's convenience, we recall some features of square-tiled surfaces in the next two subsections.

\subsection{Square-tiled surfaces}\label{ss.sts}

A {\it square-tiled surface} (or {\it origami}) is a translation surface $(M,\omega)$ such that
$$ \int_\gamma\omega\in \Zset \oplus i \Zset$$
for every relative homology class $\gamma \in H_1 (M,\Sigma, \Zset)$. 

\smallskip

Let $(M,\omega)$ be an origami and $O_j \in \Sigma$; the application
\begin{eqnarray*}
\pi: M &\rightarrow&\Cset /  \Zset \oplus i \Zset \equiv  \Tset^2 \\
A & \rightarrow &\int_{O_j}^{A} \omega \quad {\rm mod} \;  \Zset \oplus i \Zset
\end{eqnarray*}
is independent of $j$ and is a ramified covering with the following properties:
\begin{itemize}
\item $\pi$ is unramified over $\Tset^2 - \{0\}$;
\item $\omega= \pi^*(\omega_0(\Tset^2))$;
\item $\Sigma$ is contained in the fiber $\pi^{-1}(0)$.
\end{itemize}
Conversely, a ramified covering with these properties defines an origami.\footnote{For a square-tiled surface, it is more convenient to normalize the area of each square than the total area.} 

\smallskip

Let $(M,\omega)$ be an origami and let $\pi: M \rightarrow \Tset^2$ be the associated ramified covering.
A {\it square} of $(M,\omega)$ is a connected component of $\pi^{-1}((0,1)^2)$. The cardinality of the set
 $Sq (\omega)$ of squares is the degree of $\pi$. One defines two bijections $\sigma_h$, $\sigma_v$
from $Sq (\omega)$ to itself by sending  a square to the square immediately to the right, resp. to the top, of it.
The subgroup of the symmetric group over $Sq (\omega)$  generated by $\sigma_h$, $\sigma_v$ is 
called the {\it monodromy group} of $(M,\omega)$ and acts transitively on $Sq (\omega)$.

 \bigskip
 
 The action of $\SL_2(\Zset) \subset \GL_2^+(\Rset)$ on the set of translation surfaces preserves the subset of 
 square-tiled surfaces. In moduli space, the orbit of an origami for this action is finite. The action of 
 $\SL_2(\Zset)$ on origamis preserves the number of squares and the monodromy group (cf. \cite {zmia}).

\smallskip

\begin{definition}
An origami $(M,\omega)$ is  \emph{reduced} if  the group of relative periods
$\int_\gamma\omega$, $\gamma \in H_1 (M,\Sigma, \Zset)$  is equal to $ \Zset \oplus i \Zset$. 
Equivalently, writing $\pi$ for the associated ramified covering of $\Tset^2$, 
there does not exist a ramified covering  $\pi': M \rightarrow \Tset ^2$ unramified over  
$\Tset^2 - \{0\}$ and a covering $\pi'': \Tset ^2  \rightarrow  \Tset ^2$ of degree $>1$ such that $\pi= \pi'' \circ \pi'$.
\end{definition}

We will only consider from now on reduced origamis.
For later use, we also  give the definition of \emph{primitive} square-tiled surfaces. 

\begin{definition} A square-tiled surface $(M,\omega)$ is \emph{primitive} if its associated covering $\pi:M\to\mathbb{T}^2$ has no intermediate coverings other than the trivial ones, namely, $\pi$ and $id$.
\end{definition}

 A primitive square-tiled surface is  reduced (but the converse is not always true). The  automorphism group of a primitive origami is  trivial.

\smallskip

Let $(M,\omega)$ be a reduced square-tiled surface. The Veech group $\SL(M,\omega)$   
is a finite index subgroup of 
$\SL_2( \Zset)$, and, \emph{a fortiori}, it is a lattice in $\SL_2(\Rset)$. Thus $\SL_2(\Rset)/\SL(M,\omega)$
 supports an unique $\SL_2(\Rset)$-invariant probability measure. By a theorem of J. Smillie (cf. \cite{SW}), 
 the $\SL_2(\Rset)$-orbit of $(M,\omega)$ is closed in the corresponding moduli space.
  Hence there exists exactly one $\SL_2(\Rset)$-invariant probability measure supported by the 
  $\SL_2(\Rset)$-orbit of $(M,\omega)$ in moduli space.



\bigskip 


The affine group $\textrm{Aff}(M,\omega)$ embeds naturally in the mapping class group $\Gamma(M,\Sigma,\kappa)$ and its image is the stabilizer of the $\SL_2(\Rset)$-orbit of $(M,\omega)$ in \Teichmuller space.
Therefore, the restriction of the KZ cocycle $G_t^{KZ}$  to the 
$\SL_2(\Rset)$-orbit of $(M,\omega)$ in moduli space can be thought as the quotient of the trivial cocycle (over the orbit of $(M,\omega)$ in \Teichmuller space)
$$\widehat{G}_t^{KZ}:\SL_2(\Rset)\cdot(M,\omega)\times H_1(M,\Rset)\to \SL_2(\Rset)\cdot(M,\omega)\times H_1(M,\Rset)$$
by the action of the affine group $\textrm{Aff}(M,\omega)$. In other words, for recurrent times of the \Teichmuller flow, the restriction of the KZ cocycle to the orbit of $(M,\omega)$ acts on the fibers $H_1(M,\mathbb{R})$ of the real Hodge bundle through certain elements of $\textrm{Aff}(M,\omega)$. In particular, we see that: 
\begin{itemize}
\item when the automorphism group $\textrm{Aut}(M,\omega)$  is \emph{trivial}, the Veech group $\SL(M,\omega)$ 
is canonically isomorphic to the affine group $\textrm{Aff}(M,\omega)$. In this case, the  KZ cocycle is a linear dynamical cocycle in the usual sense.
\item on the other hand, when $\textrm{Aut}(M,\omega)$ is \emph{not} trivial, the action of the  KZ cocycle on the fibers of the real Hodge bundle is only defined \emph{up to} the action on homology of the finite group $\textrm{Aut}(M,\omega)$. In order to remove this ambiguity and to recover a standard linear dynamical cocycle, one must introduce  a \emph{finite cover} of moduli space. Such a finite cover is constructed, for instance, in \cite{MYZ} (by marking some horizontal separatrices). 
\end{itemize}

In the sequel, we will  avoid the technical issue pointed out in the last item by assuming from now on that
 \emph{ the origami under consideration has a trivial automorphism group}. We recall the following 
 well-known fact.
 
 
 
 \bigskip
 
 \begin{proposition}\label{no_automorphism}
A translation surface $(M,\omega)$ such that $\# \Sigma = 1$ has no nontrivial automorphism.
\end{proposition}
\begin{proof}
The quotient of $(M,\omega)$ by a nontrivial automorphism would be a cyclic Galois ramified covering over 
another translation surface $(N,\omega')$ with a single marked point $O'$. A small loop around $O'$ is a product of commutators in the fundamental group of $N-\{O'\}$, hence would lift to a loop in $M$. 
Thus the covering would be unramified and $\Sigma$ would have more than one point.
\end{proof}

\subsection {Decomposition of the homology}

Let $(M,\omega)$ be a reduced square-tiled surface (with trivial automorphism group) and let  $\pi:M\to\Tset^2$ be the associated  covering. We denote by $H_1^{(0)}(M,\Zset)$ the kernel of the linear map 
$$\pi_*: H_1(M,\Zset) \longrightarrow H_1(\Tset ^2,\Zset) \simeq \Zset ^2$$
induced by $\pi$.
We define similarly $H_1^{(0)}(M,\Rset)$ as the kernel of the map on $\Rset$-homology, 
consequently $H_1^{(0)}(M,\Rset) = H_1^{(0)}(M,\Zset) \otimes \Rset$. 

The subspace $H_1^{(0)}(M,\Rset)$ of $H_1(M,\Rset)$ is 
symplectic with respect to the symplectic form induced by the natural intersection form on homology. Denoting by $H_1^{st}(M,\Rset)$ its $2$-dimensional symplectic orthogonal, 
one has\footnote {One may define $H_1^{st}(M,\Zset)$ as the intersection 
$H_1^{st}(M,\Rset) \cap H_1(M,\Zset)$ but one should be aware that in general the direct sum 
$ H_1^{st}(M,\Zset) \oplus H_1^{(0)}(M,\Zset)$ is not equal to $H_1(M,\Zset)$, but only to a sublattice of this group.}
$$ H_1(M,\Rset) = H_1^{st}(M,\Rset) \oplus H_1^{(0)}(M,\Rset).$$

The action of the affine group ${\rm Aff} (M,\omega)$ on $H_1(M,\Rset)$ preserves this splitting. The action on the first factor 
$H_1^{st}(M,\Rset)$ (canonically identified with $\Rset ^2$ via $\pi_*$) is through the standard action of the Veech group $\SL(M,\omega)$ on $\Rset ^2$.
 The action on the second factor respects the integral lattice $H_1^{(0)}(M,\Zset)$.
 
The splitting above is extended as a constant splitting of the trivial bundle \newline $\SL_2(\Rset)\cdot(M,\omega)\times H_1(M,\Rset)$ over the orbit of $(M,\omega)$ in \Teichmuller space. 
 Going to the quotient by the action of the affine group, we get a well-defined splitting of the real Hodge bundle
 (over the orbit of $\SL_2(\Rset)$-orbit of $(M,\omega)$ in moduli space)
 which is invariant under the KZ cocycle. We denote by $G_t^{KZ,st}$, resp. $G_t^{KZ,(0)}$, the action of the KZ
 cocycle on the two subbundles.


Let $\mu$ be an ergodic $g_t$-invariant probability measure  supported on the $\SL_2(\Rset)$ orbit of $(M,\omega)$ in moduli space. The action of $G_t^{KZ,st}$ on $H^{st}_1$ is responsible for the ``tautological'' Lyapunov exponents $\lambda_1^{\mu}=1$ and 
$\lambda_{2g}^{\mu}=-\lambda_1^{\mu}=-1$. One must focus on the action $G_t^{KZ,(0)}$ on $H_1^{(0)}$ to understand the
nontrivial non-negative Lyapunov exponents $\lambda_2^{\mu}\geq\dots\geq\lambda_g^{\mu}$ . 

The integral lattice $H_1^{(0)}(M,\Zset)$ is preserved by the affine group, and the restriction of the 
intersection form to this lattice is a non-degenerate integer-valued antisymmetric form $\Omega$. 
Choosing a basis of $H_1^{(0)}(M,\Zset)$, we can think of  $G_t^{KZ,(0)}$ as a cocycle over the 
\Teichmuller flow  with values in the matrix group $\textrm{Sp}(\Omega,\mathbb{Z})$ of 
$\Omega$-symplectic $(2g-2) \times (2g-2)$-matrices with integer coefficients.

\begin{remark}\label{r.Sympl}
The integer-valued symplectic form $\Omega$ on the $(2g-2)$-dimensional integer lattice
 $H_1^{(0)}(M,\Zset)$ is in general not isomorphic to the standard symplectic form on $\Zset^{2g-2}$.
As was pointed to us by B. Patureau, for an integer-valued symplectic forms on $\Zset^{2d}$ there exist uniquely defined positive
 integers $\omega_1,\ldots,\omega_d$ with $\omega_i |\, \omega_{i+1}$ and a decomposition into $\Omega$-orthogonal $2$-dimensional planes
 $$ \Zset^{2d} = \oplus_{i=1}^d E_i $$
 such that each $E_i$ has a basis $(e_i,f_i)$ with $\Omega(e_i,f_i) = \omega_i$, see \cite[Section 5.1, Th 1]{Bourbaki-A9}. In fact, in our situation $\omega_i = 1$ for $i<g-2$ and $\omega_{g-1} | \textrm{deg}(\pi)$ (where $\textrm{deg}(\pi)$ is the degree of the covering $\pi: M\to\mathbb{T}^2$), 
see \cite[Corollary~12.1.5 and Proposition~11.4.3]{BiLa}.
 \end{remark}

In the sequel, we will study the simplicity of the Lyapunov spectrum of $G_t^{KZ,(0)}$ with respect to the (unique) $\SL_2(\Rset)$-invariant probability measure $\mu$ supported on the $\SL_2(\Rset)$-orbit of  $(M,\omega)$, for certain reduced origamis $(M,\omega)$ with $\textrm{Aut}(M,\omega)=\{\textrm{Id}\}$. During this task, we will use a (variant of a) simplicity criterion for cocycles over complete shifts (originally) due to A. Avila and M. Viana. The content of this criterion is quickly reviewed in the next subsection. 

\subsection{The Avila-Viana simplicity criterion and its variants}\label{ss.AVsetting}

In this subsection we briefly recall the setting of the Avila-Viana simplicity 
criterion \cite{AVKZ}, and we discuss some variants of it 
(of particular interest to our context). There are 
several versions of this criterion  in the literature, notably by A. Avila and M. Viana themselves \cite{AVPort}.

Let $\Lambda$ be a finite or countable alphabet. Define $\Sigma := \Lambda^{\mathbb{N}}$,  $\hat{\Sigma} := \Lambda^{\mathbb{Z}} = \Sigma_-\times\Sigma$. Denote by $f:\Sigma\to\Sigma$ and $\hat{f}:\hat{\Sigma}\to\hat{\Sigma}$ the natural (left) shift maps on $\Sigma$ and $\hat{\Sigma}$ respectively. Also, let $p^+: \hat{\Sigma}\to\Sigma$ and $p^-:\hat{\Sigma}\to\Sigma_-$ be the natural projections.

We denote by $\Omega = \bigcup\limits_{n\geq 0}\Lambda^n$ the set of words of the alphabet $\Lambda$. Given $\underline{\ell}\in\Omega$, let
$$\Sigma(\underline{\ell}):=\{x\in\Sigma: x \textrm{ starts by }\underline{\ell}\}$$
and
$$\Sigma_-(\underline{\ell}):=\{x\in\Sigma_-: x \textrm{ ends by }\underline{\ell}\}$$

Given a $f$-invariant probability measure $\mu$ on $\Sigma$, we denote by $\hat{\mu}$ the unique $\hat{f}$-invariant probability measure satisfying $p^+_*(\hat{\mu})=\mu$, and we define $\mu_-:=p^-_*(\hat{\mu})$.

Following \cite{AVKZ}, we will make the following bounded distortion 
assumption on $\mu$:

\begin{definition}[Bounded distortion]\label{def:boundeddist} We say that the $f$-invariant probability measure $\mu$ on $\Sigma$ has the \emph{bounded distortion} property if there exists a constant $C(\mu)>0$ such that
$$\frac{1}{C(\mu)}\mu(\Sigma(\underline{\ell}_1))\mu(\Sigma(\underline{\ell}_2)) \leq 
\mu(\Sigma(\underline{\ell}_1\underline{\ell}_2)) \leq C(\mu) \mu(\Sigma(\underline{\ell}_1)) \mu(\Sigma(\underline{\ell}_2))$$
for any $\underline{\ell}_1, \underline{\ell}_2\in\Omega$.
\end{definition}

\begin{remark}\label{r.bd-ergodic}It is not hard to check that the bounded distortion property implies that $\mu$ is $f$-ergodic.
\end{remark}

In this subsection, such $(f,\mu)$ and $(\hat{f},\hat{\mu})$ will be the base dynamics. We now discuss the class of cocycles we want to investigate over these base dynamics. Given a map $A$ from 
$\Sigma$ with values in a matrix group $\Gset$ acting on $\Kset^d$, the associated cocycle itself is given by
$$(f,A):\Sigma\times\mathbb{K}^d\to\Sigma\times\mathbb{K}^d,\quad(f,A)(x,v)=(f(x),A(x)\cdot v).$$
Here, the {\it basefield} $\Kset$ may be $\Rset$, $\Cset$ or the quaternion skew field $\textbf{H}$.

For $\underline{\ell}=(\ell_0,\dots,\ell_{n-1})\in\Omega$, we define
$$A^{\underline{\ell}}:=A_{\ell_{n-1}}\dots A_{\ell_0}$$
so that we have  $(f,A)^n(x,v)=(f^n(x),A^{\underline{\ell}}\cdot v)$ for  $x\in\Sigma(\underline{\ell})$ and $v\in\mathbb{K}^d$.

\bigskip

In the works \cite{AVPort} and \cite{AVKZ}, A. Avila and M. Viana treat the cases of  
$\Kset = \Rset$, $\Gset= \textrm{GL}(d,\mathbb{R})$ and $\Gset= \textrm{Sp}(d,\mathbb{R})$ respectively. 
 They  use their criterion in the symplectic case to prove the Kontsevich-Zorich conjecture. In this paper we take 
the opportunity to show (in Appendices \ref{a.AVcriterion} and \ref{a.twisting}) how the arguments of 
A. Avila and M. Viana can be adapted in a straightforward way to give a unified treatment of the following cases:
\begin{itemize}
\item  $\Kset = \Rset$, $\Gset= \textrm{GL}(d,\mathbb{R})$ and $\Gset= \textrm{Sp}(d,\mathbb{R})$;
\item  $\Kset = \Rset$ and $\Gset$ is an orthogonal group $O(p,q) = U_{\mathbb{R}}(p,q)$, $p+q=d$ ;
\item $\mathbb{K}=\mathbb{C}$ and $\Gset$ is a complex unitary group $U_{\mathbb{C}}(p,q)$, $p+q=d$ ;
\item$\mathbb{K}=\textbf{H}$ and  $\Gset$ is a quaternionic unitary group $U_{\textbf{H}}(p,q)$, $p+q=d$ .
\end{itemize}
Our main motivation to consider these cases come from the recent works \cite{MYZ}, \cite{FMZ2} and \cite{AMY} where several examples of cocycles with values in these matrix groups appear naturally as ``blocks'' of the Kontsevich-Zorich cocycle over the closure of $\SL_2(\Rset)$-orbits of certain ``symmetric'' translation surfaces.

\begin{remark}
In the unitary case $U(p,q)$,  we will always assume without loss of generality that $p\geq q$.
\end{remark}

\begin{definition}(Locally constant integrable cocycles)\label{d.integrablecocycle} A cocycle
 $A:\Sigma\to \Gset$ is said to be
 \begin{itemize}
\item \emph{ locally constant} if  $A(\underline{x})=A_{x_0}$ depends only on the initial letter $x_0$ of $\underline x$;
\item \emph{ integrable} if  $\int_{\Sigma}\log\|A^{\pm1}(\underline{x})\|\,d\mu(\underline{x}) <\infty$.
\end{itemize}
For a locally constant cocycle, the integrability condition can be rewritten as 
$$\sum_{\ell \in \Lambda} \mu(\Sigma(\ell))\log\|A^{\pm1}_{\ell}||< \infty .$$
\end{definition}



\bigskip

In the sequel, let $\mu$ be  a $f$-invariant measure  on $\Sigma$ with the bounded distortion property and let 
$A$ be a  locally constant integrable $\Gset$-valued cocycle. 
As $\mu$ is ergodic and the cocycle is integrable,  the Oseledets theorem gives the existence of Lyapunov exponents
$$\theta_1\geq \dots\geq \theta_d.$$

\begin{remark}\label{multi} The Lyapunov exponents are counted above with \emph{essential} multiplicities: in the complex case $\mathbb{K} = \mathbb{C}$,  $\Cset^d$ is a $2d$-dimensional vector space over $\Rset$ and there are
$2d$ Lyapunov exponents from the real point of view. But each appears twice; the essential multiplicity is half the real multiplicity. Similarly, in the quaternionic case, each exponent appears 4 times from the real point of view and the essential multiplicity is $\frac 14$ times the real multiplicity.
\end{remark}


The matrix group $\Gset$ determines  \emph{a priori constraints} for the Lyapunov exponents (see, e.g., \cite{MYZ} and \cite{FMZ2}):
\begin{itemize}
\item in the symplectic case  $\Gset = \textrm{Sp}(d,\mathbb{R})$, $d$ even, one has $\theta_i=-\theta_{d+1-i}$ for $1 \leq i \leq d$;
\item in the real, complex or quaternionic unitary cases $ \Gset = U_{\mathbb{K}}(p,q)$, $q\leq p$, $p+q=d$, one has 
 $\theta_i=-\theta_{d+1-i}$ for $1 \leq i \leq q$ and $\theta_i=0$ for $q<i\leq p$.
\end{itemize}
Also, in each of this case, the unstable Oseledets subspace associated to positive Lyapunov exponents is \emph{isotropic}. The same is true for the stable subspace associated to the negative exponents.

\begin{definition}\label{d.G-simplicity} The Lyapunov spectrum of the cocycle $A$ is \emph{simple} if
\begin{itemize}
\item $\theta_i>\theta_{i+1}$ for $1\leq i<d$ in the cases $\Gset = \textrm{GL}(d,\mathbb{R})$ and $\Gset = \textrm{Sp}(d,\mathbb{R})$, $d$ even;
\item $\theta_i>\theta_{i+1}$ for $1\leq i\leq q$ in the cases $\Gset = U_{\mathbb{K}}(p,q)$, 
$\mathbb{K}= \Rset, \mathbb{C},\textbf{H}$.
\end{itemize}
\end{definition}
In other words, we say that a cocycle is \emph{simple} when its Lyapunov spectrum is \emph{as simple as possible} given the constraints presented above.

\medskip

To formulate the main hypotheses on the cocycle,  we first need to introduce 
Grassmannian manifolds adapted to $\Gset$. We say that an integer $k$ is {\it admissible} if 
\begin{itemize}
\item $1 \leq k <d$ when $\Gset = \textrm{GL}(d,\mathbb{R})$ or  $\textrm{Sp}(d,\mathbb{R})$;
\item $1\leq k\leq q$ or $p\leq k<d=p+q$ in the unitary case $\Gset = U_{\mathbb{K}}(p,q)$, $p\geq q$.
\end{itemize}

Let $k$ be an  admissible integer. We denote by $G(k)$ the following Grassmannian manifold:
\begin{itemize}
\item when $\Gset = \textrm{GL}(d,\mathbb{R})$, the Grassmannian of $k$-planes of $\mathbb{R}^d$;
\item when $\Gset = \textrm{Sp}(d,\mathbb{R})$, $d$ even,  the Grassmannian of $k$-planes which are \emph{isotropic} (if $1\leq k\leq d/2$) or \emph{coisotropic} (if $d/2\leq k< d$);
\item when $\Gset = U_{\mathbb{K}}(p,q)$, the Grassmannian of $k$-planes over $\mathbb{K}$ which are \emph{isotropic} (if $1\leq k\leq q$) or  \emph{coisotropic} (if 
$p\leq k< d$).
\end{itemize}

At this point, we introduce the following two fundamental concepts:

\begin{definition}\label{d.pinching/strongtwist} The cocycle $A$ is:
\begin{itemize} 
\item \emph{pinching} if there exists $\underline{\ell}^*\in\Omega$ such that the spectrum of the matrix $A^{\underline{\ell}^*}$ is simple, i.e., the logarithms $\theta_i=\log\sigma_i$ of the singular values $\sigma_i=\sigma_i(A^{\underline{\ell}^*})$ of $A^{\underline{\ell}^*}$ satisfy the inequalities in Definition \ref{d.G-simplicity} (we then say that  $A^{\underline{\ell}^*}$ is a \emph{pinching matrix}).
\item \emph{twisting} (``strong form'') if for any $m\geq 1$, any  admissible integers $k_1,\dots,k_m$, any subspaces $F_i\in G(k_i)$, $F_i'\in G(d-k_i)$, $1\leq i\leq m$, there exists $\underline{\ell}\in\Omega$ such that  $A^{\underline{\ell}}(F_i)\cap F_i'=\{0\}$.
\end{itemize}
\end{definition}

We can know state the  Avila-Viana simplicity criterion (cf. \cite[Theorem 7.1]{AVKZ}), $\Gset$ being one of the groups mentioned above.

\begin{theorem}[A. Avila and M. Viana]\label{t.AVcriterion}  Let $\mu$ be a $f$-invariant probability measure 
on $\Sigma$ with the bounded distortion property.  Let $A$ be a locally constant integrable  
$\Gset$-valued cocycle.  Assume that $A$ is pinching and twisting. Then, the Lyapunov spectrum of $(f,A)$ with respect to $\mu$  is simple.
\end{theorem}

We will use a variant of this criterion with a relative version of twisting.

\begin{definition} Let $k$ be an admissible integer. Let $A \in \Gset$ be a pinching matrix. A matrix $B \in \Gset$ is $k$-\emph{twisting with respect to}  $A$ if one has
$$B(F)\cap F'=\{0\}$$
for every pair of $A$-\emph{invariant} subspaces $F\in G(k)$ and $F'\in G(d-k)$. 
\end{definition}

\begin{remark}\label{r.noether-3} The point of this definition is that, given a pinching matrix $A$ and an admissible integer $k$, there are only finitely many $A$-invariant subspaces $F \in G(k)$:
\begin{itemize}
\item  when $\Gset = \textrm{GL}(d,\mathbb{R})$ or  $\Gset = \textrm{Sp}(d,\mathbb{R})$, the eigenvalues of $A$ are real and simple, and $F$ is spanned by eigenvectors of $A$;
\item when $ \Gset = U_{\mathbb{K}}(p,q)$, $q\leq p$, $\Kset = \Rset$ or $ \Cset$, $A$ has $q$ unstable simple eigenvalues
 $\lambda_1, \ldots, \lambda_q \in \Kset$  with 
 $$|\lambda_1| > \ldots > |\lambda_q| >1,$$
 $q$ stable eigenvalues $\lambda'_m= \bar\lambda_m^{-1}$, and $p-q$ eigenvalues of modulus $1$. Denote by $v_1, \ldots , v_q,v'_1,\ldots,v'_q$ eigenvectors associated to $\lambda_1, \ldots, \lambda_q, \lambda'_1,\ldots, \lambda'_q$. An isotropic $A$-invariant subspace is spanned by some of these eigenvectors (not allowing both $v_m$ and $v'_m$). A coisotropic $A$-invariant subspace is the orthogonal complement of an isotropic $A$-invariant subspace.
 \item when $ \Gset = U_{\textbf{H}}(p,q)$, $q\leq p$,  consider $A$ as complex unitary of signature $(2p,2q)$. 
 Counted with multiplicity, its unstable eigenvalues are $\lambda_1, \bar \lambda_1,\ldots, \lambda_q, \bar \lambda_q$ with 
 $$|\lambda_1| > \ldots > |\lambda_q| >1.$$
 Its stable eigenvalues are $\lambda'_m= \lambda_m^{-1}, \bar \lambda'_m$, $1 \leq m \leq q$. Denote by $v_m$,$v'_m$ eigenvectors associated to $\lambda_i, \lambda'_i$. Then $v_m.j$, resp.\ $v'_m.j$ are eigenvectors associated to $\bar \lambda_m$, resp.\  $\bar \lambda'_m$. An isotropic $A$-invariant $\textbf H$-subspace is spanned by some of the eigenvectors $v_1, \ldots, v'_q$ (not allowing both $v_m$ and $v'_m$). A coisotropic $A$-invariant $\textbf H$-subspace is the orthogonal of an isotropic $A$-invariant $\textbf H$-subspace.

\end{itemize}
\end{remark}

In Appendix \ref{a.twisting} we will show the following result which relates the strong and relative versions of twisting:

\begin{proposition}\label{p.twist/strongtwist} A cocycle $A$ is pinching and twisting (in its strong form) if and only if there exists a word $\underline{\ell}^*\in\Omega$ such that $A^{\underline{\ell}^*}$ is a pinching matrix and, for each  admissible integer $k$, there exists a word $\underline{\ell}(k)\in\Omega$ such that $A^{\underline{\ell}(k)}$ is $k$-twisting with respect to $A^{\underline{\ell}^*}$.
\end{proposition}


By putting together Theorem \ref{t.AVcriterion} and Proposition \ref{p.twist/strongtwist}, we obtain the following 
variant of the Avila-Viana simplicity criterion:

\begin{theorem}\label{thm:AVcrit}
 Let $\mu$ be a $f$-invariant probability measure 
on $\Sigma$ with the bounded distortion property.  Let $A$ be a locally constant integrable  
$\Gset$-valued cocycle. Assume that there exists a word $\underline{\ell}^*\in\Omega$ with the following properties:
\begin{itemize}
\item   the matrix $A^{\underline{\ell}^*}$ is pinching;
\item for each  admissible integer $k$, there exists a word $\underline{\ell}(k)\in\Omega$ such that the matrix
$A^{\underline{\ell}(k)}$ is $k$-twisting with respect to  $A^{\underline{\ell}^*}$.
\end{itemize}
Then, the Lyapunov spectrum of $(f,A)$ with respect to $\mu$ is simple.
\end{theorem}

At this stage, this background section is complete and we pass to the discussion (in the next 3 sections) of Theorem \ref{thm:mainsimpcrit}.

\section {Moduli space and continued fraction algorithm for 
square-tiled surfaces} \label{sec:cfalgo}

The $\SL_2(\Rset)$-orbits of reduced square-tiled surfaces are finite covers of the moduli space $\SL_2(\Rset)/\SL_2(\Zset)$ of unit area lattices in the plane.  The \Teichmuller (geodesic) flow on $\SL_2(\Rset)/\SL_2(\Zset)$ is naturally coded by the continued fraction algorithm. Of course, this is classical and it is described (at least partly) in several places, see, e.g., \cite{Arnoux}, \cite{Da}, \cite{hubertlelievre}, \cite{McM}, \cite{S}, \cite{Zo7}, \cite{FMZ2}, and references therein. However, for our current purpose of coding the \Teichmuller flow and KZ cocycle over the $\SL_2(\Rset)$-orbit of a reduced square-tiled surface, we need a somewhat specific version of this coding that we were unable to locate in the literature.\footnote{When this paper was almost complete, A. Eskin and the first author \cite{EsMat} observed that the study of Lyapunov exponents of the KZ cocycle over \emph{closed} $SL(2,\mathbb{R})$-orbits can also be performed \emph{without} the aid of a coding of the \Teichmuller flow and KZ cocycle thanks to a profound theorem of H. Furstenberg on the Poisson boundary of homogenous spaces.} 
\par
In the first subsection, we introduce this specific version for the diagonal flow on the modular curve 
$\SL_2(\Rset)/\SL_2(\Zset)$.
Then, we define a similar coding for the \Teichmuller flow on the $\SL_2(\Rset)$-orbit of a reduced square-tiled surface. The KZ-cocycle has a natural discrete-time version adapted to this coding. Finally, we explain why we may assume that the base dynamics satisfies the hypotheses of Theorem \ref{thm:AVcrit}. The pinching and twisting conditions in this theorem are the subject of the next section.

\subsection{The torus case}\label{ss.torus-coding}

The continued fraction algorithm is generated by the {\it Gauss map}

\begin{eqnarray*}
 G: (0,1) \cap (\Rset - \Qset ) & \to& (0,1) \cap (\Rset - \Qset ) \\
 G(\alpha) &=& \{ \alpha^{-1} \},
 \end{eqnarray*}
where $\{x\}$ is the fractional part of $x \in \Rset$. The {\it Gauss measure} $\frac {dt}{1+t}$ is up to normalization the only $G$-invariant finite measure on $(0,1)$	 which is absolutely continuous with respect to Lebesgue measure. 
\par
Defining
$$ a(x) := \lfloor x^{-1} \rfloor \in \Zset_{>0} , \quad  \textrm {for} \; x \in (0,1) $$ 
and 
$$ \underline a(\alpha) := (a(G^n(\alpha)))_{n\geq 0}, \quad  \textrm {for} \; \alpha \in (0,1) \cap (\Rset - \Qset ) ,$$
we obtain the classical conjugacy between the Gauss map and the shift map 
$$f:  (\Zset_{>0})^{\Nset}  \to (\Zset_{>0})^{\Nset}$$
 on infinitely many symbols. For further reference, we note that
 \begin{proposition}\label{bddGauss}
 The Gauss measure, transferred to $(\Zset_{>0})^{\Nset}$ by the conjugacy, has bounded distortion.
 \end{proposition}
 \begin{proof}
 For $n>0$, $a_0,\ldots,a_{n-1} \in \Zset_{>0}$, the cylinder
  $\{\alpha,\; a(G^i(\alpha)) = a_i , \, \forall  \,0 \leq i <n\}$ is the Farey interval with endpoints 
  $\frac PQ, \frac{p+P}{q+Q}$, with
  
  $$  \left( \begin{array}{cc} p & P \\  q & Q \\  \end{array} \right) =  \left( \begin{array}{cc} 0 & 1 \\  1 & a_0 \\  \end{array} \right)  \ldots \left( \begin{array}{cc} 0 & 1 \\  1 & a_{n-1} \\  \end{array} \right).$$
  
  Observe that $Q\geq q,P \geq p$. As the density of the Gauss measure is bounded away from $0$ and 
  $+\infty$ on $(0,1)$, the Gauss measure of this cylinder is of order $Q^{-2}$.
  Let   $\{\alpha,\; a(G^j(\alpha)) = a_{n+j} , \, \forall  \,0 \leq j <m\}$ be another cylinder, and let  
  $\frac {P'}{Q'}, \frac{p'+P'}{q'+Q'}$  be the endpoints of the associated Farey interval. For the cylinder
   $$\{\alpha,\; a(G^i(\alpha)) = a_{i} , \, \forall  \,0 \leq i <m+n\},$$
 the associated Farey interval has endpoints $\frac {\bar P}{\bar Q}, \frac{\bar p+\bar P}{\bar q+\bar Q}$, with
 
 $$   \left( \begin{array}{cc} \bar p & \bar P \\  \bar q &\bar Q \\  \end{array} \right) =  
 \left( \begin{array}{cc} p & P \\  q & Q \\  \end{array} \right)   \left( \begin{array}{cc} p' & P' \\  q' & Q' \\  \end{array} \right).$$
 Its Gauss measure has order $\bar Q^{-2}$, and we have
 $$ Q Q' \leq \bar Q = qP' + Q Q' \leq 2 QQ'.$$
 The proof of the  bounded distortion property is complete.
 \end{proof}

We will use the map $\tilde G$ derived from the Gauss map as follows:

\begin{eqnarray*}
 \tilde G: \{t,b\} \times [(0,1) \cap (\Rset - \Qset )] & \to& \{t,b\} \times [ (0,1) \cap (\Rset - \Qset )] \\
 \tilde G(t,\alpha) = (b,G(\alpha)),& &G(b,\alpha) = (t,G(\alpha)).
 \end{eqnarray*}
(The letters $t$ and $b$ stand for top and bottom respectively.)
Putting the Gauss measure on each copy of $(0,1)$ gives a natural invariant measure that we still call the Gauss measure. For a symbolic model  for $\tilde G$, we consider the graph $\Gamma$ with two vertices
called $b,t$ and two countable families of arrows $(\gamma_{a,t})_{a\geq1}$ from $t$ to $b$ and  $(\gamma_{a,b})_{a\geq1}$ from $b$ to $t$. To a point $(c,\alpha)$, $c\in \{t,b\}$, we associate the path in 
$\Gamma$ starting from $c$ such that the first indices of the successive arrows are the $a(G^n(\alpha))$, $n\geq 0$. In this way, we get a conjugacy between $\tilde G$ and the shift map on the set of infinite paths in  
$\Gamma$. 

\medskip

To make the connection with the  flow on  the homogeneous space $\SL_2(\Rset) / \SL_2( \Zset)$ 
(viewed as the space of normalized  lattices in the plane) given by left multiplication  by the diagonal subgroup
$\diag (e^t, e^{- t})$, we use the following lemma.
Call a normalized lattice {\it irrational} if it intersects the vertical and horizontal axes only at the origin. The set of irrational lattices is invariant under the diagonal flow. 

\begin{lemma}\label{l.lattice} Let $L$ be an irrational normalized lattice in $\Rset ^2$. There exists a 
unique basis $v_1 = (\lambda_1, \tau_1), v_2 = (\lambda_2, \tau_2)$ of $L$ such that 
$$\text{either} \quad \lambda_2 \geq 1 > \lambda_1 >0 ,\; 0 < \tau_2 <-\tau_1 
\quad \text{or} \quad \lambda_1 \geq 1 > \lambda_2 >0,\; 0 <-\tau_1 < \tau_2.$$
\end{lemma}

\begin{proof}
We look for non zero vectors of $L$ in the squares $Q_+:= (0,1) \times (0,1)$ and 
$Q_-:= (0,1) \times (-1,0)$. First observe that $Q_+$ cannot contain two independent 
vectors of $L$, because the absolute value of their determinant would belong to $(0,1)$. 
Similarly for $Q_-$. On the other hand, the union $Q_+ \bigcup Q_-$ must contain a vector of
$L$ by Minkowski's theorem. In fact,  otherwise (using also that $L$ is irrational), there would 
exist a set of area $>1$ whose translates by $L$ are disjoint. Therefore either  $Q_+ \bigcup Q_-$ contains exactly one primitive vector of $L$ or both $Q_+$ and $Q_-$ contain exactly one primitive vector of $L$.
\par
In the first case, we can assume that $Q_-$ contains a primitive vector $v_1 = (\lambda_1, \tau_1)$. 
Let $v_2 = (\lambda_2, \tau_2)\in L$ such that 
$\lambda_1 \tau_2 - \lambda_2 \tau_1 =1$ and $\tau_2 >0 $ is minimum. Then $0 < \tau_2 < - \tau_1 <1$, 
hence $\lambda_2 >0$. As $Q_+$ does not contain any vector of $L$, we have $\lambda_2 \geq 1$ and the
 basis $(v_1,v_2)$ of $L$ has the required properties. On the other hand, if a basis $(v'_1, v'_2)$ has the required properties, either $\lambda'_2 \geq 1 > \lambda'_1 >0 ,\; 0 < \tau'_2 <-\tau'_1$ holds or $\lambda'_1 \geq 1 > \lambda'_2 >0,\; 0 <-\tau'_1 < \tau'_2$ holds. However, in the second case, the relation $\lambda'_1 \tau'_2 - \lambda'_2 \tau'_1 =1$ would imply that $v'_2 \in Q_+$, contrary to the assumption that $Q_+$does not contain any vector of $L$. Thus we have that $\lambda'_2 \geq 1 > \lambda'_1 >0 ,\; 0 < \tau'_2 <-\tau'_1$ holds. Then we must have $v'_1 \in Q_-$, hence $v'_1 =v_1$. From the inequality on $\tau'_2$,  it follows then that $v'_2 = v_2$, which concludes the proof of the lemma in this case. 
\par
In the second case we now assume that $Q_-$ contains a (unique) primitive vector $V_1 = (\Lambda_1,T_1)$ of $L$ and that $Q_+$ contains a (unique) primitive vector $V_2 = (\Lambda_2,T_2)$ of $L$. As $\Lambda_1 T_2 - \Lambda_2 T_1 \geq 1$, we have $\Lambda_1 + \Lambda_2 >1$. Observe also that we have $\Lambda_1 T_2 - \Lambda_2 T_1 < 2$, hence $\Lambda_1 T_2 - \Lambda_2 T_1 = 1$. As $L$ is irrational, we have that $T_1 + T_2 \ne 0$. If $T_1 + T_2 >0$, we set $v_1 =V_1, \; v_2 = nV_1 + V_2$, where $n\geq 1$ is the largest integer such that $nT_1 + T_2 >0$; if $T_1 + T_2 <0$, we set similarly $v_2 =V_2, \; v_1 = V_1 + nV_2$, where $n\geq 1$ is the largest integer such that $nT_2 + T_1 <0$. We obtain a basis of $L$ satisfying the required conditions.
Conversely, let $(v'_1, v'_2)$ be a basis  with the required properties. Assume for instance that $\lambda'_2 \geq 1 > \lambda'_1 >0 ,\; 0 < \tau'_2 <-\tau'_1$ holds. Then $v'_1 \in Q_-$, hence $v'_1 = V_1$. Then we must have $v'_2 = V_2 + m V_1$ for some integer $m \geq 1$, and the condition on $\tau'_2$ guarantees 
that $v'_1 =v_1$, $v'_2 = v_2$.  
\par
The proof of the lemma is complete.
\end{proof}

\begin{definition}
We say that the irrational lattice $L$ is of {\it top} type if the basis selected by the lemma 
satisfies $\lambda_2 \geq 1 > \lambda_1 >0 ,\; 0 < \tau_2 <-\tau_1$, of {\it bottom} type if it 
satisfies $\lambda_1 \geq 1 > \lambda_2 >0,\; 0 <-\tau_1 < \tau_2$.
\end{definition}

Let $L$ be an irrational normalized lattice, and let $(v_1,v_2)$ be the basis of $L$ selected by the lemma. Set $\alpha:=  \frac {\lambda_1}{\lambda_2}$ if $L$ is of top type, $\alpha:=  \frac {\lambda_2}{\lambda_1}$ if $L$ is of bottom type, so that $\alpha \in (0,1)$.

\smallskip

Assume for instance that $L$ is of top type. Let  $(L_t)_{t\geq 0}$ be the positive orbit of $L$ under the diagonal flow, and let $(v_1(t),v_2(t))$ be the basis of $L_t$ obtained by applying $\ diag (e^t, e^{-t})$ to $v_1$, $v_2$. For $0 \leq t < t^*:= - \log \lambda_1$, the basis $v_1(t),v_2(t)$ of $L_t$ satisfies the conditions of the lemma and $L_t$ is of top type. For $t=t^*$, the lattice $L_{t^*}$ is of bottom type and the basis selected by the lemma is $v^*_1= v_1 (t^*), v^*_2 = v_2(t^*) - a v_1(t^*)$, where $a \geq 1$ is the integral part of $\alpha^{-1}$. Thus, we have exchanged top and bottom with respect to the initial conditions and the new relevant ratio is
$$\alpha ^* = \frac {\lambda_2^*} {\lambda_1^*} = \frac {\lambda_2 - a \lambda_1} { \lambda_1} = 
\frac {1} {\alpha} -a$$
according to the continued fraction algorithm recipe. We have thus shown the 

\begin{proposition}\label {diag_code}
The diagonal flow on irrational normalized lattices is the suspension over the shift map on bi-infinite paths in
$\Gamma$, with roof function $\log \alpha^{-1}$.
\end{proposition}
Here, a bi-infinite path is one which extends indefinitely in both past and future. The number $\alpha$ associated to a bi-infinite path $\underline \gamma = \underline \gamma^{-} *\underline \gamma^{+} $ in position $c \in \{t,b\}$ at time $0$ is such that $\underline \gamma^{+} $ is associated to $(c,\alpha)$ in the coding of $\tilde G$ described above.

\subsection{A graph associated to a square-tiled surface}\label{ss.graph-origami}

Let now $M$ be a reduced square-tiled surface, let $\SL(M)$ be its Veech group, and let $\ell =\ell(M)$ be the index of 
$\SL(M)$ in $\SL_2(\Zset)$. The $\SL_2(\Rset)$-orbit of $M$ in the moduli space of translation 
surfaces is the homogeneous space $\SL_2(\Rset) / \SL(M)$, which is a covering of degree $\ell(M)$ of 
$\SL_2(\Rset) / \SL_2(\Zset)$. We  use this in order to encode the diagonal (\Teichmuller) flow 
on the $\SL_2(\Rset)$-orbit of $M$.
\par
Let $M_1 = M, \; M_2, \ldots , M_{\ell}$ be the square-tiled surfaces (up to isomorphism) constituting the orbit of $M$ under the action of $\SL_2(\Zset)$. Let $\Gamma(M)$ be the graph defined as follows
\begin{itemize}
\item the set of vertices $\textrm{Vert}(\Gamma(M))$ is the product $\{M_1, \ldots, M_{\ell}\} \times \{t,b\}$;
\item from every vertex $(M_i, c)$ ($c \in \{t,b\}$) and every integer $a \geq 1$, there is an arrow $\gamma_{a,i,c}$ starting from $(M_i, c)$, whose endpoint $(M_j,c')$ is given by 
$$M_j = \left( \begin{array}{cc} 1 & a \\  0 & 1 \\  \end{array} \right)\cdot M_i,\;c'=b,\quad \quad {\rm if} \; c=t,$$

$$M_j = \left( \begin{array}{cc} 1 & 0 \\  a & 1 \\  \end{array} \right)\cdot M_i, \; c'=t,\quad
 \quad {\rm if } \; c=b.$$
\item there are no other arrows.
\end{itemize}
\par
When $\ell =1$, i.e the Veech group of $M$ is equal to $\SL_2(\Zset)$, the graph $\Gamma(M)$ is the graph $\Gamma$ of the last subsection.
\par
Consider now the orbit under the \Teichmuller flow of a point $g_0.M$ ($g_0 \in \SL_2(\Rset)$) of the orbit of $M$. We first find the vertex of $\Gamma(M)$ associated to this initial point as follows. We apply Lemma \ref{l.lattice} to $L= g_0(\Zset ^2)$  (we assume that this lattice is irrational). Denoting by $(e_1,e_2)$ the canonical basis of $\Zset ^2$, we obtain from this lemma a matrix $g_0^* \in \SL_2(\Zset)$ such that the basis of $L$ with the required (top or bottom) property is $v_1 = g_0 ((g_0^*)^{-1} e_1), v_2 = g_0 ((g_0^*)^{-1} e_2)$. The vertex associated to $g_0.M$ is then $(g_0^*.M,c)$ where $c$ is the type of $L$.
\par
When we flow from this initial condition under the \Teichmuller flow, the lattice $L$ evolves under the diagonal flow and first changes type at some 
time $t^*$. As explained at the end of the last subsection,  the new selected basis is related to the old by {\bf right} multiplication by
$$\left( \begin{array}{cc} 1 & -a \\  0 & 1 \\  \end{array} \right) \quad (\text{if}\, c=t) 
\quad \text{or} \quad \left( \begin{array}{cc} 1 & 0 \\  -a & 1 \\  \end{array} \right) \quad (\text{if}\, c=b),$$
where $a$ is the integral part of $\alpha^{-1}$. 
This corresponds to the arrow of index $a$ in $\gamma(M)$ starting from $(g_0^*.M,c)$.
\par
This procedure allows to associate to every irrational orbit (i.e., one which neither starts nor ends in a cusp
 of $\SL_2(\Rset) / \SL(M)$) a bi-infinite path in $\Gamma(M)$. Conversely, every bi-infinite path in 
 $\Gamma(M)$ 
corresponds to a unique irrational orbit. Here, the time for the orbits and for the paths  runs from 
$- \infty$ to $+ \infty$. 
\par
If we consider instead infinite paths in $\Gamma(M)$ (with a starting point in $\textrm{Vert}(\Gamma(M))$ at time $0$), the shift map is conjugated to  the map $\tilde G_M$ defined by

\begin{eqnarray*}
 \tilde G_M: \textrm{Vert}(\Gamma(M))  \times [(0,1) \cap (\Rset - \Qset )] & \to&\textrm{Vert}(\Gamma(M)) \times [ (0,1) \cap (\Rset - \Qset )] \\
 \tilde G_M(M_i,c,\alpha)& =& (M_j,c',G(\alpha)),
 \end{eqnarray*}
 where $(M_j,c')$ is the endpoint of $\gamma_{a(\alpha),i,c}$.
 \par
 Summing up
 \begin{proposition}\label{Teich_code}
The diagonal flow on $\SL_2(\Rset)/ \SL(M)$ is the suspension over the shift map on bi-infinite paths in
$\Gamma(M)$, with roof function $\log \alpha^{-1}$. The number $\alpha$ associated to a bi-infinite path
$\underline \gamma = \underline \gamma^{-} *\underline \gamma^{+} $ is the second coordinate of the point associated to $ \underline \gamma^{+}$  in the coding of $\tilde G_M$.
\end{proposition}


\subsection{A discrete version of the KZ-cocycle}\label{ss.KZcoding}

Let $M$ be a reduced square-tiled surface and let $A \in \SL_2(\Zset)$. Recall that the preferred atlas of the square-tiled surface $M'=A.M$ is obtained from the preferred atlas of $M$ by postcomposition of the $\Rset^2$-valued charts by $A$. Therefore, if we consider  the identity map of $M$ as a map from $M$ to $M'$, it becomes an affine map with derivative $A$. However, as we identify in moduli space isomorphic square-tiled surfaces, this map is only well-defined in general up to precomposition by an automorphism of $M$ (or equivalently postcomposition by an automorphism of $M'$). From now on, we assume that the automorphism group of $M$ is trivial.

\smallskip

 To each arrow  $\gamma_{a,i,c}: (M_i,c) \rightarrow (M_j,c')$ of $\Gamma(M)$, we associate
the affine map $A_{a,i,c}$ from $M_i$ to $M_j$
whose linear part is 
$$\left( \begin{array}{cc} 1 & a \\  0 & 1 \\  \end{array} \right) \quad (\text{if}\, c=t) 
\quad \text{or} \quad \left( \begin{array}{cc} 1 & 0 \\  a & 1 \\  \end{array} \right) \quad (\text{if}\, c=b).$$ 

When $\underline \gamma$ is a path in $\Gamma(M)$ (starting at a vertex $(M_i,c)$, ending at a vertex $(M_j,c')$), which is the concatenation of arrows $\gamma_1 ,\,\ldots,\,\gamma_k$, we associate to  $\underline \gamma$ the affine map $A_{\underline \gamma}$ which is the composition $A_k \circ \ldots \circ A_1$ of the affine maps $A_1, \ldots , A_k$ associated to $\gamma_1 ,\,\ldots,\,\gamma_k$.
Then $A_{\underline \gamma}$ is an affine map from $M_i$ to $M_j$.

\medskip

We have given in Proposition \ref{Teich_code} a description of the restriction of the \Teichmuller flow to the $\SL_2(\Rset)$-orbit of $M$. The version of the KZ-cocycle which is adapted to this description is defined in the following way. The space is the vector bundle $\mathcal H_M$ over $\textrm{Vert}(\Gamma(M))  \times [(0,1) \cap (\Rset - \Qset )]$ whose fiber over $(M_i,c,\alpha)$ is $H_1(M_i,\Rset)$. The cocycle is the map 
$G_M^{KZ}$ fibered over $\tilde G_M$ such that 
$$ G_M^{KZ}(M_i,c,\alpha,v) = (M_j,c',G(\alpha),(A_{a(\alpha),i,c})_* (v)) ,$$
where $(A_{a(\alpha),i,c})_*$ is the homomorphism from $H_1(M_i,\Rset)$ to $H_1(M_j,\Rset)$ induced by the affine map $A_{a(\alpha),i,c}$. 
\par
As the diagonal flow on $\SL_2(\Rset) / \SL(M)$ is ergodic, the Lyapunov exponents of $ G_M^{KZ}$ are constant a.e. and proportional to the Lyapunov exponents of the continuous time version of the KZ-cocycle 
(w.r.t. Haar measure on $\SL_2(\Rset) / \SL(M)$). In order to determine whether the Lyapunov spectrum is simple, it is sufficient to consider  $ G_M^{KZ}$.

\subsection{ Return map and full shift }\label{ss.return}

We intend to prove below the simplicity of the Lyapunov spectrum by applying Avila-Viana's criterion
 (Theorem \ref{thm:AVcrit}). The base dynamics in the statement of this theorem is a full shift over
  an alphabet with at most countably many symbols. The base dynamics in the discrete version of the
   KZ-cocycle of the last subsection is the shift in the space of infinite paths in the graph 
   $\Gamma(M)$ of Subsection \ref {ss.graph-origami}. As $\Gamma(M)$ has more than one vertex, 
   this is not a full shift. 
   \par
   Nevertheless, it is easy to fall back into the setting of Theorem  \ref{thm:AVcrit}. Indeed, the map 
   $\tilde G_M$ is ergodic (for the Gauss measure). Choose any vertex $\mathcal V$ of $\Gamma(M)$. 
   Almost every path in $\Gamma(M)$ (with respect to the Gauss measure) goes through $\mathcal V$ infinitely many times.
   We consider the {\it return map}  $\tilde G_{M,\mathcal V} $ for $\tilde G_M$ to the subset 
   $\{\mathcal V\} \times  [(0,1) \cap (\Rset - \Qset )]$. The Gauss measure (in the fiber over $\mathcal V$) 
   is invariant by $\tilde G_{M,\mathcal V} $. The return map $\tilde G_{M,\mathcal V} $ is canonically 
   conjugated to the full shift over the alphabet whose letters are the loops in $\Gamma(M)$ 
   which go exactly once through $\mathcal V$.
   \par
   Observe that these letters correspond to words in the natural coding of the Gauss map $G$ considered in
   Subsection \ref{ss.torus-coding}. Therefore the bounded distortion property in the new setting follows from
    the same property in that setting (proposition  \ref{bddGauss}). 
    \par
    The cocycle over this full shift induced by the cocycle $G_M^{KZ}$ of the previous subsection is 
    clearly locally constant. It is 
    also {\it integrable}. Indeed, by submultiplicativity of the norm, it is sufficient to show that $G_M^{KZ}$ is 
    integrable. For this purpose we use that 
    the norm of $(A_{a,i,c})_*$ acting on homology has size at most $a$ for large $a$, 
    and the interval $(\frac 1{a+1},\frac 1a)$ of $\alpha$ such that $a(\alpha) =a$ has Gauss measure of 
    order $a^{-2}$. As the series $\sum_{a\geq 1} a^{-2} \log a $ is convergent, the integrability follows.

\subsection{Loops in $\Gamma(M)$ and affine pseudo-Anosov maps}\label{ss.loops-pA}

To determine which affine self-maps of $M$ correspond to {\bf loops} in $\Gamma(M)$, we recall an elementary property of $\SL(2,\Zset)$ related to the continued fraction algorithm.


\begin{definition}
A matrix $A= \left( \begin{smallmatrix} a & b \\  c & d \\  \end{smallmatrix} \right) \in \SL_2( \Zset)$ is 
{\it b-reduced} if 
$$a>\max(b,c)\geq \min(b,c)\geq d >0.$$ 
The matrix  
$A$ is  {\it t-reduced} if 
$$d>\max(b,c)\geq \min(b,c)\geq a >0.$$
Equivalently,  $A$ is t-reduced if its conjugate by $\left( 
\begin{smallmatrix} 0 & 1 \\  1 & 0 \\  \end{smallmatrix} \right)$ is b-reduced.
\end{definition}
\par

\begin{proposition} A matrix $A \in \SL_2(\Zset)$ is b-reduced if and only if 
 there exist $k\geq 1$ and integers $a_1, \ldots, a_{2k} \geq 1$ such that
$$A = \left( \begin{array}{cc} 1 & a_1 \\  0 & 1 \\  \end{array} \right)\left( \begin{array}{cc} 1 & 0 \\  a_2 & 1 \\  \end{array} \right)\ldots 
\left( \begin{array}{cc} 1 & a_{2k-1} \\  0 & 1 \\  \end{array} \right)\left( \begin{array}{cc} 1 & 0 \\  a_{2k} & 1 \\  \end{array} \right).$$
Moreover, such a decomposition of a b-reduced matrix is unique.
\end{proposition}
\par
 
\begin{proof}
 An elementary calculation shows that the product of two b-reduced matrices is b-reduced. 
Also, for $a,b \geq 1$, the product 
$$ \left( \begin{array}{cc} 1 & a \\  0 & 1 \\  \end{array} \right)\left( \begin{array}{cc} 1 & 0 \\  b & 1 \\  \end{array} \right) = 
\left( \begin{array}{cc} 1+ab & a \\  b & 1 \\  \end{array} \right)$$
is b-reduced. Therefore, products of the form appearing in the proposition are b-reduced.
Conversely, it is a classical elementary lemma that any $A \in \SL_2(\Zset)$ with nonnegative coefficients 
can be written in a unique way as 
$$A = \left( \begin{array}{cc} 1 & a_1 \\  0 & 1 \\  \end{array} \right)\left( \begin{array}{cc} 1 & 0 \\  a_2 & 1 \\  \end{array} \right)\ldots 
\left( \begin{array}{cc} 1 & a_{2k-1} \\  0 & 1 \\  \end{array} \right)\left( \begin{array}{cc} 1 & 0 \\  a_{2k} & 1 \\  \end{array} \right),$$
with $k\geq 1$, and integers $a_1 \geq 0$, $a_{2k} \geq 0$, $a_i\geq 1$ for $0<i<2k$. It is now 
trivial to check that $A$ is not b-reduced unless $a_1$ and $a_{2k}$ are both $\geq 1$.
\end{proof}
\par

\begin{corollary}\label{characterization_b-reduced}
The affine self-maps of $M$ associated to the loops of  $\Gamma(M)$ based at $(M,b)$ (resp. $(M,t)$) are exactly those which have a b-reduced (resp. t-reduced) linear part.
\end{corollary}

\begin{remark}\label{conjSL2}
It is a standard fact from Gauss theory of reduction of quadratic forms \cite{Gauss} that any matrix in $\SL_2(\Zset)$ with trace $>2$ is conjugated in $\SL_2(\Zset)$ to a b-reduced matrix. We omit the details of the proof of this fact. 
\end{remark}

\section{Galois-pinching} \label{sec:twistingcrit}

\subsection{Galois-pinching matrices}\label{Galoispinching}

Let $\Omega$ be an integer-valued symplectic form on $\Zset^{2d}$. We will denote by $\Sp(\Omega, \Zset)$ the group of matrices in $\SL(2d,\Zset)$ which preserve $\Omega$. 

\medskip

Let $A \in \Sp(\Omega,\Zset)$. The characteristic polynomial of $A$ is a monic reciprocal polynomial $P$ of 
degree $2d$ with integer coefficients.  Let $\wt R = \{\lambda_i, \lambda_i^{-1}: 1\leq i\leq d\}$ be the set of roots of $P$. For $\lambda \in \wt R$, 
define $p(\lambda) := \lambda + \lambda ^{-1}$ and let $R:= p(\wt R)$.

\begin{definition}
The matrix $A  \in \Sp(\Omega,\Zset)$ is {\it Galois-pinching} if its characteristic polynomial $P$ satisfies the following conditions
\begin{itemize}
\item $P$ is irreducible over $\Qset$;
\item all roots of $P$ are real, i.e $\wt R \subset \Rset$;
\item  the Galois group $\Gal$ of $P$ is the largest possible, that is, $\Gal$ acts on $R$ 
by the full permutation group of $R$, and the subgroup fixing each element of $R$ acts by 
independent transpositions of each of the $d$ pairs $\{ \lambda_i, \lambda_i^{-1} \}$; in other words, 
$$\Gal\simeq S_d\rtimes(\mathbb{Z}/2\mathbb{Z})^d.$$
\end{itemize}
\end{definition}

\smallskip

For each 
$\la \in \wt R$, we denote by $v_{\la} \in \Rset ^{2d}$  an eigenvector of $A$ corresponding 
to this eigenvalue with coordinates in the field $\Qset (\la)$. We assume that the choices 
are coherent in the sense that $g(v_{\la}) = v_{g.\la}$ for $g \in \Gal$.

\begin{proposition}\label{Gpinchingandpinching}
A Galois-pinching matrix is pinching.
\end{proposition}
\begin{proof}
 Indeed, by the first two conditions, all eigenvalues are simple and real.
The only possibility preventing $A$ to be pinching would be to have both $\lambda$ and $-\lambda$ as eigenvalues, but this is not compatible with the third condition: an element of the Galois group fixing $\lambda$ will also fix $-\lambda$.
\end{proof}

 The following result will be used in the proof of Theorem \ref{thm:mainsimpcrit}.
 
 \begin{proposition}\label{parabolictwist}
  Let $A, B$ be two elements of $\Sp(\Omega,\Zset)$. Assume that $A$ is Galois-pinching,  
and that $B$ is unipotent and distinct from the identity. If $A,B$  share a common proper invariant subspace, then 
  $(B - \textrm{id}) (\Rset^{2d})$ is a lagrangian subspace of $\Rset^{2d}$. 
 \end{proposition}
 
 \begin{proof}
A subspace of  $\Rset^{2d}$ which is invariant under $A$ is spanned by eigenvectors of $A$. Let 
$R^{\ddag} \subset \wt R$ be a  non-empty subset with minimal cardinality such that the subspace 
$E(R^{\ddag})$ spanned by the vectors $v_{\lambda}$, $\lambda \in R^{\ddag}$, is also invariant under $B$.
As $B$ has integer coefficients, the subset $\sigma(R^{\ddag})$ has the same property, for any element 
$\sigma$ of the Galois group $\Gal$ of $P$. As the cardinal of $R^{\ddag}$ was chosen to be minimal, we must have either $\sigma(R^{\ddag}) = R^{\ddag}$ or $\sigma(R^{\ddag}) \cap  R^{\ddag} = \emptyset$.
\par
The only proper subsets $ R^{\ddag}$ with this property are the $1$-element subsets and the $2$-elements subsets of the form $\{ \lambda, \lambda^{-1}\}$. 
\par
The first case cannot occur: if one had $B(v_{\lambda}) = c v_{\lambda}$ for some $\lambda \in \wt R$, $c\in \Rset$, then $c$ should be equal to $1$ as $B$ is unipotent, and $B$ should fix all eigenvectors of $A$ (applying the action of $\Gal$) and thus be equal to the identity.
\par
Therefore, $B$ preserves some $2$-dimensional subspace $E(\lambda,\lambda^{-1})$. 
Applying $\Gal$, we see that $B$ preserves each  subspace of this form, and the restrictions of $B$ to these 
subspaces are Galois-conjugated. As $B$ is unipotent distinct from the identity, the image of each such 
$2$-dimensional subspace by 
$(B- \textrm{id})$ has dimension $1$. These $2$-dimensional subspaces are $\Omega$-orthogonal. Therefore 
$(B - \textrm{id}) (\Rset^{2d})$ is $\Omega$-lagrangian.
 \end{proof}
 \begin{remark}
 The proof is still valid if $B \in \SL(2d,\Zset)$ instead of $\Sp(\Omega,\Zset)$.
 \end{remark}
 
\begin{remark} Concerning the Galois group $\Gal$ of $A$, the proof used only that $\Gal\supset G\rtimes(\mathbb{Z}/2\mathbb{Z})^d$ where $G\subset S_d$ is a transitive subgroup (i.e., given $i, j\in\{1,\dots, d\}$, there exists $g\in G$ with $g(i)=j$) without non-trivial blocks (i.e., there is no subset $\Delta\subset\{1,\dots, d\}$ of cardinality $1<\#\Delta<d$ such that either $g(\Delta)\cap\Delta=\emptyset$ or $g(\Delta)=\Delta$ for all $g\in G$). In particular, the same proof also works with $\Gal = A_d\rtimes(\mathbb{Z}/2\mathbb{Z})^d$ (for example). 
 \end{remark}

\subsection{A twisting criterion for Galois-pinching matrices}

 The most important ingredient towards Theorem~\ref{thm:mainsimpcrit} is the following theorem:

\begin{theorem} \label{thm:maintwisting}
 Let $A, B$ be two elements of $\Sp(\Omega,\Zset)$ . Assume that $A$ is Galois-pinching, and that 
 $A$ and $B^2$ don't share a common proper invariant subspace.
Then, there exist $m \geq 1$, and, for any $\ell^*$, integers $\ell_1,\ldots, \ell_{m-1}$ larger than $\ell^*$ such that the product 
$$   B A^{\ell_1} \cdots B A^{\ell_{m-1}} B$$
is $k$-twisting with respect to $A$ for all $1\leq k\leq d$.

\end{theorem}

\begin{remark} This theorem becomes false if we replace $B^2$ by $B$ in the assumption ``$A$ and $B^2$ don't share a common proper invariant subspace''. Indeed, the matrix $A=\left(\begin{smallmatrix} 
2 & 1 \\  1 & 1\end{smallmatrix}\right)\in \SL(2,\mathbb{Z})$ is Galois-pinching, the matrix $B=\left(\begin{smallmatrix} 0 & -1 \\  1 & 0\end{smallmatrix}
\right)\in \SL(2,\mathbb{Z})$ has no invariant subspaces, but $B A^{\ell_1} \cdots B A^{\ell_{m-1}} B$ is \emph{never} $1$-twisting with respect to $A$ (because $B$ permutes the eigenspaces of $A$). Of course, the matrices $A$ and $B$ don't fit the assumptions of the previous theorem because $A$ and $B^2=-\textrm{Id}$ share two common proper invariant subspaces (namely, the stable and unstable eigenspaces of $A$).
\end{remark}
\par
The proof of this result occupies the rest of this section.
\par
\smallskip
We keep the notations of Subsection \ref{Galoispinching}. As $A$ is assumed to be Galois-pinching, the set 
$\wt R$ has $2d$ elements and is contained in $\Rset$; the set $R:=p(\wt R)$ (where $p(\lambda):=\lambda+\lambda^{-1}$) has $d$ elements and is contained in 
$\Rset - [-2,2]$. 

\par

For $1 \leq k \leq d$, denote by $\wt R_k$ the set of all subsets of $\wt R$ with $k$ elements,
 by $ R_k$ the set of all subsets of $ R$ with $k$ elements, and by $\wh R_k$ the set of all
 subsets of  $\wt R$ with $k$ elements on which the restriction of $p$ is injective (so that 
their images under $p$ belong to $R_k$). We identify $\wt R_1 = \wh R_1$ with $\wt R$.
\par
For $\ula =\{ \la_1 < \cdots < \la_k \} \in \wt R_k$, let $v_{\ula} := v_{\la_1} \wedge \cdots
 \wedge v_{\la_k}\in\wedge^k \Rset^{2d}$. It is an eigenvector of $\wedge^k A$ with eigenvalue $N(\ula):= \prod_i \lambda_i$. The $v_{\ula}$, for 
 $\ula\in \wt R_k$, form a basis of $\wedge^k \Rset^{2d}$.

\subsection{Transversality condition} Using these notations, we can now translate the condition that a matrix $C$ is $k$-twisting with respect to $A$ (a condition concerning only \emph{isotropic} $A$-invariant $k$-dimensional and \emph{coisotropic} $A$-invariant $(2d-k)$-dimensional subspaces) in terms of \emph{certain} matrix entries of its $k$th exterior power 
$\wedge^k C$.

\begin{lemma}\label{l.twisting-complete} Let $C^{(k)}_{\ula \, \ula'}$, $\ula, \ula' \in \wt R_k$, be the coefficients of the matrix of $\wedge^k C$ in the
 basis $(v_{\ula})_{\ula \in \wt R_k}$. Then, $C$ is $k$-twisting with respect to $A$
if and only if all coefficients $C^{(k)}_{\ula \, \ula'}$ with
 $\ula, \ula' \in \wh R_k$ are non-zero.
\end{lemma}
\par
\begin{proof}
Let $E$, $F$ be $A$-invariant subspaces of respective dimensions $k$ and $2d-k$. Let $\ula_E \in \wt R_k$ be the subset such that $E$ is generated by the $v_\la$ with $\la \in \ula_E$. Observe that $E$ is isotropic if and only if $\ula_E $ belongs to $\wh R_k$. Similarly, let $\ula'_F \in \wt R_k$ be the subset such that $F$ is generated by the $v_\la$ with $\la \notin \ula'_F$. It belongs to $\wh R_k$ if and only if $F$ is coisotropic. Write $\ula_F$ for the complement of $\ula'_F$ in $\wt R$. Now $C(E)$ is transverse to $F$ if and only if the exterior product $\wedge^k C(v_{\ula_E}) \wedge v_{\ula_F}$ is nonzero. This happens precisely if and only if the coefficient $C^{(k)}_{\ula_E \, \ula'_F}$ is nonzero.
\end{proof} 

\begin{remark} The proof is still valid if we assume only that $A$ has real and simple spectrum (instead of $A$ being Galois-pinching). 
\end{remark}

\subsection{Mixing graphs}
Recall that an oriented graph $\Gamma$ (with  a finite set of vertices) is {\it strongly connected} if for every vertices $x,y$ of $\Gamma$, there is an {\bf oriented} path from $x$ to $y$. It is {\it  mixing} if there exists an integer $m$ such that, for every vertices $x,y$ of $\Gamma$, there is an  oriented path of length $m$ from $x$ to $y$. If it is the case, any large\footnote{In fact, this property occurs for some $m\leq (v-1)^2+1$ where $v$ is the number of vertices of $\Gamma$: cf. Exercise 8 in Section 8.7 of Allouche-Shallit's book \cite{AllSha}.} enough integer $m$ has this property.
\par
Let $C \in \Sp(\Omega,\mathbb{Z})$. For $1 \leq k \leq d$, we define an oriented graph 
$\Gamma_k = \Gamma_k(C)$ as follows:
the vertices of $\Gamma_k$ are the elements of $\wh R_k$; for $\ula_0, \ula_1 \in \wh R_k$, 
there is an arrow from $\ula_0$ to $\ula_1$ if and only if the coefficient $C^{(k)}_{\ula_0 \, \ula_1}$ 
of the matrix of $\wedge^k C$, written in the basis $(v_{\ula})_{\ula \in \wt R_k}$, is nonzero. For later use, we observe that $\Gamma_k(C)$ is \emph{invariant} 
under the natural action of the Galois group $\Gal$.
\par
By Lemma \ref{l.twisting-complete}, $C$ is $k$-twisting with respect to $A$ if and only if the graph $\Gamma_k(C)$ is complete (i.e., it has all possible edges between all pairs of vertices). In general, it is not easy to verify that 
$\Gamma_k(C)$ is a complete graph, but this is not an obstacle because, as we're going to see now, for our purposes it suffices to check the mixing property for $\Gamma_k(C)$.
\par
More precisely, let $1 \leq k \leq d$ and {\bf assume that $\Gamma_k(C)$ is  mixing}. Let $m$ be a positive 
integer such that, for every vertices $\ula_0,\ula_1$ of $\Gamma_k(C)$, there is an 
oriented path of length $m$ from $\ula_0$ to $\ula_1$. For nonnegative integers 
$\ell_1, \cdots , \ell_{m-1}$, consider
$$D = C A^{\ell_1} \cdots C A^{\ell_{m-1}} C.$$ 
\par
\begin{proposition} \label{prop:hyperplanes}
There are finitely many hyperplanes $V_1, \cdots ,V_t$ in $\Rset ^{m-1}$ such that, if $\underline \ell :=(\ell_1, \cdots , \ell_{m-1})\in \Zset ^{m-1}$ goes to infinity along a rational line not lying in any of the $V_s$, then, for $||\underline \ell ||$ large enough, the matrix $D$ is $k$-twisting with respect to $A$. In other words, if $\underline{\ell} = n\cdot(\overline{\ell}_1,\dots, \overline{\ell}_{m-1})\in\mathbb{Z}^{m-1}$ with $(\overline{\ell}_1,\dots, \overline{\ell}_{m-1})\notin V_1\cup\dots\cup V_t$ and $n\in\mathbb{N}$, then, for $n$ large enough, $D=C A^{n\overline{\ell}_1} \cdots C A^{n\overline{\ell}_{m-1}} C$ is $k$-twisting with respect to $A$.
\end{proposition}
\begin{proof}
We write $\wedge^k C$, $\wedge^k D$ in the basis  $(v_{\ula})_{\ula \in \wt R_k}$. 
Define a graph $\Gamma'_k$ with set of vertices $\wt R_k$ and an arrow from $\ula_0$ to $\ula_1$ 
if and only if the coefficient $C^{(k)}_{\ula_0 \, \ula_1}$ is nonzero. We have, for $\ula_0, \ula_m \in \wt R_k$
$$ D^{(k)}_{\ula_0 \, \ula_m}= \sum\limits_{\ula_1,\dots,\ula_{m-1}} C^{(k)}_{\ula_0 \, \ula_1} N(\ula_1)^{\ell_1} C^{(k)}_{\ula_1 \, \ula_2} \cdots N(\ula_{m-1})^{\ell_{m-1}} C^{(k)}_{\ula_{m-1} \, \ula_m}.$$
\par
In this sum, the nonzero terms correspond exactly to the paths $\gamma$ of length $m$ in $\Gamma'_k$ 
from $\ula_0$ to $\ula_m$. Writing $n(\ula) = \log |N(\ula)|$, the absolute value of such a term is 
a nonzero constant (independent of $\underline \ell$) times $\exp (\sum_{i=1}^{m-1} \ell_i n(\ula_i))$. 
Write $L_{\gamma}$ for the linear form $(\sum_{i=1}^{m-1} \ell_i n(\ula_i))$ on $\Rset ^{m-1}$. The 
important fact about these linear forms, which follows from our hypothesis on the Galois group $\Gal$, is the following:
\par
\begin{lemma} Let $\ula_0, \ula_m$ belong to $\wh R_k \subset \wt R_k$; let $\gamma$ be a path of length 
$m$ in $\Gamma_k$ from $\ula_0$ to $ \ula_m$ (such a path exists by our choice of $m$) and let  $\gamma'$ 
be a path of length $m$ in $\Gamma'_k$ from $\ula_0$ to $ \ula_m$ distinct from $\gamma$. Then the 
linear forms $L_{\gamma}$ and $L_{\gamma'}$ are distinct.
\end{lemma}
\par
\begin{proof}
The assertion of the lemma is a consequence of the following stronger assertion. Let $\ula$ be an 
element of $\wh R_k$, and let $\ula'$ be an element of $\wt R_k$ distinct from $\ula$. Then the absolute 
values of $N(\ula)$ and $N(\ula')$
are distinct. Indeed, assume by contradiction that we have a non trivial relation
$$ \la_1 \cdots \la_k = \pm \la'_1 \cdots \la'_k,$$
where $\ula = \{ \la_1, \cdots, \la_k \}$, $\ula' = \{ \la'_1, \cdots, \la'_k \}$.
Choose $\la_i \in \ula$ not belonging to $\ula'$; let $g \in \Gal$ be the element which exchanges $\la_i$ and $\lai_i$ and fixes all other roots. When we apply $g$ to both sides of the above relation, the left-hand side is multiplied by $\la_i^{-2}$ (because $\lai_i$ does not belong to $\ula$ as $\ula \in \wh R_k$); the right-hand side is multiplied by $\la_i^2$ if $\lai_i \in \ula'$, by $1$ otherwise. In any case, we obtain the required contradiction.  
\end{proof}
\begin{remark} The proof is still valid if we assume that $A$ has real and simple spectrum, and $\Gal$ contains the kernel $\{id\}\rtimes(\mathbb{Z}/2\mathbb{Z})^d$ of the natural map $S_d\rtimes (\mathbb{Z}/2\mathbb{Z})^d\to S_d$ (instead of $A$ is Galois-pinching).
\end{remark}
\par
The hyperplanes $V_1, \ldots, V_t$ of the proposition are defined as follows: for every path $\gamma $ 
of length $m$ in $\Gamma_k$, every path $\gamma' \ne \gamma$ of length $m$
in $\Gamma'_k$ with the same endpoints as $\gamma$, we exclude (for the direction of $\underline \ell$)
 the hyperplane $V(\gamma,\gamma')$ where the linear forms $L_{\gamma}$ and $L_{\gamma'}$ are equal. 
Then, along a rational line in $\Qset ^{m-1}$ not lying in any of these hyperplanes, none of the expressions above for the coefficients $D^{(k)}_{\ula_0 \, \ula_m}$, as finite linear combinations of exponentials of various rates, is identically zero; this completes the proof of the proposition. 
\end{proof}

\begin{remark}\label{r.prop-hyperplanes-1-twisting} The proof of Proposition \ref{prop:hyperplanes} also works for $k=1$ and $A\in \GL(n,\mathbb{Z})$ with real and simple spectrum. We will use this fact later in Subsection \ref{ss.main-twisiting-d-2}. 
\end{remark}

\subsection {The first step: transversality for $k=1$}\label{ss.transv1}

Starting from the matrix $B$ in the main theorem of this section (Theorem \ref{thm:maintwisting}), we consider the graph $\Gamma_1 = \Gamma_1(B)$. 
\par
\begin{claim} The graph $\Gamma_1$ is mixing. 
\end{claim}
\par
\begin{proof} First note that $\Gamma_1$ has at least one arrow as $B$ is invertible.
Next, for $d=1$, since $\Gamma_1$ is $\Gal$-invariant, $\Gamma_1$ is not mixing if and only if there are only two arrows in 
$\Gamma_1$ (either two loops or two arrows in both directions between $\la$ and $\lai$). 
In both cases $\Rset v_{\la} $ is invariant under $B^2$, a contradiction.
\par
We now consider the case where $d>1$ and the only arrows of $\Gamma_1$ join a vertex 
$\lambda$ to $\la$ or $\lai$. Then each $2$-dimensional subspace generated by $v_{\la}, v_{\lai}$
is $B$-invariant, a contradiction.
\par
Finally, assume that $d>1$ and $\Gamma_1$ has one arrow joining a vertex $\la$ to a vertex 
$\la' \ne \lambda^{\pm 1}$.  Using the action of the Galois group $\Gal$, every such arrow must be in $\Gamma_1$. This implies that $\Gamma_1$ is mixing when $d>2$. When $d=2$, the only case where $\Gamma_1$ is not mixing is when there is no other arrow in $\Gamma_1$; but then each $2$-dimensional subspace generated by $v_{\la}, v_{\lai}$ is $B^2$-invariant, a contradiction.
\end{proof}
\par 
Applying Proposition \ref{prop:hyperplanes} above with $k=1$, we find $m \geq 1$ and
$$ C = B A^{\ell_1} \cdots B A^{\ell_{m-1}} B$$
a $1$-twisting matrix with respect to $A$. Observe that $\Gamma_1(C)$ is then the 
complete graph on $2d$ vertices, in particular it is mixing.
\par
Thus, Theorem \ref{thm:maintwisting} is now proven for $d=1$, so we assume in the following that $d \geq 2$.
\par
\subsection {Existence of arrows in $\Gamma_k(C)$ for $k>1$}
As $\wh R_k$ is now a non-trivial subset of $\wt R_k$, the existence of an arrow in $\Gamma_k(C)$ is 
less trivial than for $k=1$ and requires  the symplecticity of $C$.
\par
\begin{lemma} \label{le:graphhasarrow}
Let $C \in \Sp(\Omega,\mathbb{Z})$. Then, for $1 \leq k \leq d$, the graph $\Gamma_k(C)$ has 
at least one arrow.
\end{lemma}
\begin{proof}
Let $1 \leq k \leq d$, $\ula \in \wh R_k$. Let $\Gamma'_k(C)$ be the oriented graph with set of vertices $\wt R_k$ such that there is an arrow from $\ula_0$ to $\ula_1$ if and only if $C^{(k)}_{\ula_0 \ula_1}\neq0$. In particular, $\Gamma_k(C)$ is a subgraph of $\Gamma'_k(C)$. As $\wedge^k C$ is invertible, there exists at least 
one arrow in $\Gamma'_k(C)$ starting from $\ula = \{ \la_1 < \cdots < \la_k \}$. Consider such an 
arrow, with endpoint $\ula'= \{ \la'_1 < \cdots < \la'_k \}$, with the property that the 
cardinality $k'$ of the image $p(\ula') \subset R$ is the greatest possible. We want to 
prove that $k' =k$. Assume by contradiction that $k' <k$, i.e with appropriate indexing 
$\la'_1 \la'_2 =1$. Write, for $\la \in \wt R$,
$$C(v_{\la}) = \sum_{\la'}  C_{\la \, \la'} v_{\la'}.$$
By assumption, the minor of the matrix of $C$ obtained by taking the lines in $\ula $ 
and the columns in $\ula'$ is nonzero. Therefore, we can find vectors $w_1, \ldots , w_k
 \in \Rset^{2d}$, generating the same isotropic subspace as
$v_{\la_1}, \ldots ,v_{\la_k}$, such that, for $1 \leq i \leq k$
$$C(w_i) = v_{\la'_i} + \sum\limits_{\la' \notin \ula'} C^*_{i \, \la'} v_{\la'}.$$
We claim that, if $\la'$ {\bf and} $\la'^{-1}$ do not belong to $\ula'$, then the 
coefficient $C^*_{1 \, \la'}$ is equal to zero. Indeed, otherwise, the minor of $C$ 
obtained by taking the lines in $\ula $ and the columns in $\ula_1':= (\ula' - \{\lambda_1'\}) \cup \{\la'\} $ would be nonzero, with $\# p(\ula'_1) = k'+1$, in contradiction with the definition of $k'$. Similarly, we have  $C^*_{2 \, \la'} =0$. But then, we have 
$$\Omega (C(w_1), C(w_2)) = \Omega ( v_{\la'_1}, v_{\la'_2}) \ne 0$$
as $\la'_1 \la'_2 =1$, while $\Omega(w_1,w_2) =0$, a contradiction. 
\end{proof}
\begin{remark}
 The action of the Galois group $\Gal$ has not been used; the assertion of this lemma is true 
for any $C \in \Sp(2d,\mathbb{R})$.
 \end{remark}

\subsection {Second step: transversality for $1 \leq k <d$}\label{ss.transvlessd}

For $d=2$, there is nothing new to prove, so we may assume $d \geq 3$. 
\par
\begin{proposition}\label{prop:transvlessd}
Let $C \in \Sp(2d,\mathbb{Z})$. If $C$ is $1$-twisting with respect to $A$, then $\Gamma_k(C)$ is mixing for all $1 \leq k <d$.
\end{proposition}
\par
By putting together this result with Proposition \ref{prop:hyperplanes}, one has:
\begin{corollary}\label{cor:transvlessd} 
Assume that $C \in \Sp(2d,\mathbb{Z})$ is $1$-twisting with respect to $A$. Then, for any large enough integer $m$ (depending only on $d$), the matrix
$$D = C A^{\ell_1} \cdots C A^{\ell_{m-1}} C$$
is $k$-twisting with respect to $A$ for all $1\leq k<d$ provided $\underline \ell :=(\ell_1, \cdots , \ell_{m-1})\in \Zset ^{m-1}$ goes to infinity along a rational line not lying in a finite number of hyperplanes of $\Rset ^{m-1}$. 
\end{corollary}
\begin{proof}[Proof of Proposition~\ref{prop:transvlessd}]
The graph $\Gamma_1(C)$ is complete on $2d$ vertices hence mixing. We assume now that $2 \leq k <d$. We list the orbits of the action of the Galois group $\Gal$ on 
$\wh R_k \times \wh R_k$ (an ordered pair being seen as the origin and the end of a possible arrow in $\Gamma_k$): for every pair of integers $\ell, \wt \ell$ with $0 \leq \wt \ell \leq \ell \leq k$ and $\ell \geq 2k-d$, there is an orbit $\mathcal O_{\wt \ell, \ell}$ formed by the pairs $(\ula, \ula')$ satisfying
 $$\# (\ula \cap \ula') = \wt \ell,\quad \quad\# (p(\ula) \cap p(\ula')) =  \ell.$$ 
A $\Gal$-invariant graph with vertices $\wh R_k$ is determined by which orbits $\mathcal O_{\wt \ell, \ell}$ 
are associated to arrows. We will prove in the next subsection the following:
\par
\begin{proposition} \label{prop:notmixingcrit}
A $\Gal$-invariant graph with vertices $\wh R_k$ is {\bf not} mixing if and only if  its arrows 
are associated to some of the $\mathcal O_{\wt \ell, k}$ ($0 \leq \wt \ell \leq k$) or, {\bf when $d$ is even and $k = \frac d2$} , to $\mathcal O_{0,0}$.
\end{proposition}
\par
We now finish the proof of the Proposition~\ref{prop:transvlessd} by showing that the 
non-mixing cases above do not occur for $\Gamma_k(C)$ when $C$ is $1$-twisting with respect to $A$.
 We already know that $\Gamma_k(C)$ has a non empty set of arrows, so at least one orbit of the $\Gal$-action on $\wh R_k \times \wh R_k$ must occur.
\par
{\bf Case 1.} Assume that the only occurring orbits have the form $\mathcal O_{\wt \ell, k}$ 
($0 \leq \wt \ell \leq k$).
Let $\ula = \{ \la_1 < \cdots < \la_k \} \in \wh R_k$; let $\gamma$ be an arrow of $\Gamma_k=\Gamma_k(C)$ starting from $\ula$; its endpoint $\ula' = \{ \la'_1 < \cdots < \la'_k \}$ satisfies $p(\ula) = p(\ula')$. As in the 
previous subsection, 
we can find vectors $w_1, \ldots , w_k \in \Rset^{2d}$, generating the same isotropic subspace as
$v_{\la_1}, \ldots ,v_{\la_k}$, such that, for $1 \leq i \leq k$
$$C(w_i) = v_{\la'_i} + \sum\limits_{\la' \notin \ula'} C^*_{i \, \la'} v_{\la'}.$$
We claim that, if $\la'$ {\bf and} $\la'^{-1}$ do not belong to $\ula'$, then the coefficient $C^*_{i \, \la'}$ is equal to zero for every $1 \leq i \leq k$. Indeed, otherwise, the minor of 
$C$ obtained by taking the lines in $\ula $ and the columns in $\ula_i':= (\ula' - \{\la_i'\}) \cup \{\la'\} $ would be nonzero, and the corresponding arrow of $\Gamma_k(C)$ would not be as assumed. We conclude that the image  $C(v_{\la_1})$ is a linear combination of the $2k$ vectors $v_{\la_i} , v_{\lai_i}$, $1 \leq i \leq k$,  in contradiction with the fact that $C$ is $1$-twisting with respect to $A$.
\par
{\bf Case 2.}
 Assume that $d$ is even $\geq 4$, $k = \frac d2$, and that the only occurring orbits are $\mathcal O_{0,0}$ and possibly some of the  $\mathcal O_{\wt \ell, k}$ ($0 \leq \wt \ell \leq k$). Let $\gamma$ be an arrow of $\Gamma_k$  associated to $\mathcal O_{0,0}$. Let  $\ula = \{ \la_1 < \cdots < \la_k \} \in \wh R_k$ be the origin of $\gamma$ and let $\ula' = \{ \la'_1 < \cdots < \la'_k \}$
be its endpoint. Then $p(\ula)$ and $p(\ula')$ are complementary subsets of $R$. As above, we can find vectors $w_1, \ldots , w_k \in \Rset^{2d}$, generating the same isotropic subspace as
$v_{\la_1}, \ldots ,v_{\la_k}$, such that, for $1 \leq i \leq k$
$$C(w_i) = v_{\la'_i} + \sum\limits_{\la' \notin \ula'} C^*_{i \, \la'} v_{\la'}.$$
Again, if $\la'$ {\bf and} $\la'^{-1}$ do not belong to $\ula'$,  
the coefficient $C^*_{i \, \la'}$ is equal to zero for every 
$1 \leq i \leq k$. We conclude that the image  $C(v_{\la_1})$ is a linear 
combination of the $2k =d$ vectors $v_{\la'_i} , v_{\la_i'^{-1}}$, $1 \leq i \leq k$,  in contradiction with the fact that $C$ is $1$-twisting with respect to $A$.
\end{proof}

\subsection{Proof of Proposition~\ref{prop:notmixingcrit}}

For convenience, we divide the proof in two steps. We first project down from $\wt R$ to $R$
and state the result at this level. 
\par
Let $S(R)$ be the full permutation group of $R$. The orbits of its action on $R_k \times R_k$ are as follows: for each integer $\ell$ with $ k \geq \ell \geq \max (0, 2k-d)$, one has an orbit $\mathcal O_{\ell} $ formed of the pairs $(\umu, \umu')$ with $\# ( \umu \cap \umu') = \ell$. A $S(R)$-invariant oriented graph $\Gamma$ with vertices $R_k$ is determined by which of these orbits are associated to arrows of $\Gamma$.
\par
\begin{proposition}
We assume that $d>2$ and $1\leq k <d$. Such a graph $\Gamma$ is {\bf not} mixing if and only if the orbits associated to the arrows of $\Gamma$
are $ \mathcal O_{k} $ and /or, {\bf when $d$ is even and $k = \frac d2$}, $ \mathcal O_{0} $.
\end{proposition}
\par
\begin{proof}
The excluded cases are clearly not mixing. It is therefore sufficient to prove that the $S(R)$-invariant graph $\Gamma(\ell)$ (with $ k > \ell \geq \max (0, 2k-d)$, $\ell>0$ when $k = \frac d2$) whose arrows are associated to the single orbit $\mathcal O_{\ell} $ is mixing. By passing to complements, it is sufficient to consider the case where $k \leq \frac d2$.
\par
We first show that $\Gamma(\ell)$ is strongly connected. Let $\umu_0$, $\umu_1$ be two elements of $R_k$. Let $m = \# (\umu_0 \cap \umu_1 )$. We want to find a path in $\Gamma(\ell)$ from $\umu_0$ to $\umu_1$. When $m=\ell$, a single arrow will do. 
\par
We deal with the case $m>\ell$ by ascending induction on $m$, assuming that the result is true for $m-1$. Choose $\mu \in \umu_0 \cap \umu_1$, and two distinct elements $\mu_0, \mu_1$ in $R - (\umu_0 \cup \umu_1)$ (this is possible because of our restrictions on $k, \ell$). Let $\umu'_0:= (\umu_0 -\{\mu\})\cup \{\mu_0\}, \umu'_1:= (\umu_1 -\{\mu\})\cup \{\mu_1\}$. We have $ \# (\umu_0 \cap \umu'_1 )=\# (\umu'_0 \cap \umu_1 )=\# (\umu'_0 \cap \umu'_1 )=m-1$.
By concatenation of three paths $\umu_0 \rightarrow \umu'_1 \rightarrow \umu'_0 \rightarrow \umu_1$ , we obtain a path from $\umu_0$ to $\umu_1$.
\par
We deal with the case $0<m<\ell$ by descending induction on $m$, assuming that the result is true for larger values of $m$. Choose $\mu_0 \in \umu_0 - \umu_1 , \mu_1 \in \umu_1 - \umu_0, \mu \in  R - (\umu_0 \cup \umu_1)$. Define 
$\umu'_0 = (\umu_0 -\{\mu_0\})\cup \{\mu\}, \umu'_1 = (\umu_1 -\{\mu_1\})\cup \{\mu\}$. We have $ \# (\umu_0 \cap \umu'_0 )=\# (\umu_1 \cap \umu'_1 )= k-1, \# (\umu'_0 \cap \umu'_1 )=m+1$. By concatenation of three paths $\umu_0 \rightarrow \umu'_0 \rightarrow \umu'_1 \rightarrow \umu_1$, we obtain a path from $\umu_0$ to $\umu_1$. The same argument works when $m=0$, $k<\frac d2$.
\par
Consider finally  the case $m=0$, $k= \frac d2$. As $d>2$, we have $k\geq 2$. Choose distinct elements $\mu_0, \mu'_0 \in \umu_0$ and $\mu_1, \mu'_1 \in \umu_1$. Define $\umu'_0 = (\umu_0 -\{\mu'_0\})\cup \{\mu_1\}, \umu'_1 = (\umu_1 -\{\mu'_1\})\cup \{\mu_0\}$. We have $ \# (\umu_0 \cap \umu'_0 )=\# (\umu_1 \cap \umu'_1 )= k-1, \# (\umu'_0 \cap \umu'_1 )=2$. By concatenation of three paths $\umu_0 \rightarrow \umu'_0 \rightarrow \umu'_1 \rightarrow \umu_1$ , we obtain a path from $\umu_0$ to $\umu_1$.
\par
We now show that $\Gamma(\ell)$ is mixing. If not, there is a prime number $\pi$ such that all loops of $\Gamma(\ell)$ have length divisible by $\pi$. If $\umu_0 \rightarrow \umu_1$ is an arrow, so is $\umu_1 \rightarrow \umu_0$, so the only possibility is $\pi =2$. On the other hand, the proof of connectedness has produced by induction between any vertices $\umu_0$, $\umu_1$ a path of odd length. Taking $\umu_0 = \umu_1$ (i.e $m=k$ above) gives a loop of odd length.
\end{proof}
\par
\begin{proof}[Proof of Proposition~\ref{prop:notmixingcrit}] Again, the excluded cases are 
clearly not mixing. 
Therefore, it is sufficient to prove that the oriented graph $\Gamma(\wt \ell,\ell)$ with vertex set 
$\wh R_k$ and arrows associated to a single orbit $\mathcal O_{\wt \ell, \ell}$ (not of the excluded 
type) is mixing.
\par
We first show that $\Gamma(\wt \ell,\ell)$ is strongly connected. From the previous proposition, it is sufficient to connect any two
vertices $\ula_0, \ula_1$ with the same image $\umu$ by $p$. We can even further assume that $\# (\ula_0 \cap \ula_1) = k-1$. Let $\umu' \in R_k$ such that $\# (\umu \cap \umu' )= k-1$; by the previous proposition, there is a path in $\Gamma(\ell)$ from $\umu$ to $\umu'$; lifting this path gives a path $\gamma$ in $\Gamma(\wt \ell,\ell)$ from $\ula _0$ to some vertex $\ula'$ with $p(\ula') = \umu'$. Now, there is an element $g$ of $G$ sending $\ula_0$ to $\ula'$ and $\ula'$ to $\ula_1$. Concatenating $\gamma$ and $g.\gamma$ gives a path from $\ula_0$ to $\ula_1$.
\par
Finally, we show that $\Gamma(\wt \ell,\ell)$ is mixing. Again, the only possible divisor of the lengths of all loops is $2$. To get a loop of odd length, we start from such a loop in $\Gamma(\ell)$, which we lift to get a path  of odd length between two vertices $\ula_0, \ula_1$ with the same image  by $p$. But we have just constructed above a path of even length from $\ula_1$ to $\ula_0$. By concatenation, we get the required loop.
\end{proof}

\subsection {The last step: transversality for all $1 \leq k \leq d$, $d\geq 3$}\label{ss.transv3}
In this subsection, we assume $d\geq 3$, unless otherwise stated.
\begin{proposition} \label{prop:transvalld}
Let $D \in \Sp(\Omega,\mathbb{Z})$. If $D$ is $k$-twisting with respect to $A$ for each $1 \leq k <d$, then $\Gamma_d(D)$ is mixing.
\end{proposition}
\par
By putting this result together with Proposition \ref{prop:hyperplanes}, one has:
\begin{corollary}\label{cor:transvalld} Assume that $D \in \Sp(\Omega,\mathbb{Z})$ is $k$-twisting with respect to $A$ for each $1 \leq k <d$. 
Then, for any large enough integer $m$ (depending only on $d$), the matrix
$$E = D A^{\ell_1} \cdots D A^{\ell_{m-1}} D$$
is $k$-twisting with respect to $A$ for all $1 \leq k \leq d$ provided 
$\underline \ell :=(\ell_1, \cdots , \ell_{m-1})\in \Zset ^{m-1}$ goes to infinity 
along a rational line not lying in any of a finite number of hyperplanes of $\Rset ^{m-1}$. 
\end{corollary}
\par
\begin{proof}[Proof of Proposition~\ref{prop:transvalld}]
 The orbits of the action of the Galois group on $\wh R_d \times \wh R_d$ are as described in Subsection \ref{ss.transvlessd} (now with $k=d$); the restriction $2k-d \leq \ell \leq k$ now forces $\ell =d$ and we are left with one parameter $0 \leq \wt \ell \leq d$. We write $\mathcal O(\wt \ell)$ for $\mathcal O_{\wt \ell , d}$.
\par 
 A $Gal$-invariant graph $\Gamma_d$ on $\wh R_d$ is determined by the subset $J \subset \{0,\ldots ,d \}$ of integers $\wt \ell$ indexing  the orbits  associated  to the arrows of $\Gamma_d$. 

 \begin{lemma} \label{le:consecint}
 If $J$ contains two consecutive integers $\wt \ell$, $\wt \ell +1$, then $\Gamma_d$ is mixing.
 \end{lemma}

\begin{proof}[Proof of Lemma~\ref{le:consecint}] We first show that $\Gamma_d$ is strongly connected. 
To prove this, it is sufficient to connect two vertices $\ula_0$, $\ula_1$ such that $ \# (\ula_0 \cap \ula_1) = d-1$. This is done by choosing a subset $\ula$ of size $d -\wt\ell -1$ in  
 $\ula_0 \cap \ula_1$ and calling $\ula_2$ the element of $\wh R_d$ obtained from $\ula_0$ by replacing the elements of $\ula$ by their inverses; one has $ \# (\ula_0 \cap \ula_2) = \wt\ell +1,  \# (\ula_1 \cap \ula_2) = \wt\ell$, hence there are arrows $\ula_0 \rightarrow \ula_2 \rightarrow \ula_1$. 

As before, the only possible non trivial common divisor for the lengths of the loops is $2$. But concatenating  a single arrow associated to $\mathcal O(\wt \ell)$ to $ d-\wt \ell$  paths of length $2$ as constructed above, we get a loop of odd length. This shows that $\Gamma_d$ is mixing. 
 \end{proof}
 
\begin{remark}\label{r.le-consecint} This proof works for any $d\geq 2$. We will use this fact in the next subsection. 
 \end{remark}

In view of this lemma, the proof of the proposition is complete as soon as we show that the subset $J$ associated to  $\Gamma_d(D)$ 
contains two consecutive integers. We assume by contradiction that 
this is not the case. From Lemma~\ref{le:graphhasarrow} we know that  
$J$ is not empty. We consider successively several cases.
\par
{\bf Case 1.} Assume first that $J$ contains an integer $\ell$ with $2 \leq \ell <d$. 
Let $\ula_0 \rightarrow \ula_1$ be an arrow of this type in $\Gamma_d (D)$, with 
$\ula_0 = \{ \la_1 , \cdots , \la_d \}, \ula_1 = \{ \la_1 , \cdots , \la_{\ell},\la_{\ell +1}^{-1}, 
\ldots, \la_d^{-1}   \}$. 
\par
We can find vectors $w_1, \ldots , w_d \in \Rset^{2d}$, generating the same Lagrangian subspace as
$v_{\la_1}, \ldots ,v_{\la_d}$, such that, for $1 \leq i \leq d$
$$D(w_i) = v(i) + \sum_1^d D^*_{i \, j} v'(j),$$
where $v(i)= v_{\la_i}, v'(i) = v_{\la_i^{-1}}$ for $1 \leq i \leq \ell$,
 $v(i) = v_{\la_i^{-1}}, v'(i)= v_{\la_i}$ for
$\ell <i \leq d$.
As $\ell \pm 1$ do not belong to $J$ by hypothesis, there is no arrow in $\Gamma_d$ from 
$\ula_0$ to any vertex obtained from $\ula_1$ by replacing one of its element by its inverse. 
This implies that the diagonal coefficients $D^*_{i \,i}$ are zero for $1 \leq i \leq d$.
\par
Next, take $1 \leq i \leq \ell <j \leq d$, and $\ula'_1$ to be the vertex obtained by replacing in $\ula_1$ both $\la_i$ and $\lai_j$ by their inverses. We have  $\# (\ula_0 \cap \ula'_1) = \ell$, hence there is an arrow from $\ula_0$ to $\ula'_1$ in $\Gamma_d(D)$. This implies , as $D^*_{i \,i} = D^*_{j \,j}=0$, that $D^*_{i \,j}D^*_{j \,i}\ne 0$. On the other hand, as $\Omega(D(w_i),D(w_j)) =0$, we have, writing 
$\Omega(v_{\la_i},v_{\la_i^{-1}}):= \omega_i$
$$ D^*_{j \,i} \omega_i + D^*_{i \,j} \omega_j =0.$$
\par
When we take instead $1 \leq i <j \leq \ell$, as $\Omega(D(w_i),D(w_j)) =0$, we have
$$ D^*_{j \,i} \omega_i - D^*_{i \,j} \omega_j =0.$$

Now, let $\ula''_1$  be the vertex obtained by replacing (in $\ula_1$)  $\la_1, \la_2$ and $\lai_d$ by their inverses. We have  $\# (\ula_0 \cap \ula''_1) = \ell -1$, hence there is no arrow from $\ula_0$ to $\ula''_1$ in $\Gamma_d(D)$.
Computing the corresponding $3 \times 3$ minor in $D^*$ gives (as the diagonal terms vanish)
$$ D^*_{1 \,2}D^*_{2 \,d}D^*_{d \,1}+ D^*_{1 \,d}D^*_{d \,2}D^*_{2 \,1}=0.$$
As the $\omega_i$ are nonzero, the symmetry/antisymmetry properties of the $D^*_{i \,j}$ force 
$$ D^*_{1 \,2}D^*_{2 \,d}D^*_{d \,1} =0 = D^*_{1 \,d}D^*_{d \,2}D^*_{2 \,1}.$$ 
As $D^*_{1 \,d}, D^*_{d \,1}, D^*_{2 \,d}, D^*_{d \,2}$ are non zero, we have $ D^*_{1 \,2} = D^*_{2 \,1}=0$. Take a nonzero linear combination $w$  of $w_1, w_2$ which is also a linear combination of $v_{\la_1}, \ldots, v_{\la_{d-1}}$ (eliminating the coefficient of $v_{\la_d}$). The image  $D(w)$ is then a linear combination of $v(1), v(2)$, and the $v'(j))$ for $2<j \leq d$. This contradicts the fact that $D$ is $(d-1)$-twisting with respect to $A$.
\par
{\bf Case 2.} The case where $1 \in J$ (with $d \geq 3$) is dealt with in a symmetric way.
\par
{\bf Case 3.} 
Assume that $J$ contains no element except (possibly) the endpoints $0$ and $d$. Assume that for instance $d \in J$. Let $\ula_0 = \{ \la_1 , \cdots , \la_d \}$. As the loop at $\ula_0$ is an arrow of $\Gamma_d(D)$, we can find vectors $w_1, \ldots , w_d \in \Rset^{2d}$, generating the same Lagrangian subspace as
$v_{\la_1}, \ldots ,v_{\la_d}$, such that, for $1 \leq i \leq d$
$$D(w_i) = v(i) + \sum_1^d D^*_{i \, j} v'(j),$$
where $v(i)= v_{\la_i}, v'(i) = v_{\la_i^{-1}}$ for $1 \leq i \leq d$.
\par
As above, from $d-1 \notin J$, we obtain that the diagonal coefficients $D^*_{i \,i}$ are zero for $1 \leq i \leq d$. But now, as $d-2 \notin J$ (recall that $d\geq3>2$), the $2 \times 2$ diagonal minors of $D^*$ are zero, implying $D^*_{i \,j}D^*_{j \,i}= 0$ for $1 \leq i <j \leq d$. As we still have 
$$ D^*_{j \,i} \omega_i - D^*_{i \,j} \omega_j =0,$$ 
we have in fact $D^*_{i \,j}=D^*_{j \,i}= 0$. We conclude as before that the fact that $D$ is $(d-1)$-twisting with respect to $A$ is violated.
\par
The case where $0$ is the unique element of $J$ is treated in the same way.
\par
This concludes the proof of the proposition.
\end{proof}

\begin{remark}\label{r.case3-d2} The argument used in Case 3 above is still valid for $d=2$ if we assume that $J=\{0\}$ or $J=\{2\}$. Indeed, a quick inspection of the previous argument reveals that it works once $d\in J$, and $d-1\notin J$ and $d-2\notin J$, or $J=\{0\}$.
\end{remark}

At this point, by putting together the result of Subsection \ref{ss.transv1} and Corollaries \ref{cor:transvlessd}, \ref{cor:transvalld}, our Theorem \ref{thm:maintwisting} is now proven for $d=1$ and $d\geq 3$. Thus, it remains only to consider the special case $d=2$. This is the content of the next subsection.

\subsection{The case $d=2$}\label{ss.main-twisiting-d-2}

 The kernel $K$ of the linear form $\wedge^2 \Rset^4 \rightarrow \Rset $ determined by the symplectic 
form $\Omega$ is invariant by the action of the symplectic group $\Sp(\Omega,\mathbb{R})$ on $\wedge^2 \Rset^4$. 
A basis of $K$ is formed by the four vectors $v_{\ula}$, $\ula \in \wh R_2$ and 
$$v_*:= \frac {v_{\la_1} \wedge v_{\lai_1} }{\omega_1} - \frac {v_{\la_2} \wedge v_{\lai_2} }{\omega_2},$$
where $\wt R = \{\la_1,\lai_1,\la_2, \lai_2\}$ and $\omega_i:= \Omega(v_{\la_i},v_{\lai_i})$ for $i=1,2$.
Observe that $v_*$ is an eigenvector of $\wedge^2 A$ with eigenvalue $1$. Thus, the  eigenvalues of $\wedge ^2 A$ on $K$ are {\bf distinct}. This would not be true for $d>2$, as $1$ is then a multiple eigenvalue.

Let $C \in \Sp(\Omega,\mathbb{Z})$. We define an oriented graph $\Gamma^*_2 = \Gamma^*_2(C)$ as follows:
the vertices of $\Gamma^*_2$ are the four elements of $\wh R_2$ and a special vertex $*$ (associated to the  eigenvalue $1$ of $\wedge ^2 A$); for vertices $\ula_0, \ula_1 $, there is an arrow from $\ula_0$ to $\ula_1$ if and only if the coefficient $C^{(2)}_{\ula_0 \, \ula_1}$ of the matrix of $\wedge^2 C$, written in the basis $(v_{\ula})$ of $K$, is nonzero.
\par
 {\bf Assume that $\Gamma^*_2(C)$ is  mixing}. Let $m$ be a positive integer such that, 
for every vertices $\ula_0,\ula_1$ of $\Gamma^*_2(C)$, there is an oriented path of length $m$ from $\ula_0$ to $\ula_1$. For nonnegative integers $\ell_1, \cdots , \ell_{m-1}$, consider
$$D = C A^{\ell_1} \cdots C A^{\ell_{m-1}} C.$$ 
\par
\begin{proposition}
There are finitely many hyperplanes $V_1, \cdots ,V_t$ in $\Rset ^{m-1}$ such that, if $\underline \ell :=(\ell_1, \cdots , \ell_{m-1})\in \Zset ^{m-1}$ goes to infinity along a rational line not lying in any of the $V_p$, then, for $||\underline \ell ||$ large enough, the matrix $D$ is $2$-twisting with respect to $A$.
\end{proposition}
\par
\begin{proof}
The proof is the same as for Proposition~\ref{prop:hyperplanes} (cf. Remark \ref{r.prop-hyperplanes-1-twisting}), using that all eigenvalues of 
$\wedge^2 A$ are real and simple.
\end{proof}
\par 
In view of this proposition, in order to conclude the proof of Theorem \ref{thm:maintwisting} in the case $d=2$, it suffices to prove the following statement.
\par
\begin{proposition}
Let $C \in \Sp(\Omega,\mathbb{Z})$. If $C$ is $1$-twisting with respect to $A$, then $\Gamma_2(C)$ or $\Gamma^*_2(C)$ is mixing.
\end{proposition}
\par
\begin{proof}
We start by considering $\Gamma_2(C)$. We can define $J \subset \{ 0,1,2 \} $ as in the proof of 
Proposition~\ref{prop:transvalld}. This non-empty subset determines the $\Gal$-invariant graph $\Gamma_2(C)$. 
If $J$ contains two consecutive integers, Lemma~\ref{le:consecint} (still valid for $d=2$, cf. Remark \ref{r.le-consecint}) 
implies that $\Gamma_2(C)$ is mixing. On the other hand, if $J=\{0\}$ or $\{2\}$, one 
proves as in the last subsection (cf. Remark \ref{r.case3-d2}) that $C$ is not $1$-twisting with respect to $A$, a contradiction. 
\par
Therefore, there are two remaining cases to be considered: $J = \{1\}$, and $J=\{0,2\}$. In these 
cases, $\Gamma_2(C)$ is not mixing, but we will show that $\Gamma^*_2(C)$ is mixing. Notice that the mixing
 property follows from the existence of arrows in $\Gamma^*_2(C)$ from any vertex of $\wh R_2$ 
to the special vertex $*$, and of arrows from this special vertex to any vertex of $\wh R_2$.
\par
{\bf Existence of arrows in $\Gamma^*_2(C)$ from any vertex of $\wh R_2$ to the special vertex $*$}. 
We assume that $J = \{1\}$, the other case is dealt in the same way. Let 
$\ula_0 = \{ \la_1, \la_2 \} \in \wh R_2$. We can find vectors $w_1 , w_2 \in \Rset^{4}$, generating the same Lagrangian subspace as
$v_{\la_1} ,v_{\la_2}$, such that
$$C(w_1) = v_{\la_1} +  C^*_{1 \, 1} v_{\lai_1} + C^*_{1 \, 2} v_{\la_2},$$
$$C(w_2) = v_{\lai_2} +  C^*_{2 \, 1} v_{\lai_1} + C^*_{2 \, 2} v_{\la_2}.$$
As $0,2 \notin J$, we have $ C^*_{1 \, 1}= C^*_{2 \, 2} = 0$. As $1 \in J$, we have 
$C^*_{1 \, 2} C^*_{2 \, 1} \ne 0$.
It follows that there is an arrow from $\ula_0$ to the special vertex $*$ in $\Gamma^*_2(C)$.
\par
{\bf Existence of arrows in $\Gamma^*_2(C)$ from the special vertex $*$ to any vertex of $\wh R_2$}.
If there were no arrows from the special vertex $*$ to any other vertex in $\Gamma^*_2(C)$, then $v_*$ 
would be an eigenvector of $\wedge^2 C$. By Proposition~\ref{prop:215} in the next subsection (see below), the image by $C$ of the 
symplectic plane generated by $v_{\la_1}, v_{\lai_1}$ is consequently either itself or the 
symplectic plane generated by $v_{\la_2}, v_{\lai_2}$. In any case, this  contradicts the fact that $C$ is $1$-twisting with respect to $A$.
\end{proof}

\subsection{Symplectic $2$-planes} Let $H_1, H_2$ be  orthogonal $2$-planes in $\Rset^4$ equipped with the standard symplectic structure. Choose a basis $e_1, f_1$ of $H_1$, a basis $e_2, f_2$ of $H_2$, normalized by $\Omega(e_1, f_1) = \Omega(e_2, f_2)=1$. The non-zero bivector 
$$ V(H_1):= e_1\wedge f_1 - e_2 \wedge f_2 $$
does not depend on the choices of the two bases of $H_1$ and its symplectic orthogonal $H_2$, i.e., it depends only on $H_1$.

\begin{proposition}\label{prop:215}
The map $H \rightarrow V(H)$ from symplectic $2$-planes in $\Rset^4$ to $\wedge^2 \Rset^4$ is injective. More precisely,
if $H,H'$ are symplectic $2$-planes in $\Rset^4$ such that $V(H)$ and $V(H')$ are collinear, then $H, H'$ are either equal or orthogonal.
\end{proposition}
\begin{proof}
Let $e_1, f_1,e_2, f_2$ be the canonical basis of $\Rset^4$, and let $H$ be a symplectic $2$-plane. We have to show that, if $V(H)$ is proportional to $ e_1\wedge f_1 - e_2 \wedge f_2 $, then $H$ is either spanned by $e_1,f_1$ or by $e_2,f_2$. Let $E_1,F_1$ be a basis of $H$ and $E_2,F_2$ be a basis of the orthogonal $H'$ of $H$, such that
$\Omega(E_1,F_1)=\Omega(E_2,F_2)=1$. Write
$$E_i = \sum_1^2 a_{i\,j} e_j +  \sum_1^2 b_{i\,j} f_j ,$$
$$F_i = \sum_1^2 c_{i\,j} e_j +  \sum_1^2 d_{i\,j} f_j .$$
Let 
$$M_{i\,j} = \left( \begin{array}{cc}   a_{i\,j}& b_{i\,j} \\  c_{i\,j} & d_{i\,j} \\  \end{array} \right).$$

By changing the symplectic bases of the planes $\langle e_1,f_1 \rangle$, $\langle e_2,f_2\rangle$, $\langle E_1,F_1\rangle$, 
$\langle E_2,F_2\rangle$, we change the matrices $M_{i\,j}$ to $M'_{i\,j}= P_i M_{i\,j} Q_j$, where $P_1, P_2,Q_1,Q_2$ are arbitrary matrices in $\SL(2,\Rset)$.
\begin{lemma}
For $i = 1,2$, at least one of the matrices $M_{i\,1},M_{i\,2}$ is invertible. For $j = 1,2$, at least one of the matrices $M_{1\,j},M_{2\,j}$ is invertible.
\end{lemma}
\begin{proof}
We show that $M_{1\,1}$ or $M_{1\,2}$ is invertible. Otherwise, by an appropriate choice of $Q_1, Q_2$, we obtain $b_{1 \,1}=d_{1 \,1}=b_{1 \,2}=d_{1 \,2}=0$, which contradicts $\Omega(E_1,F_1)=1$.

The proof of the other cases is similar.
\end{proof}
Changing $H$ to its orthogonal $H'$ changes $V(H)$ to $-V(H)$ and  exchanges $M_{1\,j}$ and  $M_{2\,j}$. From the lemma, we may thus assume after performing if necessary such a change, that $M_{1\,1}$ and $M_{2\,2}$ are invertible.
Then, choosing appropriately $P_1, P_2,Q_1,Q_2 \in \SL(2,\Rset)$ allows to get
$$b_{1 \,1}=b_{1 \,2}=b_{2 \,2}=c_{1\,1}=c_{1\,2}=c_{2\,2}=0, $$
$$a_{1\,1}=a_{2\,2}=1, d_{1 \,1}\ne 0, d_{2 \,2}\ne 0.$$
Then, the relations $\Omega( E_1,E_2) = \Omega(F_1,F_2)=0$ give also $b_{2\,1} = c_{2\,1} =0$, so all $M_{i\,j}$ are diagonal. The relations $\Omega( E_1,F_2) = \Omega(F_1,E_2)=0$ give
$$ d_{2\,1} + a_{1\,2} d_{2\,2}= 0 = d_{1\,2} + a_{2\,1} d_{1\,1}.$$
On the other hand,  the relations $\Omega(E_1,F_1) = \Omega(E_2,F_2)=0$ give
 $$ d_{1\,1} + a_{1\,2} d_{1\,2}= 1 = d_{2\,2} + a_{2\,1} d_{2\,1}.$$
 Finally, the assumption 
 $$ E_1\wedge F_1 - E_2 \wedge F_2= k(e_1\wedge f_1 - e_2 \wedge f_2)$$
 give
 $$ d_{2\,1} - a_{1\,2} d_{2\,2}= 0 = d_{1\,2} - a_{2\,1} d_{1\,1},$$
 $$ d_{1\,1} - a_{1\,2} d_{1\,2}=k=d_{2\,2} - a_{2\,1} d_{2\,1}.$$
 We obtain, as $d_{1\,1}$, $d_{2\,2}$ are non-zero, that $d_{1\,2}=d_{2\,1}= a_{1\,2}=a_{2\,1}=0$, which is the required conclusion.
\end{proof}

\section{Proof of the simplicity criteria} \label{sec:redtoAV}

In this Section, we will complete the proofs of Theorem~\ref{thm:mainsimpcrit} and 
Corollary~\ref{cor:mainsimpcrit}. Then we will present and prove a variant of Theorem~\ref{thm:mainsimpcrit}.

\subsection{Two lemmas}

\begin{lemma}\label {powers}
Let $K$ be a field of characteristic zero and let $B$ be a unipotent endomorphism of a finite-dimensional 
vector space $E$ over $K$. Let $m$ be a nonzero integer. Then, any subspace of  $E$ which is invariant 
under $B^m$ is also invariant under $B$.
\end{lemma}
\begin{proof}
Write $B = \textrm{id} + N$ with $N$ nilpotent, and $B^m = \textrm{id} + N'$. Then, by the binomial formula, $N'$ is nilpotent (since $N$ is nilpotent). It follows that $B=(I+N')^{1/m}$ can be computed by the appropriately truncated binomial series\footnote{That is, we use that $N'$ is nilpotent to interpret the (formal) binomial series $(I+N')^{a} = \sum\limits_{k=0}^{\infty} \binom{a}{k}(N')^k$ (where $a\in\mathbb{C}$ and $\binom{a}{k}:=a(a-1)\dots(a-k+1)/k!$) as a polynomial function of $N'$.} and, hence, $B$ is a polynomial function of $N'=I-B^m$. Thus $B$ is a polynomial function of $B^m$. The assertion of the lemma is an immediate consequence of this fact.
 \end{proof}
 
 \begin{lemma}\label {b-reducedproduct}
 Let $A,B$ be two elements of $\SL_2(\Zset)$. Assume that $A$ is b-reduced and that $\textrm{tr} \,(B)=2$.
 Then, after replacing if necessary $B$ by $B^{-1}$, there exists $n_0$ such that, for any $n_1,n_2,n_3 \geq n_0$, the matrix $A^{n_1} B^{n_2} A^{n_3} $ is b-reduced.
 \end{lemma} 
 \begin{proof}
 If $B = \textrm{id}$, the assertion of the lemma follows from the fact that the product of b-reduced matrices is b-reduced. For the rest of the proof, we assume that $B \ne \textrm{id}$, and denote by $L$ a 
 rational line which is fixed by $B$. 
 \par
Define the cones in $\Rset^2$ (using the natural basis $(e_1,e_2)$ and coordinates $(x_1,x_2)$)

$$ \mathcal C^+ := \{ 0 < x_2<x_1\}, \quad \mathcal C^- := \{ 0<-x_1<x_2\}.$$

Observe that, if $C \in \SL_2(\Zset)$ satisfies, for $i=1,2$, the two conditions
$$ C(e_i) \in  \mathcal C^+ \quad \text{and} \quad C^{-1} (e_i) \in  \mathcal C^- \cup -\mathcal C^-$$
then $C$ is b-reduced. Conversely, if $C$ is b-reduced, its trace is $\geq 3$. Moreover, by a simple application of the Perron-Frobenius theorem (or a direct calculation),
the cone $\mathcal C^+$ contains a unit eigenvector $e_u(C)$ of $C$ associated to the eigenvalue $>1$, and  the cone $\mathcal C^-$ contains a unit eigenvector $e_s(C)$ associated to the eigenvalue in $(0,1)$.

\smallskip

As $L$ is rational and the line spanned by $e_s(A)$ is not rational, there exists a unit vector $e$ generating 
$L$ such that, in the decomposition $e = \alpha_s e_s(A) + \alpha_u e_u(A)$, one has 
$\alpha_u >0$. Using next that the line spanned by $e_u(A)$ is distinct from $L$, we may assume, 
replacing $B$ by $B^{-1}$ if necessary, that  the oriented direction of $B^n(e_u(A))$ for large $n$ is close to 
$e$. Then  the oriented direction of $A^{n_1} B^{n_2} A^{n_3}(e_i) $ is, for large $n_1,n_2,n_3$ and 
$i=1,2$, close to $e_u(A)$. 
The oriented direction of $A^{-n_3} B^{-n_2} A^{-n_1}(e_i) $ is close to $\pm e_s(A)$.
Taking into account the remarks above on b-reduced matrices, the proof of the lemma is complete.
\end{proof}

\subsection {  Proof of Theorem \ref{thm:mainsimpcrit} } \label{ss:pf}

We want to show Theorem \ref{thm:mainsimpcrit} by applying the version of the simplicity criterion of 
Avila-Viana restated as Theorem~\ref{thm:AVcrit} .
\par
Let $M,A,B$ be as in Theorem \ref{thm:mainsimpcrit}. 
The linear part  of $A$ has trace $>2$, hence is conjugated in $\SL(2,\Zset)$ to a b-reduced matrix
(cf. Remark \ref{conjSL2}). After replacing $M$ if necessary by another origami in the same $\SL_2(\Zset)$-orbit,
we may assume that the linear part of $A$ is b-reduced.
\par The full shift appearing in the statement of Theorem \ref{thm:AVcrit} is the one that has been considered in Subsection \ref{ss.return}: the letters of the corresponding alphabet are the loops in $\Gamma(M)$ which go 
exactly once through $(M,b)$. It was shown in that subsection that the version of the Gauss measure corresponding to this setting has the bounded distortion property, and that the KZ-cocycle is locally constant and integrable.
\par
 By Corollary \ref{characterization_b-reduced}, the affine map $A$ is associated to some loop in $\Gamma(M)$ through $(M,b)$, i.e to a word in this alphabet.

 \par
By Lemma \ref{b-reducedproduct}, there exists $n_0 >0$ such that, for $n_1,n_2,n_3 \geq n_0$, the linear part of $A^{n_1} B^{n_2} A^{n_3} $ is b-reduced. Define $B' := B^{n_0}$.
\par 
We claim that the hypotheses of Theorem \ref{thm:maintwisting} (with $d=g-1$, $\Omega$ the symplectic intersection form on $H_1^{(0)}(M,\Zset)$) are satisfied by the endomorphisms of 
$H_1^{(0)}(M,\Zset)$ induced by $A$ and $B'$. Indeed, the endomorphism induced by $A$ is Galois-pinching from the hypothesis of Theorem \ref {thm:mainsimpcrit}. If there were a proper subspace of $H_1^{(0)}(M,\Rset)$ invariant by the action of $A$ and $B'^2$, it would also be invariant by $B$ (Lemma \ref{powers}).  Applying Proposition \ref{parabolictwist} would give a contradiction with the hypothesis of Theorem  \ref {thm:mainsimpcrit}.
\par
Let $C =  B' A^{\ell_1} \cdots B' A^{\ell_{m-1}} B'$ be an element (given by Theorem \ref{thm:maintwisting})
whose linear part is $k$-twisting with respect to $A$ for all $1\leq k\leq g-1$. We may assume, according to the statement of this theorem, that the integers $\ell_1$, $\ldots$, $\ell_{m-1}$ are $\geq 2 n_0$. Observe that the linear
part of $C':= A^{n_0} C A^{n_0}$ is also $k$-twisting with respect to $A$ for all $1\leq k\leq g-1$.
\par 
By Lemma \ref{b-reducedproduct} (using also that the product of b-reduced matrices is b-reduced), the linear part of $C'$ is b-reduced. Therefore, the affine map $C' \in \textrm{Aff}(M)$ is associated to some loop through 
$(M,b)$ in $\Gamma(M)$, i.e to some word in the alphabet.
\par 
Applying Theorem~\ref{thm:AVcrit} , we conclude that the Lyapunov spectrum of the discrete version of the KZ-cocycle
$G_M^{KZ}$ discussed in Subsection \ref {ss.KZcoding} is simple. 
This concludes the proof of Theorem \ref {thm:mainsimpcrit}. \hfill  $\Box$

\subsection {  Proof of Corollary \ref{cor:mainsimpcrit} } \label{ss:pfcor}

Let $(M,\omega)$ be an origami without nontrivial automorphism. 
The hypotheses regarding $A$ are the same in Theorem~\ref{thm:mainsimpcrit} and 
Corollary~\ref{cor:mainsimpcrit}. To prove  Corollary~\ref{cor:mainsimpcrit}, 
it is therefore sufficient to find, assuming the existence of a rational 
direction with homological dimension $\ne 1,g$, 
an affine homeomorphism $B$
which acts on $H_1^{(0)}(M,\mathbb{Q})$ through a unipotent endomorphism distinct from the identity, and
 such that the image of $B- \textrm{id}$ is not a lagrangian subspace of $H_1^{(0)}(M,\mathbb{Q})$.

\par
We may assume that the rational direction in the hypothesis is the horizontal direction. 
Let $\mathcal{C}$ be the set of horizontal cylinders. For a cylinder $C \in \mathcal{C}$, we denote by $\sigma_C \in H_1(M,\Zset)$ the homology class of the waist curve of $C$, oriented rightwards. We also choose a path with endpoints in $\Sigma$ crossing $C$ upwards and denote by $\upsilon_C$ its relative homology class in $H_1(M,\Sigma,\Zset)$. Let $E$ be the 
subspace of $H_1(M,\Qset)$ spanned by the $\sigma_C$.
By hypothesis, we have 
$$ 1 < \dim E < g .$$
On the other hand, the image of $E$ under the map 
$\pi_*: H_1(M,\Qset) \to   H_1(\Tset^2,\Qset) \simeq \Qset^2$
is the one-dimensional horizontal subspace. Therefore the intersection of $E$ with $H_1^{(0)}(M,\mathbb{Q})$
satisfies  
$$ 0 < \dim (E \cap H_1^{(0)}(M,\mathbb{Q}))< g-1 .$$
\par 
Let $K >0$ be an integer such that
\begin{itemize}
\item The matrix $\left( \begin{smallmatrix} 1 & K \\ 0 & 1 \end{smallmatrix} \right) $ belongs to the Veech group of 
$(M,\omega)$.
\item The affine homeomorphism $B$ of $(M,\omega)$ with linear part $\left( \begin{smallmatrix} 1 & K \\ 0 & 1 \end{smallmatrix} \right) $ fixes each horizontal separatrix.
\end{itemize}
We claim that $B$ has the required properties. 
\par

Indeed, the relative homology group $H_1(M,\Sigma,\Zset)$ is spanned by the $\upsilon_C$, for $C \in \mathcal C$, and the classes of the horizontal saddle connections. The class of any horizontal saddle connection is fixed by $B_*$, while the image of $\upsilon_C$ satisfies 
$$ B_*(\upsilon_C) - \upsilon_C  = m_C \sigma_C$$
for some positive integer $m_C$. As $\sigma_C$ is fixed by $B_*$, the operator $B_*-\textrm{id}$ on $H_1(M,\Sigma,\Qset)$ is nilpotent of degree $2$. The image of  $H_1^{(0)}(M,\Qset)$ under $B_*-\textrm{id}$  is contained in
 $E \cap H_1^{(0)}(M,\mathbb{Q})$, hence is not lagrangian. 

\par
It remains only to show that the restriction of $B_*$ to 
 $H_1^{(0)}(M,\Qset)$ is not equal to the identity. This is a direct consequence of the following result (of independent interest):
 
 \begin{lemma}
Let $(M,\omega)$ be an origami such that the horizontal direction has homological dimension $>1$. Let
 $B$ be an affine homeomorphism of $M$ whose derivative is of the form $\left( \begin{smallmatrix} 1 & K \\ 0 & 1 \end{smallmatrix} \right) $ for some $K>0$. Then $B_*$ acts nontrivially on  $H_1^{(0)}(M,\Qset)$.
 \end{lemma}

\begin{proof}
Replacing if necessary $B$ by $B^m$ for some $m>0$, we may assume that $B$ fixes each horizontal separatrix. Define an oriented graph $\Gamma$ as follows:
\begin{itemize}
\item The vertices of $\Gamma$ are the connected components of $M_h := M - \bigcup_{C \in \mathcal C} C$;
\item The edges of $\Gamma$ are in one-to-one correspondence with the cylinders in $\mathcal C$: to each $C \in \mathcal C$ is associated an oriented edge $e_C$ of $\Gamma$ going from the component $b(C)$ of $M_h$ containing the bottom boundary of $C$ to the  the component $t(C)$ of $M_h$ containing the top boundary of $C$.
\end{itemize} 
Observe that, for each vertex $v$ of $\Gamma$, one has
$$(\star)\qquad  \sum_{b(C) = v} \sigma_C = \sum_{t(C) = v} \sigma_C.$$
Indeed, these are the homology classes of the two components of an $\va$-neighborhood of $v$ in $M$.

\par

Each edge $e_C$ of $\Gamma$ is contained into at least  one {\bf oriented} loop  of $\Gamma$: one can for instance consider the projection  to $\Gamma$ of an orbit of the vertical flow of $(M,\omega)$ through a generic point of $C$. It follows (after eliminating redundant loops) that each edge of $\Gamma$ is contained into at least one {\bf simple oriented} loop of $\Gamma$.

\par

Next, we claim that there exist at least two distinct simple oriented loops in $\Gamma$: otherwise, $\Gamma$ would consist of a single simple oriented loop; from the relation $(\star)$ above, all classes $\sigma_C$, $C \in \mathcal C$ would be equal, which contradicts the hypothesis on the homological dimension.

\par

Consider two distinct simple oriented loops $\gamma_1$, $\gamma_2$ of $\Gamma$. As both loops are simple,
there exists an edge of $\gamma_1$ not contained in $\gamma_2$ and  an edge of $\gamma_2$ not contained in $\gamma_1$.

\par

For $i = 1,2$, consider a loop $\wt \gamma_i$ in $M$ which is a concatenation of the $\upsilon_C$ such that $e_C \in \gamma_i$ and some horizontal saddle connections. As $B$ fixes each horizontal saddle connection, the homology class $\bar \gamma_i$ of $\wt \gamma_i$ satisfies
$$ B_* (\bar \gamma_i) - \bar \gamma_i = \sum_{e_C \subset \gamma_i} m_C \sigma_C.$$
\par

Denote by $I_{i\,j}$ the homological intersection of $ B_* (\bar \gamma_i) - \bar \gamma_i $ and $\bar \gamma_j$. We have, for $i=1,2$
$$ I_{i\,i} = \sum_{e_C \subset \gamma_i}  m_C, \quad I_{1\,2} = I_{2\,1} = \sum_{e_C \subset \gamma_1 \cap \gamma_2}  m_C,$$
and therefore, as the $m_C$ are positive integers,
$$ \min( I_{1\,1}, I_{2\,2}) >  I_{1\,2} = I_{2\,1}.$$

\par

It follows that, for any $(c_1,c_2) \in \Rset^2$, $(c_1,c_2) \ne (0,0)$, the image of $c_1 \bar \gamma_1 + c_2 \bar \gamma_2$ under $B_* - {\rm id}$ intersects at least one of the $\bar \gamma_j$ in a non trivial way.

\par
But one can choose $(c_1,c_2) \ne (0,0)$ and a class $\sigma \in E$ such that  $c_1 \bar \gamma_1 + c_2 \bar \gamma_2 + \sigma \in H_1^{(0)} (M,\Zset)$. The image of this class under    $B_* - {\rm id}$  is the same than the image of $c_1 \bar \gamma_1 + c_2 \bar \gamma_2$, hence is different from $0$.
\end{proof}

\subsection {A variant of Theorem \ref{thm:mainsimpcrit} } \label{ss:variant}

Instead of using one unipotent affine homeomorphism to obtain the twisting condition relative to the 
Galois-pinching  homeomorphism $A$, one may use another  pseudo-Anosov homeomorphism.

\begin{theorem} \label{thm:variant}
Let $(M,\omega)$ be a reduced square-tiled surface having no nontrivial automorphism. Assume that there
 exist two affine homeomorphisms $A,B$ of $(M,\omega)$ with the following properties:
 \begin{itemize}
 \item[i)] Both linear parts $DA,DB$ are b-reduced;
 \item [ii)] $A$ is Galois-pinching;
 \item [iii)] The minimal polynomial of the endomorphism of $H_1^{(0)}(M,\Qset)$ induced by $B$ has degree $>2$, no irreducible factor of even degree, and its splitting field is disjoint from the splitting field of the characteristic polynomial  of the endomorphism of $H_1^{(0)}(M,\Qset)$ induced by $A$.
 \end{itemize}
Then, the Lyapunov spectrum of the KZ-cocycle, relative to the $\SL_2(\Rset)$-invariant probability measure supported by the $\SL_2(\Rset)$-orbit of $(M,\omega)$ in moduli space, is simple.
\end{theorem}

\begin{proof}

As in Subsection \ref{ss:pf}, we plan to apply Theorem~\ref{thm:AVcrit}, using Theorem \ref{thm:maintwisting}
to get the twisting hypothesis. The \Teichmuller flow restricted to the $\SL_2(\Rset)$-orbit of $(M,\Omega)$
 is still viewed as a suspension over the full shift on a countable alphabet, whose letters are the loops in $\Gamma(M)$ (based at $(M,b)$). Then hypothesis i) means that $A,B$ correspond to words in this alphabet
 (see Subsection \ref{ss.loops-pA}). It has already been checked that the Haar measure has the bounded distortion property, and that the appropriate version of the KZ-cocycle is locally constant and integrable. So the proof of Theorem \ref{thm:variant} reduces to the following lemma, where we write  $A_*, B_*$ for the endomorphisms of  $H_1^{(0)}(M,\Qset)$ induced by $A,B$.
 \begin{lemma}
 Under the hypotheses  ii), iii) of the Theorem, $A_*$ and $B_*^2$ do not share a common proper invariant subspace.
 \end{lemma}
We first prove that a subspace invariant under  $B_*^2$ is invariant under $B_*$. Indeed, it is sufficient to prove
 this when the base field is $\Cset$. Then, a $B_*$-invariant subspace is the sum of its intersection with the 
 characteristic subspaces of $B_*$, and the same holds for $B_*^2$. But $B_*$ and $B_*^2$ have the same
  characteristic subspaces, because hypothesis iii) on the non existence of an irreducible even factor means 
  that, for any $\lambda \in \Cset$, $\lambda$ and $-\lambda$ cannot be both eigenvalues of $B_*$. 
  Thus it is sufficient to prove that a subspace contained in some characteristic subspace of $B_*$ 
  which is invariant under $B_*^2$ is also invariant under $B_*$. Up to a scalar factor, this is the unipotent case which was dealt with in Lemma \ref{powers}.
  \par
 To prove the lemma, it is therefore sufficient to see that $A_*$ and $B_*$ do not share a proper invariant 
 subspace. Otherwise, as we have seen in the proof of Proposition~\ref{parabolictwist}, a proper common 
 invariant subspace of minimal dimension has to be either one-dimensional, spanned by an eigenvector
 $v_{\lambda}$ of $A_*$, or two-dimensional, spanned by two eigenvectors $v_{\lambda}$, $v_{\lambda^{-1}}$.
 In both cases, considering the action of the Galois group of the characteristic polynomial of $A_*$, we conclude 
 that $H_1^{(0)}(M,\Qset)$ splits into $2$-dimensional summands $\langle v_{\lambda}, v_{\lambda^{-1}}\rangle$ 
which are invariant under $B_*$ (and $A_*$). 
 \par
These subspaces are defined over the field $\Qset(\lambda + \lambda^{-1})$. The trace and the determinant 
of the restriction of $B_*$ to these subspaces belong to this field. In view of the hypothesis of disjointness 
of the splitting fields for $A_*$ and $B_*$, the trace and the determinant must be rational. But then the minimal 
polynomial of $B_*$ has degree $\leq 2$,  in contradiction to hypothesis iii).
\end{proof}

\section{Origamis in $\mathcal H(4)$}
\label{origamisH(4)}

Simplifying the full notation $\mathcal{H}(M,\Sigma,\kappa)$, we will denote by $\mathcal H(4)$ the moduli 
space of translation surfaces of genus $3$ with a single marked point, which is a zero of order $4$ of the associated $1$-form.

\subsection{Some basic facts} \label{basicH(4)}

For origamis in $\mathcal H(4)$, it is equivalent to be primitive or reduced.

\begin{proposition}\label{primitive_reduced}
A reduced square-tiled surface in $\mathcal H(4)$ is primitive.
\end{proposition}
\begin{proof} Let $(M,\omega)$ be a  reduced square-tiled surface in $\mathcal H(4)$, and let 
$\pi: (M,\omega) \to (N,\omega')$ be a ramified covering of degree $>1$ over another square-tiled surface $(N,\omega')$. 
The Riemann-Hurwitz formula states that 
$$ 4 = 2g(M) -2 = \deg(\pi)(2g(N) -2) + \sum \textrm{ord}(c),$$
where the sum is over critical points $c$ of $\pi$ and $ \textrm{ord}(c)$ is the ramification order. As $\deg(\pi) >1$,
we have either $g(N) =2$, $\deg(\pi)=2$ and no critical point or $g(N) =1$. In the first case, as $\pi$ is unramified, 
the $1$-form $\omega$ has $2$ or $4$ zeros, contradicting the assumption that  
$(M,\omega) \in \mathcal H(4)$. In the second case, $p$ has a single critical point of order $4$. As 
$(M,\omega)$ is reduced,  $(N,\omega')$ must be the standard torus and $\pi$ must be the canonical 
covering associated to $(M,\omega)$.
\end{proof}

The following result of Zmiaikou (\cite[Theorem~3.12]{zmia}) classifies monodromy groups.

\begin{proposition} \label {monodromyH4}
The monodromy group of a primitive square-tiled surface in $\mathcal H(4)$ with $N \geq 7$ squares is equal to the full symmetric group $S_N$ or to the alternating group $A_N$.
\end{proposition}

\subsection{Connected components of $\mathcal H(4)$}\label{ccH(4)}

As a translation surface $(M,\omega)$  in $\mathcal H(4)$ has no nontrivial automorphism, it admits at most one 
anti-automorphism, which must be an involution. The unique zero $O$ of $\omega$ is a fixed point of this
 involution (when it exists). Quotienting by this involution (denoted by $ \iota$), we get a ramified double cover
  $\pi: M \to N$; the ramification points are the fixed points of $\iota$. The Riemann-Hurwitz formula relates 
  the number $l(\iota)$ of fixed points of $\iota$ to the genus of $N$  
  $$4 = 2g_M -2 = 2 (2g_N -2) + l(\iota).$$
  As $l(\iota) >0$, there are only two possibilities:  $g_N = 0, l(\iota) = 8$ and $g_N = 1, l(\iota) = 4$.
  
  \bigskip
  
  Recall   that the moduli space $\mathcal H(4)$, the simplest one which is not connected,  has two connected components (\cite{V90}, \cite{KoZo03}), 
  which are respectively called 
  the {\it hyperelliptic} and the {\it odd} component. A translation surface in  $\mathcal H(4)$ belongs to the
   hyperelliptic component iff it admits an anti-automorphism with  $8$ fixed points. On the other hand, 
   a translation surface in the odd component of $\mathcal H(4)$ may or may not admit an anti-automorphism.
   When it does, one says, following McMullen \cite{mcmullenprym}, that the translation surface belongs to the 
   Prym locus or is simply Prym.\footnote{Prym varieties are certain abelian varieties constructed from morphisms of algebraic curves. When an algebraic curve is equipped with an holomorphic involution, one may consider
   the abelian subvarieties of the Jacobian obtained by duality from the eigenforms associated to the eigenvalues $\pm 1$.}
   
   \begin{remark} 
   The $1$-form $\omega$ induces on the quotient $N$ a quadratic differential $q$
   related to $\omega$ through  $\pi^*q = \omega^2$. The quadratic differential $q$ is not the square of a
    $1$-form; it has a zero of order $3$
   at $\pi(O)$ and a simple pole at the images of the other fixed points.
   \end{remark}
   
 \subsection{Cylinders and saddle configurations}\label{configurations}

 Let $(M,\omega)$ be a translation surface. A direction $v \in \Pset(\Rset^2)$ is {\it completely periodic} if 
 every separatrix in  direction $v$ extends to a saddle connection. Then, the saddle connections (endpoints included) separate the surface into a finite number of {\it cylinders}. 
 Each cylinder is foliated by  periodic orbits of the linear flow in the $v$-direction.
 
 \smallskip
 For a square-tiled surface, the completely periodic directions are exactly the rational directions $v \in \Pset(\Qset^2)$.
 
 \smallskip
Let $v  \in \Pset(\Rset^2)$ be a completely periodic direction for the translation surface $(M,\omega)$. Denote 
by $C_1, \ldots,  C_m$  the associated cylinders and by $\gamma_1,\ldots ,  \gamma_m$ the homology classes
in $H_1(M, \Zset)$ defined by the waist curves of these cylinders (oriented in a consistent way). As these curves
do not intersect, the subspace of  $H_1(M, \Rset)$ generated by $\gamma_1,\ldots ,  \gamma_m$ is isotropic.

\smallskip

If moreover the translation surface $(M,\omega)$ has a single marked point $O$ (i.e., $\omega$ has an unique zero $O$), the homology classes 
$\gamma_1,\ldots ,  \gamma_m$ are linearly independent: indeed, $O$ belongs to each component of the boundary of each cylinder $C_i$, hence one can find a loop at $O$ which intersects once $\gamma_i$, but not the $\gamma_j, \, j\ne i$. A consequence is that the number of cylinders in this case is at most the genus of $M$.
For $(M,\omega) \in \mathcal H(4)$, one can have between one and three cylinders in completely periodic directions.

\smallskip

Consider a completely periodic direction $v$ for a  translation surface $(M,\omega) \in \mathcal H(4)$. 
After applying if necessary an element of $\SL_2(\Rset)$, we may assume that $v$ is the horizontal direction.
Index in the cyclical order the $10$ horizontal separatrices from the marked point $O$ by $\Zset_{10}$ 
(choosing an arbitrary separatrix as $S_0$). Each of the $5$  saddle connections connect two separatrices 
$S_i$ and $S_j$ such that $j-i$ is odd. The pairing of separatrices determined by the saddle connections 
is called a \emph{saddle configuration}.



\smallskip
S. Leli\`evre has determined which saddle configurations occur in each component of 
$\mathcal H(4)$, and how many cylinders correspond to each saddle configuration. 
The full list of the $16$ possible saddle configurations appears in Appendix~\ref{a.Samuel} of this 
paper (by S. Leli\`evre), and, for this reason, we will not discuss it in details in this section. Instead, we 
will only extract (in Proposition~\ref{listLelievre}) from the list the information which is relevant for us in Sections~\ref{s.odd} and \ref{s.hyperelliptic}.

\begin{definition}
A saddle-connection is {\it balanced} if it connects $S_i$ and $S_{i+5}$. 
\end{definition}

 \begin{proposition}\label{listLelievre}
 \begin{enumerate}
 \item 
 When $(M,\omega)$ belongs to the hyperelliptic component of $\mathcal H(4)$, the number $b$ of balanced
 saddle connections in a given saddle configuration and the number $c$ of cylinders generated by this saddle configuration satisfy $b + 2c =7$.
 \item 
 When $(M,\omega)$ belongs to the odd component of $\mathcal H(4)$, if a saddle configuration has at least $2$ balanced saddle connections, then it has $3$ such saddle connections and they separate $M$ into exactly $2$ cylinders.
 \end{enumerate}
 \end{proposition}
 
\subsection{The HLK-invariant}\label{HLK}   
   
   Assume now that $(M,\omega)$ is a reduced {\it square-tiled} surface in $\mathcal H(4)$, and that $\iota$ is
    an anti-automorphism of $(M,\omega)$. Denote by $\pi: M \to \Tset^2$ the  covering associated to $\omega$.
   Following the work of  E. Kani \cite{kani} and P. Hubert -S. Leli\`evre \cite{hubertlelievre} in genus two, it is
    natural to partition the fixed points of $\iota$ in the following way. The anti-automorphism $\iota$ is a lift 
    under $\pi$ of the anti-automorphism $\iota_0(z) := -z$ of $\Tset^2$. The fixed points of $\iota$ are sent by $\pi$
    to fixed points of $\iota_0$. The fixed points of $\iota_0$ are the $4$ points of order $2$ in $\Tset^2$.
    Thus, it is natural to count, for each point of order $2$ in $\Tset^2$, how many fixed points of $\iota$ sit 
    above it.
    
    \medskip
    
    The action of $\SL(2,\Zset)$ on the torus fixes the origin $0$,  and preserves the $3$-element set of points 
    of exact order $2$, acting on this set through the full symmetric group $S_3$. As we want to define 
    an invariant for the action of $\SL_2(\Zset)$ on origamis, we define the HLK-invariant $\ell(\iota)$= 
    $\ell(M,\omega)$ to be $(l_0, [l_1,l_2,l_3])$ where
    \begin{itemize}
    \item $l_0$ is the number of fixed points of $\iota$, distinct from the zero $O$ of $\omega$, sitting above 
    the origin of $\Tset^2$;
    \item $l_1, l_2, l_3$ are the number of fixed points above the 3 points of $\Tset^2$ of exact order $2$;
    they form an unordered triple that we write for convenience with the convention $l_1 \geq l_2 \geq l_3$.
    \end{itemize} 
    
    Notice that $l_0 + l_1 + l_2 + l_3$ is equal to $7$ in the hyperelliptic case and to $3$ in the Prym case.
    Another restriction is that $l_0 +1,l_1, l_2, l_3$ and the number $N$ of squares are congruent mod. $2$: 
    indeed, the fiber of $\pi$ over a point of order $2$ is preserved by the involution $\iota$ and contains 
    $N$ elements. A further restriction is given by the
    
    \begin{proposition}\label{HLKlist}
    In the hyperelliptic case, at most one of the three numbers $l_1,l_2,l_3$ is equal to zero.
    \end{proposition}

    \begin{proof}
    Assume on the contrary that at least two of these numbers are equal to $0$. By applying an appropriate
     element of $\SL_2(\Zset)$ if necessary, we may assume that there is no fixed point of $\iota$ above 
     $(0,\frac 12)$ and $(\frac 12, \frac 12)$. There are between $1$ and $3$ horizontal cylinders.
    \begin{lemma} 
    Every horizontal cylinder is fixed by $\iota$.
    \end{lemma}
 \begin{proof} Assume that there is at most one horizontal cylinder which is fixed by $\iota$. Such a cylinder
 would contain $2$ fixed points of $\iota$ in its interior. The other fixed points of $\iota$ are $O$
  (the zero of $\omega$) and the middle points of every horizontal saddle-connection which is fixed by $\iota$.
   To account for the $8$ fixed points, there must exist a  horizontal cylinder $C$ fixed by $\iota$, and 
    every  horizontal saddle-connection must be  fixed by $\iota$. Thus every  horizontal saddle-connection
    in the top boundary of $C$ is also in the bottom boundary. This means that there is only one 
    horizontal cylinder and ends the proof of the lemma.
  \end{proof}
    Each horizontal cylinder, being fixed by $\iota$, contains two fixed points of $\iota$. As these fixed points sit 
    over $(0,0)$ or $(\frac 12,0)$, the height of each horizontal cylinder is even. But then $(M,\omega)$ is not 
    reduced.
    \end{proof}
    
    \bigskip
    
    The various restrictions leave the following possibilities for the HLK-invariant:
    \begin{itemize}
    \item In the hyperelliptic case, $(4,[1,1,1]), (2,[3,1,1]), (0,[5,1,1]), (0,[3,3,1])$ for an odd number of squares, and
    $(3,[2,2,0]), (1,[4,2,0]), (1,[2,2,2])$ for an even number.
    \item In the Prym case, $(0,[1,1,1])$ for an odd number of squares, and $(1,[2,0,0])$, $(3,[0,0,0])$ for an even number.
    \end{itemize}

\subsection{A conjecture of Delecroix and Leli\`evre} \label{DLconj}

On the basis of computer experiments (with Sage), V. Delecroix and S. Leli\`evre have formulated the 
following conjecture.

\begin{conjecture} \label{conj:comps-h4}

For $N>8$, the number of $\SL_2(\Zset)$-orbits of primitive $N$-square origamis in 
$\mathcal{H}(4)$ is as follows:

\begin{itemize}
\item  there are precisely two such $\SL_2(\Zset)$-orbits  
in the odd component of $\mathcal H(4)$ outside of the Prym locus, distinguished by their monodromy group 
being $A_N$ or $S_N$;

\item  for  odd $N$,  there are precisely four such $\SL_2(\Zset)$-orbits in the hyperelliptic component of $\mathcal H(4)$, distinguished by their HLK-invariant being 
 $(4,[1,1,1])$, $(2,[3,1,1])$, $(0,[5,1,1])$ or  $(0,[3,3,1])$;

\item for  even $N$,  there are precisely three such $\SL_2(\Zset)$-orbits in the hyperelliptic component of $\mathcal H(4)$, distinguished by the HLK-invariant being 
 $(3,[2,2,0])$, $(1,[4,2,0])$ or $(1,[2,2,2])$.
 
 \end{itemize}
 
 \end{conjecture}
 
\subsection{Prym covers} \label{ss:Prym}

We don't have any new result for origamis of Prym type in $\mathcal H(4)$. For the sake of completeness, we recall a couple of important facts.

\smallskip 
 
Let $(M, \omega)$  be a square-tiled surface of Prym type in $\mathcal H(4)$. Denote by $\iota$ the anti-automorphism 
of $(M,\omega)$. It commutes with any affine homeomorphism of $(M,\omega)$. It acts as an involution on the 
space $H_1^{(0)}(M, \Qset)$.  The eigenspaces $H_\pm$ associated with the eigenvalues $\pm1$ both have 
dimension $2$, and the splitting
       $$ H_1^{(0)}(M, \Qset) = H_+ \oplus H_-$$
  into orthogonal   symplectic subspaces   is invariant under the affine group of $(M,\omega)$. 
  It means that it is also invariant under the KZ-cocycle. Denote by $\pm \lambda_+$ (resp. $\pm \lambda_-$) 
  the Lyapunov exponents of the restriction of the KZ cocycle to $H_+$ (resp. $H_-$), 
  with $\lambda_{\pm} \geq 0$.  
  \par 
  Chen and M\"oller have shown \cite{chenmoeller}  that $\lambda_+ + \lambda_- = \frac 35$, i.e the sum of the nontrivial non-negative exponents is the same as for the Masur-Veech measure of the odd component. On the other hand, 
 Eskin, Kontsevich and Zorich have proved (\cite[Theorem~2]{ekz}) that \linebreak $\lambda_- - \lambda_+ = \frac 15$.
 
  
  Put together, these results prove that the nontrivial exponents for an origami in $\mathcal H(4)$ of Prym type
  are $\pm \frac 15$, $\pm \frac 25$.
  
  Regarding the classification of $\SL_2(\Zset)$-orbits of primitive origamis of Prym type, this was settled 
   by E.~Lanneau and D.-M.~Nguyen~\cite{manhlann} before the formulation of the conjecture of Delecroix-Leli\`evre.

 \begin{theorem}
For $N>8$, the number of  $\SL_2(\Zset)$-orbits of primitive $N$-square origamis in the Prym locus of 
$\mathcal{H}(4)$  is as follows: 

\begin{itemize}

\item for odd $N$, there is precisely one $\SL_2(\Zset)$-orbit; the HLK-invariant is  $(0,[1,1,1])$;

\item For $N\equiv 0 \textrm{ mod } 4$, there is precisely one $\SL_2(\Zset)$-orbit;  
its HLK- invariant is equal to $(1,[2,0,0])$;
\item If $N\equiv 2 \textrm{ mod } 4$, there are precisely 
two $\SL_2(\Zset)$-orbits, distinguished by their  HLK-invariant being
$(1,[2,0,0])$ or $(3,[0,0,0])$.
\end{itemize}
 \end{theorem}

\subsection{Splitting fields of monic quartic reciprocal polynomials}\label{Galoistheory}

In the next two sections, we will construct certain affine maps of origamis in $\mathcal H(4)$ which  
potentially act on $H_1^{(0)}$ as Galois-pinching matrices. We present now some elementary 
Galois theory which is relevant to this question.

\medskip

Let   $P(x) = x^4 + a x^3 + b x^2 + a x + 1 \in \Qset[x]$ be a reciprocal polynomial of degree $4$. 
We define
$$ t:= -a-4, \quad d:= b+2a +2,$$
so that $\lambda$ is a zero of $P$ iff $\mu:= \lambda + \lambda^{-1} -2$ is a zero of $Q(y) := y^2 -ty +d$.
The discriminant of $Q$ is\footnote{We use the substitution  $\mu:= \lambda + \lambda^{-1} -2$ rather than $\mu:= \lambda + \lambda^{-1} $ because it appears naturally later.}
$$ \Delta_1 := t^2 -4d = a^2 -4b +8.$$

The polynomial $Q$ is reducible over $\Qset$ iff $\Delta_1$ is a rational square. It follows that

\begin{lemma} \label{delta1}
The polynomial $P$ is the product of two reciprocal polynomial of  degree $2$ with rational coefficients  iff $\Delta_1$ is a rational square.
\end{lemma}

From now on, we assume that $\Delta_1$ is {\it not} a rational square.
Denote by $\mu_1, \mu_2$ the roots of $Q$ and by $\lambda_1^{\pm1}, \lambda_2^{\pm1}$ those of $P$, with
$\mu_i = \lambda_i + \lambda_i^{-1} -2$. As $Q$ is irreducible,  $\mu_1, \mu_2$ are not rational, hence $\lambda_1,\lambda_2$ are also not rational.

Define 

\begin{eqnarray*}
 \Delta_2&:=& (\lambda_1 -\lambda_1^{-1})^2  (\lambda_2 -\lambda_2^{-1})^2 \\
                &=& \mu_1 \mu_2 (\mu_1 +4)(\mu_2 +4) \\
                &=& d(d+4t+16) = (b+2+2a)(b+2-2a).
  \end{eqnarray*}

\begin{lemma}\label{Pirreducible}
If $P$ is not irreducible over $\Qset$ (and $\Delta_1$ is not a rational square), then $\Delta_2$ is a rational square. 
\end{lemma}

\begin{proof} As $\Delta_1$ is not a rational square, $P$ has no rational root. Assume that $P$ is reducible over $\Qset$. Let $P =P' P''$ be its decomposition into  monic irreducible polynomials. The polynomials $P'$, $P''$ have degree $2$, rational coefficients  and are not reciprocal, hence we may rename the roots in order to have 
$$ P'(x) = (x-\lambda_1)(x-\lambda_2), \quad P''(x) =   (x-\lambda_1^{-1})(x-\lambda_2^{-1}).$$

Since  $P'$ has rational coefficients, $\lambda_1\lambda_2 \in \Qset$, $\lambda_1+\lambda_2 \in \Qset$, 
hence $\lambda_1^2 + \lambda_2^2 \in \Qset$ and consequently
$$  (\lambda_1 -\lambda_1^{-1}) (\lambda_2 -\lambda_2^{-1}) = \lambda_1 \lambda_2 + \frac {1-\lambda_1^2 - \lambda_2^2}{\lambda_1 \lambda_2}$$
is a rational number. Therefore $\Delta_2$ is a rational square.
\end{proof}

From now on, we assume that $P$ is irreducible over $\Qset$.
Denote by $\textrm{Gal}$ the Galois group of $P$, by 
$\rho: \textrm{Gal} \to S_2$ the surjective homomorphism corresponding to the action of  $\textrm{Gal}$ on 
$\{\mu_1,\mu_2\}$ and by $\textrm{Gal}_0$ the kernel of $\rho$. As $\textrm{Gal}$ acts transitively on the roots 
of $P$, there are two possibilities:
\begin{enumerate}
\item $ \textrm{Gal}_0 \simeq \Zset_2 \times \Zset_2$ has order $4$, allowing independent switches of 
$\lambda_i$ and $ \lambda_i^{-1}$, $i=1,2$.
\item $ \textrm{Gal}_0 \simeq \Zset_2$ has order $2$, the nontrivial element switching simultaneously 
$\lambda_1, \lambda_1^{-1}$ and $\lambda_2, \lambda_2^{-1}$.
\end{enumerate}
In case (1), the Galois group has order $8$, is largest possible and is the centralizer of the permutation
$(\lambda_1, \lambda_1^{-1})(\lambda_2, \lambda_2^{-1})$. In case (2), the Galois group has order $4$
 and there are 2 subcases:
 
 \smallskip
 
 \hspace{5mm} (2a) $\textrm{Gal} \simeq \Zset_2 \times \Zset_2$ is generated by  $(\lambda_1,\lambda_2) ( \lambda_1^{-1}, \lambda_2^{-1})$ and $(\lambda_1, \lambda_1^{-1})(\lambda_2, \lambda_2^{-1})$.

\smallskip

\hspace{5mm} (2b) $\textrm{Gal} \simeq \Zset_4$ is generated by the $4$-cycle $(\lambda_1,\lambda_2, \lambda_1^{-1}, \lambda_2^{-1})$.

\medskip

\begin{lemma}\label{case2a}
Case (2a) occurs iff $\Delta_2$ is a square.
\end{lemma}
\begin{proof}
If case (2a) occurs, then $ (\lambda_1 -\lambda_1^{-1})  (\lambda_2 -\lambda_2^{-1})$ is invariant under the 
Galois group, hence is rational, and $\Delta_2 $ is a square. Conversely, assume that 
$\Delta_2= \delta_2^2$ is a square.
One has $\lambda_i^{\pm1} = \frac 12 ( \mu_i +2 \pm \sqrt{\mu_i(\mu_i +4)})$. The splitting field of $P$ is the quadratic extension of $\Qset(\sqrt {\Delta_1}) = \Qset(\mu_1) = \Qset(\mu_2)$ generated by a square root
of $\mu_1(\mu_1 +4) $ because

$$\sqrt{\mu_2(\mu_2 +4) } = \frac {\delta_2}{\sqrt{\mu_1(\mu_1 +4)}}.$$

This extension has a Galois group isomorphic to $\Zset_2 \times \Zset_2$. 
\end{proof}

\begin{lemma}\label{case2b}
Case (2b) occurs iff $\Delta_1 \Delta_2$ is a square.
\end{lemma}
\begin{proof}
Observe that $\Delta_1 \Delta_2 = \delta^2$, with 
$$\delta = (\lambda_1 + \lambda_1^{-1} - \lambda_2  -\lambda_2^{-1}) (\lambda_1 -\lambda_1^{-1})  (\lambda_2 -\lambda_2^{-1}).$$
If case (2b) occurs, then $\delta$ is invariant under the Galois group, hence is rational, and $\Delta_1 \Delta_2$ is a square.  Conversely, assume that $\Delta_1 \Delta_2$ is a square. The splitting field of $P$ is still the quadratic extension of   $\Qset(\sqrt {\Delta_1})$ generated by  a square root
of $\mu_1(\mu_1 +4)$, because now 
$$\sqrt {\mu_2(\mu_2 +4)} = \frac{\delta}{\sqrt{\Delta_1}\sqrt {\mu_1(\mu_1 +4)}}.$$
Thus the Galois group has order $4$. As $\Delta_1$ is not a square, $\Delta_2$ is also not a square. By Lemma \ref{case2a}, case (2a) does not occur; hence case (2b) must occur. 
\end{proof}

We can now conclude:

\begin{proposition}\label{case1}
The Galois group has maximal order $8$ iff none of $\Delta_1$, $\Delta_2$, $\Delta_1 \Delta_2$ is a square.
\end {proposition}

\begin{remark}\label{quadraticsubfields}
In case (1), the Galois group of order $8$ contains $3$ subgroups of index $2$, namely  $ \textrm{Gal}_0$ and the Galois groups of case (2a) and (2b). The three quadratic fields contained in  the splitting field of $P$ which correspond to these subgroups are easily seen to be $\Qset(\sqrt {\Delta_1})$, $\Qset(\sqrt {\Delta_2})$, and $\Qset(\sqrt {\Delta_1\Delta_2})$ respectively.
\end{remark}

Recall that a Galois-pinching matrix must have all its eigenvalues real. For a  reciprocal polynomial of degree $4$ (with real coefficients), we have the following observation

\begin{proposition}\label{positiveroots}
The roots of $P$ are simple, real and positive iff $\Delta_1>0$, $d>0$ and $t>0$.
\end{proposition}
\begin{proof}
 Recall that $\lambda$ is a root of $P$ iff $\mu = \lambda + \lambda^{-1} -2$ is a root of $Q(y) := y^2 -ty+d$. The roots $\lambda$ are simple, real and positive iff the roots $\mu$ have the same property.
The conclusion of the proposition follows.
\end{proof}

In Sections \ref{s.odd} and \ref{s.hyperelliptic}, we will consider one-parameter families 
$(M_n,\omega_n)$ of origamis in $\mathcal H(4)$ indexed by an integer $n$. For each origami in the family,
we will construct an element $A_n$ of the affine group $\textrm{Aff}(\omega_n)$. The characteristic polynomial
of the action of $A_n$ on $H_1^{(0)}(M_n,\Qset)$ is a monic reciprocal polynomial $P_n$ of degree $4$ with integer coefficients.

\smallskip

For each integer $n$, we will compute the quantities $\Delta_1(P_n),\, \Delta_2(P_n)$ from the last subsection.
The constructions of $M_n, \omega_n, A_n$ are such that both $\Delta_1(P_n)$ and $ \Delta_2(P_n)$ will be 
{\it polynomial functions of the variable $n$ with integer coefficients}. In order to claim that neither  $\Delta_1(P_n)$ nor 
$ \Delta_2(P_n)$ nor the product $\Delta_1(P_n) \Delta_2(P_n)$ are squares for all but finitely many $n$, 
we will appeal to Siegel's theorem on the finiteness of integral points on curves of genus $>0$. In order to 
do this, we introduce the following definition.

\begin{definition} 
Let $\Delta \in \Zset[x]$ be a polynomial. Write $\Delta = \delta^2 \Delta^\red$ with 
$\delta, \Delta^\red \in \Zset[x]$ and  $\Delta^\red$ square-free. The {\it reduced degree} of $\Delta$, denoted by
$\deg^\red(\Delta)$ is the degree of $\Delta^\red$.
\end{definition} 

The following result is then a special case of Siegel's theorem.

\begin{proposition}\label{Siegel}
If $\deg^\red(\Delta) \geq 3$, there are only finitely many values of the integer $n$ such that $\Delta(n)$ is a square.
\end{proposition}
\begin{proof} Indeed each such value of $n$ gives an integral point on the curve $y^2 =  \Delta^\red (x)$, which
 is nonsingular (except for a possible double point at $\infty$) of genus $>0$.
 \end{proof}

\section{The odd case} \label{s.odd}

\subsection {A model for odd origamis in $\mathcal H(4)$} \label{oddmodel}

Consider the origami $\mathcal O$ constructed as indicated on Figure~\ref{f.H4-odd} below. It 
depends on $6$ parameters $H_1,H_2,H_3,V_1,V_2,V_3$ which are positive integers. 

\begin{figure}[htb!]
\includegraphics[scale=1]{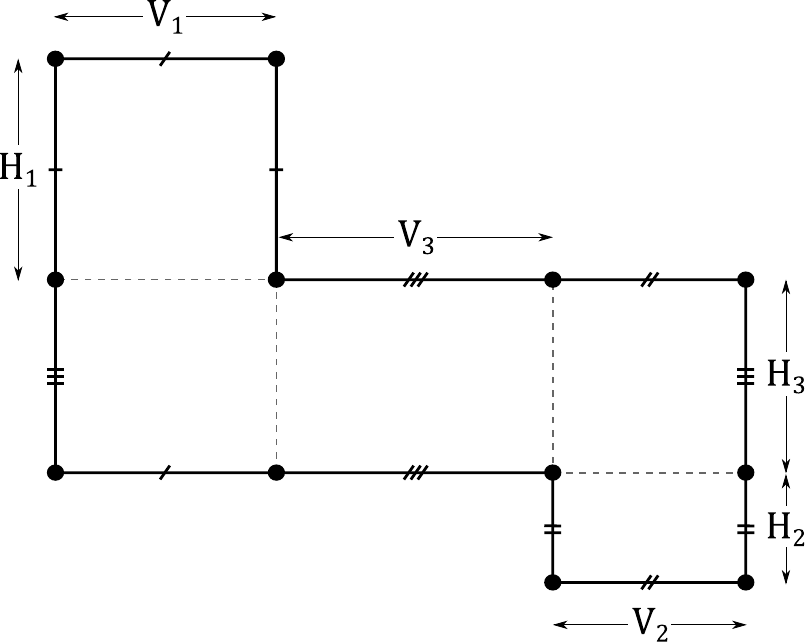}
\caption{A origami in the odd component of $\mathcal{H}(4)$.}\label{f.H4-odd}
\end{figure}

\smallskip

In the horizontal direction, there are 3 cylinders $Ch_1,Ch_2,Ch_3$. The height of these cylinders 
are respectively $H_1,H_2,H_3$, while the length of the waist curves are respectively $\ell (Ch_1)=V_1, \ell (Ch_2)=V_2,\ell (Ch_3)=V_1+V_2+V_3$.

\smallskip

In the vertical direction, there are also 3 cylinders $Cv_1, Cv_2, Cv_3$. The height of these cylinders 
are respectively $V_1,V_2,V_3$, while the length of the waist curves are respectively 
$\ell (Cv_1)=H_1+H_3, \ell (Cv_2)=H_2+H_3,\ell (Cv_3)=H_3$.

\smallskip

We will denote by $\sigma_i$, $i=1,2,3$, the homology class of the waist curve of $Ch_i$, oriented rightwards,
and by $\zeta_i$ , $i=1,2,3$, the homology class of the waist curve of $Cv_i$, oriented upwards. The symplectic
intersection form $\iota$ on $H_1 (\mathcal O , \Zset)$ satisfies
$$ \iota (\sigma_i,\sigma_j) = \iota (\zeta_i, \zeta_j) =0, \quad \iota(\sigma_i,\zeta_j) = I_{ij},$$
with
$$ I = \left(\begin{array}{ccc}1 &  0 &0  \\ 0& 1 &0 \\ 1 &1 &1 \end{array}\right) .$$

\begin{remark} The intersection matrix $I$ above has maximal rank (equal to $3$). By Forni's criterion \cite{Fo3}, it follows that all Lyapunov exponents of any origami $\mathcal{O}$ of the form indicated in Figure~\ref{f.H4-odd} are non-zero. 
\end{remark}

We review some elementary properties of $\mathcal O$:

\begin{itemize}
\item $ \mathcal O$ has a single conical singularity of total angle $10 \pi$, so it belongs to 
$\mathcal H (4)$.
\item The number of squares is 
$$  N(\mathcal O) = H_1 V_1 + H_2 V_2 + H_3 (V_1 +V_2 +V_3).$$
\item  $\mathcal O$ is reduced if and only if  ${\rm gcd}(H_1,H_2,H_3) = {\rm gcd}(V_1,V_2,V_3) =1$: indeed,
 the condition is clearly necessary. It  is also sufficient because the vectors $(V_i,0)$ and $(0,H_i)$ are 
 periods for $i=1,2,3$.
 \item The classes $\sigma_1, \sigma_2, \sigma_3, \zeta_1, \zeta_2, \zeta_3$ form a basis of the 
integral homology $H_1 (\mathcal O , \Zset)$: indeed, the matrix $I$ above is invertible over $\Zset$.
\item The origami $\mathcal O$ admits an affine involution with derivative $-\textrm{Id}$  iff $H_1=H_2$ and $V_1=V_2$. Indeed, such an involution must send an horizontal cylinder to an horizontal cylinder, and similarly for vertical cylinders. Therefore, it must preserve $Ch_3$ (the horizontal cylinder with the longest waist curve) , $Cv_3$ (the vertical cylinder with the shortest waist curve) and the rectangular intersection $Ch_3 \cap Cv_3$. As the derivative is $-\textrm{Id}$, it must exchange $Ch_1$ and $Ch_2$, and also $Cv_1$ and $Cv_2$. This forces $H_1=H_2$ and $V_1=V_2$. Conversely, if these equalities hold, the central symmetry preserving $Ch_3 \cap Cv_3$ defines the required involution.
\item $\mathcal O$ belongs to the odd component of $\mathcal H(4)$. When the parameters 
$H_i,V_i$ vary among positive real numbers we get a family of translation surfaces in a single component of $\mathcal H(4)$. It cannot be the hyperelliptic component because most of these surfaces do not have the required affine involution.
\end{itemize}

We also recall that an origami in $\mathcal H (4)$ is primitive iff it is reduced (Proposition \ref{primitive_reduced}).

\begin{proposition}\label{monodromy_odd}
Assume that $\mathcal O$ is primitive with $N= N(\mathcal O) \geq 7$ . Then the monodromy group is equal to
 $A_N$ if $N, H_1+H_2+H_3$ and $V_1+V_2+V_3$ have the same parity, and it is equal to $S_N$ otherwise.
\end{proposition}

\begin{proof} 
Indeed, by the theorem of Zmiaikou mentioned earlier (cf. Proposition \ref{monodromyH4}), the monodromy group is equal to $A_N$ or $S_N$. The signature of a permutation of $N$ elements with $c$ cycles is $(-1)^{N-c}$. For the permutations generating the monodromy group, the number of cycles  is $H_1+H_2+H_3$ in the horizontal direction, $V_1+V_2+V_3$ in the vertical direction. The assertion of the proposition is now clear.
\end{proof}

The map $\pi_*: H_1 (\mathcal O , \Zset) \to H_1(\Tset^2,\Zset) = \Zset^2$ induced by the canonical covering
$\pi: \mathcal O \to \Tset^2$ is given by

$$ \pi_*(\sigma_i) = (\ell(Ch_i),0), \quad \pi_*(\zeta_i) = (0, \ell(Cv_i)), \quad i=1,2,3.$$

Define, for $i = 1,2$
$$ \Sigma_i := \ell(Ch_3) \sigma_i - \ell(Ch_i) \sigma_3, \quad Z_i = \ell(Cv_3) \zeta_i - \ell(Cv_i) \zeta_3.$$ 

Then $\Sigma_1, \Sigma_2, Z_1,Z_2$ are elements of $H_1^{(0)}(\mathcal O, \Zset)$ which span a subgroup 
of finite index of this group; they form a basis of  $H_1^{(0)}(\mathcal O, \Qset)$.

\subsection {Two parabolic elements in $\textrm {Aff}(\mathcal O)$}\label{oddparabolic}

Set 

$$ L_h := \ell(Ch_1)  \ell(Ch_2)  \ell(Ch_3), \quad P_h :=  \left(\begin{array}{cc}1 &  L_h  \\ 0& 1  \end{array}\right) .$$

The matrix $P_h$ belongs to the Veech group $\SL(\mathcal O)$. Indeed, the associated element $p_h$ 
of the affine group acts on homology according to $p_h.\sigma_i= \sigma_i$ and 

$$ p_h. \zeta_i  = \zeta_i + \sum_{j=1}^3 I_{ji} H_j \frac {L_h}{\ell(Ch_j)} \sigma_j , \quad i=1,2,3.$$

We deduce from this formula the action on $H_1^{(0)}(\mathcal O, \Qset)$. We have $p_h. \Sigma_i = \Sigma_i$ for $i=1,2$ and
$$   p_h. Z_i  = Z_i + \sum_{j=1}^3 H_{ij} H_j \frac {L_h}{\ell(Ch_j)} \sigma_j ,\quad i=1,2.$$
with 
$$ H_{ij}:= \ell(Cv_3) I_{ji} - \ell(Cv_i) I_{j3}.$$
The last formula can be rewritten as 
$$   p_h. Z_i  = Z_i + \sum_{j=1}^2 H_{ij} H_j \ell(Ch_{3-j}) \Sigma_j ,\quad i=1,2.$$
Substituting the values of $I_{ji}$, $\ell(Ch_j)$, $ \ell(Cv_j)$, we get
$$ H_{11} = H_{22} = H_3, \quad  H_{12} = H_{21} =0,$$
\begin{eqnarray*}
 p_h. Z_1  &=& Z_1 +H_1 H_3 V_2 \Sigma_1, \\
 p_h. Z_2  &=& Z_2 +H_2 H_3 V_1 \Sigma_2.
 \end{eqnarray*}
Turning to the vertical direction, we set
$$ L_v := \ell(Cv_1)  \ell(Cv_2)  \ell(Cv_3), \quad P_v :=  \left(\begin{array}{cc}1 & 0  \\  L_v & 1  \end{array}\right) .$$
The matrix $P_v$ belongs to the Veech group $\SL(\mathcal O)$. The associated element $p_v$ 
of the affine group acts on homology according to $p_v.\zeta_i= \zeta_i$ and 
$$ p_h. \sigma_i  = \sigma_i + \sum_{j=1}^3 I_{ij} V_j \frac {L_v}{\ell(Cv_j)} \zeta_j , \quad i=1,2,3.$$
For the action on $H_1^{(0)}(\mathcal O, \Qset)$, we have $p_v. Z_i = Z_i$ for $i=1,2$ and
\begin{eqnarray*}
   p_v. \Sigma_i & = &\Sigma_i + \sum_{j=1}^3 V_{ij} V_j \frac {L_v}{\ell(Cv_j)} \zeta_j \\
                            & = &\Sigma_i + \sum_{j=1}^2 V_{ij} V_j  \ell(Cv_{3-j}) Z_j,
 \end{eqnarray*}
with 
$$ V_{ij}:= \ell(Ch_3) I_{ij} - \ell(Ch_i) I_{3j}.$$
\smallskip
Substituting the values of $I_{ij}$, $\ell(Ch_j)$, $ \ell(Cv_j)$, we get
$$ V_{11} = V_2 +V_3, \quad V_{22} = V_1 +V_3, \quad V_{12} = -V_1, \quad V_{21} = -V_2,$$
\begin{eqnarray*}
 p_v. \Sigma_1  &=& \Sigma_1 + V_1 (V_2+V_3)(H_2+H_3)Z_1 - V_1V_2(H_1+H_3)Z_2, \\
 p_v. \Sigma_2  &=& \Sigma_2 -V_1V_2(H_2+H_3) Z_1 + V_2(V_1+V_3)(H_1+H_3)Z_2.
 \end{eqnarray*}
\par 
\smallskip
 We define the shorthand notation $H_{13} = H_1 + H_3$, $H_{23} = H_2 + H_3$ and 
 $$ Q_h :=  \left(\begin{array}{cc} H_1 H_3 V_2 &   0\\0 &  H_2 H_3 V_1 \end{array}\right), \quad
 Q_v :=  \left(\begin{array}{cc}  V_1 (V_2+V_3)H_{23} &    -V_1V_2H_{23} \\ - V_1V_2H_{13}&  V_2(V_1+V_3)H_{13} 
\end{array}\right),$$ 
 so that the matrices of $p_h,p_v$ in the basis $\Sigma_1, 
\Sigma_2, Z_1,Z_2$ are respectively 
 $$  \left(\begin{array}{cc}1 &  Q_h  \\ 0& 1  \end{array}\right), \quad  \left(\begin{array}{cc}1 & 0  \\  Q_v & 1  \end{array}\right) .$$
We will investigate whether  $A:= p_v \circ p_h$ is Galois-pinching.

\subsection {Eigenvalues and eigenvectors for $Q_h. Q_v$, $Q_v.Q_h$ and $A$} \label{eigenodd}

A vector
$$w= x_1 \Sigma_1 + x_2 \Sigma_2 + y_1 Z_1 +y_2 Z_2$$
 is eigenvector of $A$ associated 
to the eigenvalue $\lambda$ iff  $x:=(x_1,x_2), y:=(y_1,y_2)$ satisfy
$$ x = \frac 1{\lambda -1} Q_h.y, \quad y = \frac  {\lambda}{\lambda -1} Q_v.x.$$
Then, $x$ and $y$ are eigenvectors of $Q_h. Q_v$, $Q_v.Q_h$  respectively, associated to the same eigenvalue
$ \mu := \lambda + \lambda^{-1} -2$.
Let
\begin{eqnarray*}
 d_h & := & \det Q_h = H_1 H_2 H_3^2 V_1 V_2 =:  H_3^2 V_1 V_2 \bar d_h, \\
 d_v &:=& \det Q_v = V_1 V_2 V_3 (V_1+V_2+V_3) (H_1+H_3)(H_2+H_3) =: V_1 V_2 \bar d_v,\\
 d&:=& d_h d_v,  \quad \bar d := \bar d_h \bar d_v ,\\
 t&:=& \textrm{tr} (Q_h.Q_v) = \textrm{tr} (Q_v.Q_h) \\
 & = & V_1 V_2 H_3 [H_1 H_2(V_1+V_2 +2 V_3) + H_1 H_3 (V_2+V_3) + H_2 H_3 (V_1 +V_3)]\\
 & =:& V_1 V_2 H_3 \bar t.
 \end{eqnarray*}
The eigenvalues of $Q_h. Q_v$ are the solutions of $\mu^2 - t\mu + d = 0$ with discriminant
$$ \Delta_1 := t^2 - 4d = V_1^2 V_2^2 H_3^2 (\bar t^2 - 4 \bar d_h \bar d_v).$$
Thus, we get 
\begin{eqnarray*}
\bar \Delta_1 &:=& \bar t^2 - 4 \bar d_h \bar d_v\\
        &=& [H_1 H_2(V_1+V_2 +2 V_3) + H_1 H_3 (V_2+V_3) + H_2 H_3 (V_1 +V_3)]^2\\
        &  &  - 4 H_1 H_2 (H_1 +H_3)(H_2 +H_3) V_3(V_1 +V_2 +V_3).
 \end{eqnarray*}
With Proposition~\ref{case1} in mind, we also define
$$ \Delta_2 := d(d+4t +16),\quad \bar\Delta_2:=\bar d (d+4t +16).$$

\subsection{One-parameter subfamilies}\label{subfalmiliesodd}

Until now, the parameters $H_i,V_j$ have only be constrained by the condition 
${\rm gcd}(H_1,H_2,H_3) = {\rm gcd}(V_1,V_2,V_3) =1 $. 
\par
We now restrict our attention to nine one-parameter subfamilies  which will provide enough origamis to prove 
Theorem~\ref{thm:introH4}. In each subfamily, the values of $H_1,H_2,H_3,V_1,V_2$
are fixed and $V_3$ runs along an arithmetic progression.
\par
In each of the nine one-parameter subfamilies, one has $V_1 = 1$, $V_2 = 2$, $H_3 =1$. 
The values of $H_1,H_2$ for the nine families are $H_1 =1, 2  \leq H_2 \leq 4$ and $H_1 = 2$,
$3 \leq H_2 \leq 8$. Finally, we write $V_3 = 3n$ when $H_1 =1$, $V_3 = 6n $ when $H_1 =2$ and $H_2$ is even, $V_3 = 6n+3 $ when $H_1 =2$ and $H_2$ is odd.

It is clear that each origami in these families is reduced hence primitive.

\begin{proposition}\label{directionodd}
For each origami in these nine families, the direction $(3,1)$ is a $2$-cylinder direction, and therefore has homological dimension  $2$.
\end{proposition}
\begin{proof}
Index  the separatrices in the direction $(3,1)$ in cyclical order by $\Zset_{10}$ as indicated in Figure~\ref{f.H4-odd-sc} below. 

\begin{figure}[htb!]
\includegraphics[scale=1]{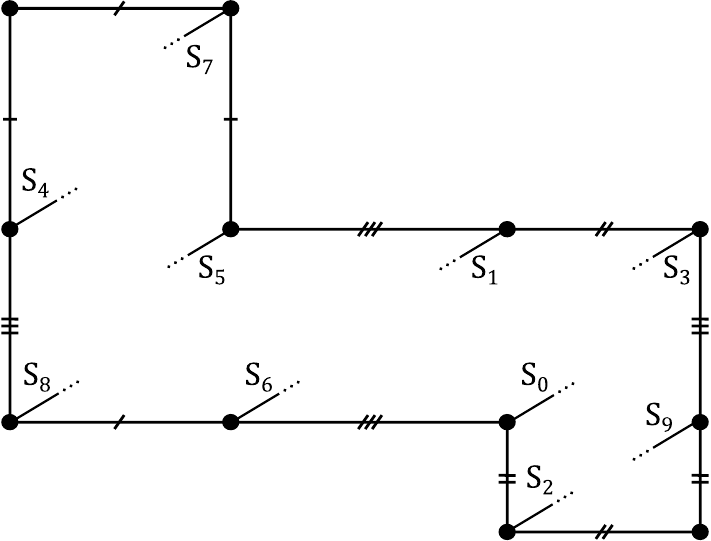}
\caption{Separatrices in the direction $(3,1)$ of an origami in the odd component of $\mathcal{H}(4)$.}\label{f.H4-odd-sc}
\end{figure}
The saddle-connections in the $(3,1)$ direction join $S_0$ to $S_5$, $S_6$ to $S_1$, $S_8$ to $S_3$. By 
Proposition~\ref{listLelievre}, it follows that $(3,1)$ is a $2$-cylinder direction.
\end{proof}

The number of squares is 
$$ N(\mathcal O) = 3 + H_1 + 2H_2 + V_3.$$
Next, we compute the monodromy group. We have 
\begin{eqnarray*}
H_1 +H_2 +H_3 &=& 1+H_1+H_2,\\
V_1+V_2+V_3 &=& 3 + V_3,
\end{eqnarray*}
From Proposition \ref{monodromy_odd}, we obtain
\begin{proposition}\label{monodromy}
The monodromy group of an origami in these subfamilies is the full symmetric group if $H_1 =1$, the alternating group if $H_1 =2$.
\end{proposition}
\begin{proof}
When $H_1 =1$, $N$ and $V_1 +V_2 +V_3$ do not have the same parity. On the other hand, when $H_1 =2$, 
$H_2$ and $V_3$ have the same parity, hence $N$, $H_1 + H_2 +H_3$ and $V_1 +V_2 +V_3$ have the
same parity.
\end{proof}

We plan to apply the elementary Galois theory of Subsection \ref{Galoistheory} in order to prove that, when $n$
 is large enough, the affine homeomorphism $A$ constructed in subsection \ref{oddparabolic} is Galois-pinching. Specializing the formulas of the last subsection gives
 
 \begin{eqnarray*}
 d = 4\bar d &=& 4 H_1 H_2 (H_1+1)(H_2 +1) V_3(V_3 +3)\\
 t = 2 \bar t &=& 2 [ V_3 (2H_1 H_2 +H_1 +H_2) +3H_1 H_2 +2H_1 +H_2 ] \\
 \Delta_1 =  4\bar \Delta_1 &=& 4(\bar t^2 - 4 \bar d) \\
     \bar \Delta_1 &=& [(H_2 -H_1) V_3 -( H_1H_2 +2H_1 -H_2)]^2 + \delta (H_1, H_2)       
  \end{eqnarray*}
with 
$$  \delta (H_1, H_2) = (3H_1 H_2 +2H_1 +H_2)^2 - ( H_1H_2 +2H_1 -H_2)^2 >0.$$

We conclude that, for $V_3$ large enough, $ \Delta_1$ is not a square. 
\par 
We also notice that, for all origamis in each of the nine one-parameter families, the quantities $d,t,\Delta_1$
are positive, hence the eigenvalues of $A$ are simple, real and positive (Proposition \ref{positiveroots}).
\par
Consider next 
  $$\Delta_2 = d(d+4t +16) = 16\bar d (\bar d + 2\bar t + 4).$$
This is a polynomial of degree $4$ in $V_3$ with integer coefficients. We will check that its reduced degree  (see Subsection 
\ref{Galoistheory}) is $4$, which allows to apply Proposition \ref{Siegel}. The roots of $\bar d$, as a 
polynomial in $V_3$, are $0$ and $-3$. For $V_3 = 0$, one has
$$ \bar d + 2 \bar t +4 = 2( 3H_1 H_2 +2H_1 +H_2) +4 >0.$$
For $V_3 = -3$, we have 
$$ \bar d + 2 \bar t +4 = 4 -2 (3H_1 H_2 +H_1 +2H_2) <0.$$
It remains to check that the degree two polynomial $ \bar d + 2 \bar t +4 =: a_2 V_3^2 + a_1 V_3 + a_0 $ 
has simple roots. Actually, for  each of the nine subfamilies, the coefficients $a_i$ are positive with $a_1 > 2 \max (a_2, a_0)$, hence the discriminant is positive. We conclude that the reduced degree of $\Delta_2$ is equal to $4$.
\par
Finally, we claim  that the reduced degree of the degree six polynomial $\Delta_1 \Delta_2$ is $6$. Indeed 
the formula above for $\bar \Delta_1$ shows that it has no real roots, while we have seen that the roots of 
$\Delta_2$ are real and simple.  Applying Propositions~\ref{Siegel} and~\ref{case1}, we get the following.

\begin{proposition}\label{Galodd}
In each of the nine subfamilies, if  $V_3$ is large enough, the affine homeomorphism $A$ constructed in 
Subsection~\ref{oddparabolic} is Galois-pinching.
\end{proposition}

\subsection {Conclusion for the odd case}\label{conclusionodd}

We now have all the ingredients to prove Theorem \ref{thm:introH4} for origamis of odd type. Observe that for
the origamis in the nine one-parameter subfamilies $V_1 \ne V_2$, hence they are not of Prym type.
By Propositions \ref{no_automorphism}, \ref{directionodd} and \ref{Galodd}, they satisfy for $V_3$ large enough the hypotheses of 
Corollary \ref {cor:mainsimpcrit}, hence their Lyapunov spectra are simple. 
\par
In the three subfamilies with $H_1 =1$, the monodromy group is the full symmetric group (Proposition \ref{monodromy}). The number of squares in this case is
$$  N(\mathcal O) = 4 + 2H_2 + 3n$$
so the $3$ choices for $H_2$ allow to get any large number of squares. 
\par
In the six subfamilies with $H_1 =2$, the monodromy group is the alternating group (Proposition \ref{monodromy}). The number of squares in this case is
\begin{eqnarray*}
  N(\mathcal O) & =& 5 + 2H_2 + 6n \quad \textrm { if  $H_2$   is even} \\
                          &=& 8 + 2H_2 + 6n \quad \textrm {if  $H_2$   is odd}
  \end{eqnarray*}                        
 so the $6$ choices for $H_2$ allow to get any large number of squares. 
The proof of Theorem~\ref{thm:introH4} is complete in the odd case. \hfill $\Box$

\section{The hyperelliptic case} \label{s.hyperelliptic}

\subsection {A model for hyperelliptic origamis in $\mathcal H(4)$} \label{hyperellipticmodel}

Consider the origami $\mathcal O$ constructed as indicated on Figure~\ref{f.H4-hyp} below. It 
depends on $6$ parameters $H_1,H_2,H_3,V_1,V_2,V_3$ which are positive integers.

\begin{figure}[htb!]
\includegraphics[scale=1]{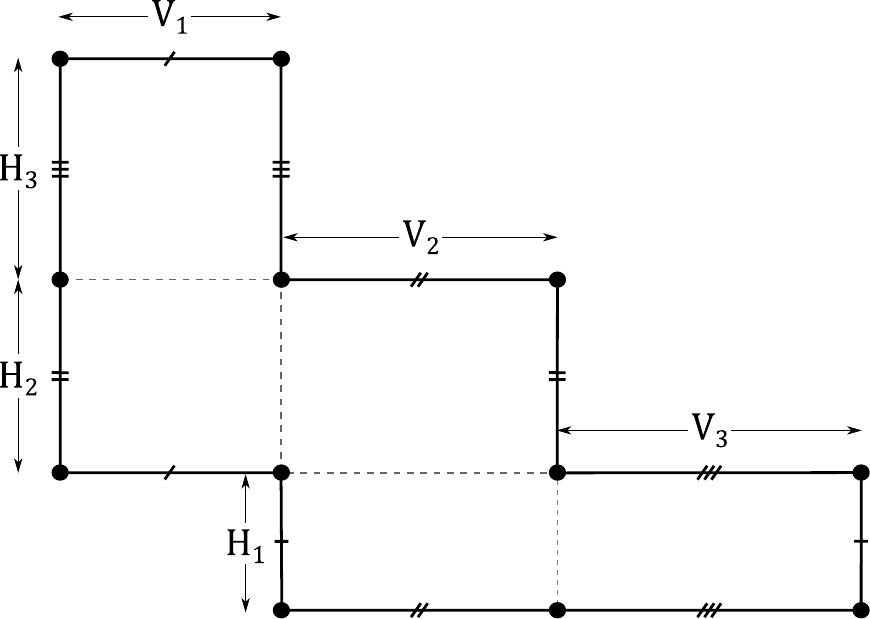}
\caption{A origami in the hyperelliptic component of $\mathcal{H}(4)$.}
\label{f.H4-hyp}
\end{figure}

In the horizontal direction, there are 3 cylinders $Ch_1,Ch_2,Ch_3$. The height of these cylinders 
are respectively $H_1,H_2,H_3$, while the length of the waist curves are respectively $\ell (Ch_1)=V_2 +V_3, \ell (Ch_2)=V_1 + V_2,\ell (Ch_3)=V_1$.

\smallskip

In the vertical direction, there are also 3 cylinders $Cv_1, Cv_2, Cv_3$. The height of these cylinders 
are respectively $V_1,V_2,V_3$, while the length of the waist curves are respectively 
$\ell (Cv_1)=H_2+H_3, \ell (Cv_2)=H_1+H_2,\ell (Cv_3)=H_1$.

\smallskip

We will denote by $\sigma_i$, $i=1,2,3$, the homology class of the waist curve of $Ch_i$, oriented rightwards,
and by $\zeta_i$ , $i=1,2,3$, the homology class of the waist curve of $Cv_i$, oriented upwards. The symplectic
intersection form $\iota$ on $H_1 (\mathcal O , \Zset)$ satisfies
$$ \iota (\sigma_i,\sigma_j) = \iota (\zeta_i, \zeta_j) =0, \quad \iota(\sigma_i,\zeta_j) = I_{ij},$$
with
$$ I = \left(\begin{array}{ccc}0 &  1 &1  \\ 1& 1 &0 \\ 1 &0 &0 \end{array}\right) .$$

\begin{remark} Again, the intersection matrix $I$ above has maximal rank (equal to $3$). By Forni's criterion \cite{Fo3}, it follows that all Lyapunov exponents of any origami $\mathcal{O}$ of the form indicated in Figure~\ref{f.H4-hyp} are also non-zero. 
\end{remark}

We review some elementary properties of $\mathcal O$:

\begin{itemize}
\item $ \mathcal O$ has a single conical singularity of total angle $10 \pi$, so it belongs to 
$\mathcal H (4)$.
\item The number of squares is 
$$  N(\mathcal O) = H_1 (V_2 +
V_3) + H_2 (V_1 +V_2) + H_3 V_1.$$
\item  $\mathcal O$ is reduced if and only if  ${\rm gcd}(H_1,H_2,H_3) = {\rm gcd}(V_1,V_2,V_3) =1$: indeed,
 the condition is clearly necessary. It  is also sufficient because the vectors $(V_i,0)$ and $(0,H_i)$ are 
 periods for $i=1,2,3$.
 \item The classes $\sigma_1, \sigma_2, \sigma_3, \zeta_1, \zeta_2, \zeta_3$ form a basis of the 
integral homology $H_1 (\mathcal O , \Zset)$: indeed, the matrix $I$ above is invertible over $\Zset$.
\item Proposition \ref {monodromy_odd}, determining the monodromy group of $\mathcal O$,  is also valid in the hyperelliptic case, but will not be used.

\end{itemize}

 The origami $\mathcal O$ admits an anti-automorphism with $8$ fixed points, as indicated in Figure~\ref{f.H4-hyp-W} below.

 \begin{figure}[htb!]
\includegraphics[scale=1]{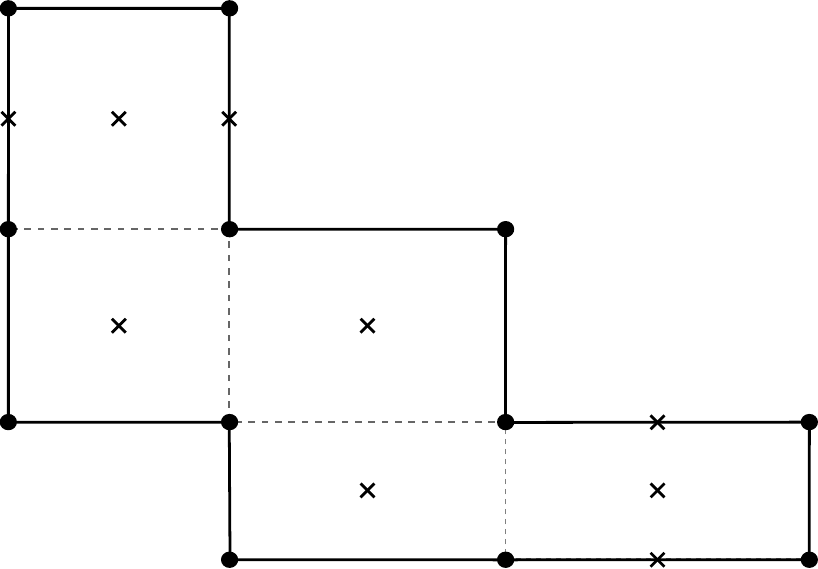}
\caption{Fixed points of the anti-automorphism of a origami in the hyperelliptic component of $\mathcal{H}(4)$.}
\label{f.H4-hyp-W}
\end{figure}

Therefore, it belongs to the hyperelliptic component of $\mathcal H(4)$. The anti-automorphism preserves each 
horizontal cylinder $Ch_i$ and each vertical cylinder $Cv_j$. From Figure~\ref{f.H4-hyp-W}, we read off the 
Hubert-Leli\`evre-Kani invariant of $\mathcal O$. It depends only on the classes in $\Zset_2$ of the $H_i$ and $V_j$. Writing these congruence classes as $ \left(\begin{array}{ccc} \bar H_1 & \bar H_2 & \bar H_3 \\  \bar V_1 & \bar V_2 & \bar V_3 \end{array} \right)$, the invariant is equal to

\begin{itemize}
\item $(4,[1,1,1])$  for \setlength\arraycolsep{2pt} $\tiny{\left(\begin{array}{ccc} 1&0&0\\0&1&0  \end{array} \right)}$,
$\tiny{\left(\begin{array}{ccc} 0&1&0\\1&0&0  \end{array} \right)}$, 
$\tiny{\left(\begin{array}{ccc} 1&0&0\\0&0&1  \end{array} \right)}$,
$\tiny{\left(\begin{array}{ccc} 0&0&1\\1&0&0  \end{array} \right)}$,
$\tiny{\left(\begin{array}{ccc} 0&1&0\\0&1&0  \end{array} \right)}$;
\vspace{7pt}
\item $(2,[3,1,1])$  for $\tiny{\left(\begin{array}{ccc} 1&0&0\\1&1&0  \end{array} \right)}$,
$\tiny{\left(\begin{array}{ccc} 1&1&0\\1&0&0  \end{array} \right)}$, 
$\tiny{\left(\begin{array}{ccc} 1&0&0\\1&0&1  \end{array} \right)}$,
$\tiny{\left(\begin{array}{ccc} 1&0&1\\1&0&0  \end{array} \right)}$,
$\tiny{\left(\begin{array}{ccc} 0&1&0\\1&0&1  \end{array} \right)}$,
$\tiny{\left(\begin{array}{ccc} 1&0&1\\0&1&0  \end{array} \right)}$, 
$\tiny{\left(\begin{array}{ccc} 0&1&0\\0&1&1  \end{array} \right)}$,
$\tiny{\left(\begin{array}{ccc} 0&1&1\\0&1&0  \end{array} \right)}$, 
$\tiny{\left(\begin{array}{ccc} 0&0&1\\1&0&1  \end{array} \right)}$,
$\tiny{\left(\begin{array}{ccc} 1&0&1\\0&0&1  \end{array} \right)}$, 
$\tiny{\left(\begin{array}{ccc} 0&0&1\\1&1&0  \end{array} \right)}$,
$\tiny{\left(\begin{array}{ccc} 1&1&0\\0&0&1  \end{array} \right)}$, 
$\tiny{\left(\begin{array}{ccc} 0&1&1\\1&1&0  \end{array} \right)}$,
$\tiny{\left(\begin{array}{ccc} 1&1&0\\0&1&1  \end{array} \right)}$, 
$\tiny{\left(\begin{array}{ccc} 1&1&0\\1&1&0  \end{array} \right)}$;
\vspace{7pt}
\item $(0,[5,1,1])$  for $\tiny{\left(\begin{array}{ccc} 0&0&1\\1&1&1  \end{array} \right)}$,
$\tiny{\left(\begin{array}{ccc} 1&1&1\\0&0&1  \end{array} \right)}$,
$\tiny{\left(\begin{array}{ccc} 1&1&1\\1&1&1  \end{array} \right)}$;
\vspace{7pt}
\item $(0,[3,3,1])$  for $\tiny{\left(\begin{array}{ccc} 1&0&1\\1&1&1  \end{array} \right)}$,
$\tiny{\left(\begin{array}{ccc} 1&1&1\\1&0&1  \end{array} \right)}$,
$\tiny{\left(\begin{array}{ccc} 0&1&1\\1&1&1  \end{array} \right)}$,
$\tiny{\left(\begin{array}{ccc} 1&1&1\\0&1&1  \end{array} \right)}$,
$\tiny{\left(\begin{array}{ccc} 0&1&1\\0&1&1  \end{array} \right)}$;
\vspace{7pt}
\item $(3,[2,2,0])$  for $\tiny{\left(\begin{array}{ccc} 1&0&0\\0&1&1  \end{array} \right)}$,
$\tiny{\left(\begin{array}{ccc} 0&1&1\\1&0&0  \end{array} \right)}$,
$\tiny{\left(\begin{array}{ccc} 0&1&0\\0&0&1  \end{array} \right)}$,
$\tiny{\left(\begin{array}{ccc} 0&0&1\\0&1&0  \end{array} \right)}$, 
$\tiny{\left(\begin{array}{ccc} 0&1&0\\1&1&0  \end{array} \right)}$,
$\tiny{\left(\begin{array}{ccc} 1&1&0\\0&1&0  \end{array} \right)}$, 
$\tiny{\left(\begin{array}{ccc} 1&0&0\\1&0&0  \end{array} \right)}$,
$\tiny{\left(\begin{array}{ccc} 0&0&1\\0&0&1  \end{array} \right)}$;
\vspace{7pt}
\item $(1,[4,2,0])$  for $\tiny{\left(\begin{array}{ccc} 1&0&0\\1&1&1  \end{array} \right)}$,
$\tiny{\left(\begin{array}{ccc} 1&1&1\\1&0&0  \end{array} \right)}$,
$\tiny{\left(\begin{array}{ccc} 0&1&0\\1&1&1  \end{array} \right)}$,
$\tiny{\left(\begin{array}{ccc} 1&1&1\\0&1&0  \end{array} \right)}$, 
$\tiny{\left(\begin{array}{ccc} 0&1&1\\0&0&1  \end{array} \right)}$,
$\tiny{\left(\begin{array}{ccc} 0&0&1\\0&1&1  \end{array} \right)}$, 
$\tiny{\left(\begin{array}{ccc} 1&1&1\\1&1&0  \end{array} \right)}$,
$\tiny{\left(\begin{array}{ccc} 1&1&0\\1&1&1  \end{array} \right)}$;
\vspace{7pt}
\item $(1,[2,2,2])$  for $\tiny{\left(\begin{array}{ccc} 1&1&0\\1&0&1  \end{array} \right)}$,
$\tiny{\left(\begin{array}{ccc} 1&0&1\\1&1&0  \end{array} \right)}$, 
$\tiny{\left(\begin{array}{ccc} 0&1&1\\1&0&1  \end{array} \right)}$,
$\tiny{\left(\begin{array}{ccc} 1&0&1\\0&1&1  \end{array} \right)}$, 
$\tiny{\left(\begin{array}{ccc} 1&0&1\\1&0&1  \end{array} \right)}$.
\end{itemize}

The map $\pi_*: H_1 (\mathcal O , \Zset) \to H_1(\Tset^2,\Zset) = \Zset^2$ induced by the canonical covering
$\pi: \mathcal O \to \Tset^2$ is given by

$$ \pi_*(\sigma_i) = (\ell(Ch_i),0), \quad \pi_*(\zeta_i) = (0, \ell(Cv_i)), \quad i=1,2,3.$$

Define , for $i = 1,2$

$$ \Sigma_i := \ell(Ch_3) \sigma_i - \ell(Ch_i) \sigma_3, \quad Z_i = \ell(Cv_3) \zeta_i - \ell(Cv_i) \zeta_3.$$ 

Then $\Sigma_1, \Sigma_2, Z_1,Z_2$ are elements of $H_1^{(0)}(\mathcal O, \Zset)$ which span a subgroup 
of finite index of this group; they form a basis of  $H_1^{(0)}(\mathcal O, \Qset)$.

\subsection {Two parabolic elements in $\textrm {Aff}(\mathcal O)$}\label{hyperparabolic}

Set 

$$ L_h := \ell(Ch_1)  \ell(Ch_2)  \ell(Ch_3), \quad P_h :=  \left(\begin{array}{cc}1 &  L_h  \\ 0& 1  \end{array}\right) .$$

The matrix $P_h$ belongs to the Veech group $\SL(\mathcal O)$. Indeed, the associated element $p_h$ 
of the affine group acts on homology according to $p_h.\sigma_i= \sigma_i$ and 

$$ p_h. \zeta_i  = \zeta_i + \sum_{j=1}^3 I_{ji} H_j \frac {L_h}{\ell(Ch_j)} \sigma_j , \quad i=1,2,3.$$

We deduce from this formula the action on $H_1^{(0)}(\mathcal O, \Qset)$. We have $p_h. \Sigma_i = \Sigma_i$ for $i=1,2$ and

$$   p_h. Z_i  = Z_i + \sum_{j=1}^3 H_{ij} H_j \frac {L_h}{\ell(Ch_j)} \sigma_j ,\quad i=1,2.$$
with 
$$ H_{ij}:= \ell(Cv_3) I_{ji} - \ell(Cv_i) I_{j3}.$$

The last formula can be rewritten as 

$$   p_h. Z_i  = Z_i + \sum_{j=1}^2 H_{ij} H_j \ell(Ch_{3-j}) \Sigma_j ,\quad i=1,2.$$

Up to now, the formulas were the same as in the odd case. Substituting the values of $I_{ji}$, $\ell(Ch_j)$, $ \ell(Cv_j)$, we get in the new setting

$$ H_{12} = H_{22} = H_1, \quad  H_{11} = -(H_2+H_3), \quad H_{21} =-H_2,$$

\begin{eqnarray*}
 p_h. Z_1  &=& Z_1 -H_1(H_2 + H_3)(V_1 + V_2) \Sigma_1 + H_1H_2(V_2+V_3) \Sigma_2, \\
 p_h. Z_2  &=& Z_2 - H_1H_2 ( V_1+V_2) \Sigma_1 + H_1 H_2 (V_2+V_3) \Sigma_2.
 \end{eqnarray*}

In the hyperelliptic model under consideration, the horizontal and the vertical direction play the same role.
We define

$$ L_v := \ell(Cv_1)  \ell(Cv_2)  \ell(Cv_3), \quad P_v :=  \left(\begin{array}{cc}1 & 0  \\  L_v & 1  \end{array}\right) .$$

The matrix $P_v$ belongs to the Veech group $\SL(\mathcal O)$. The associated element $p_v$ 
of the affine group acts on $H_1^{(0)}$ accordind to $p_v.Z_i = Z_i$ and 

\begin{eqnarray*}
 p_v. \Sigma_1  &=& \Sigma_1 -V_1(V_2 + V_3)(H_1 + H_2) Z_1 + V_1V_2(H_2+H_3) Z_2, \\
 p_v. \Sigma_2  &=& \Sigma_2 - V_1V_2 ( H_1+H_2) Z_1 + V_1 V_2 (H_2+H_3) Z_2.
 \end{eqnarray*}

 We define
 
 \begin{eqnarray*}
  Q_h &:=&  \left(\begin{array}{cc}- H_1(H_2+ H_3)(V_1+ V_2) &   -H_1H_2(V_1+V_2)\\H_1H_2(V_2+V_3) &  H_1H_2(V_2+V_3) \end{array}\right), \\
 Q_v &:=&  \left(\begin{array}{cc} - V_1 (V_2+V_3)(H_1+H_2) &    -V_1V_2(H_1+H_2) \\  V_1V_2(H_2+H_3)&   V_1V_2(H_2+H_3)\end{array}\right),
  \end{eqnarray*}
 so that the matrices of $p_h,p_v$ in the basis $\Sigma_1, \Sigma_2, Z_1,Z_2$ are respectively
 
 $$  \left(\begin{array}{cc}1 &  Q_h  \\ 0& 1  \end{array}\right), \quad  \left(\begin{array}{cc}1 & 0  \\  Q_v & 1  \end{array}\right) .$$

We will investigate whether  $A:= p_v \circ p_h$ is Galois-pinching.

\subsection {Eigenvalues and eigenvectors for $Q_h. Q_v$, $Q_v.Q_h$ and $A$}

A vector
$$w= x_1 \Sigma_1 + x_2 \Sigma_2 + y_1 Z_1 +y_2 Z_2$$
 is eigenvector of $A$ associated 
to the eigenvalue $\lambda$ iff  $x:=(x_1,x_2), y:=(y_1,y_2)$ satisfy

$$ x = \frac 1{\lambda -1} Q_h.y, \quad y = \frac  {\lambda}{\lambda -1} Q_v.x.$$

Then, $x$ and $y$ are eigenvectors of $Q_h. Q_v$, $Q_v.Q_h$  respectively, associated to the same eigenvalue
$ \mu := \lambda + \lambda^{-1} -2$.
Let
\begin{eqnarray*}
 d_h & := & \det Q_h =- H_1^2 H_2 H_3 (V_1 +V_2)(V_2+V_3)  =:  H_1^2 \bar d_h, \\
 d_v &:=& \det Q_v = -V_1^2 V_2 V_3  (H_1+H_2)(H_2+H_3) =: V_1^2 \bar d_v,\\
 d&:=& d_h d_v,  \quad \bar d := \bar d_h \bar d_v ,\\
 t&:=& \textrm{tr} (Q_h.Q_v) = \textrm{tr} (Q_v.Q_h) \\
 & = & H_1 V_1 [H_2^2 V_3(V_1+V_2)+ H_1 H_2 V_1 (V_2+V_3) + H_1 H_3(V_1+V_2)(V_2+V_3) \\
  & &+ H_2H_3(V_2^2 + V_1V_3 + 2 V_2V_3)]\\
 & =:& H_1 V_1 \bar t.
 \end{eqnarray*}
The eigenvalues of $Q_h. Q_v$ are the solutions of $\mu^2 - t\mu + d = 0$ with discriminant
$$ \Delta_1 := t^2 - 4d = H_1^2  V_1^2 (\bar t^2 - 4 \bar d_h \bar d_v).$$
Thus, we get 
\begin{eqnarray*}
\bar \Delta_1 &:=& \bar t^2 - 4 \bar d_h \bar d_v\\
        &=& [H_2^2 V_3(V_1+V_2)+ H_1 H_2 V_1 (V_2+V_3) + H_1 H_3(V_1+V_2)(V_2+V_3) \\
  & &+ H_2H_3(V_2^2 + V_1V_3 + 2 V_2V_3)]^2\\
        &  &  - 4H_2 H_3 (H_1+H_2)(H_2+H_3)V_2 V_3(V_1 +V_2)(V_2+V_3)    .
 \end{eqnarray*}
With Proposition~\ref{case1} in mind, we also define
$$ \Delta_2 := d(d+4t +16),\quad \bar\Delta_2:=\bar d (d+4t +16).$$

\subsection{One-parameter subfamilies}\label {subfamilieshyper}
As in the odd case, we now restrict our attention to a finite number of one-parameter subfamilies.
In each subfamily, we fix the values of $H_2,H_3,V_1,V_2,V_3$ , while $H_1$ runs along an arithmetic progression.
\par
In all subfamilies, one has $V_1 = H_2 =1$. The others parameters are as follows
\begin{itemize}
\item In the first four subfamilies, one has $V_2 = V_3 =1$,  $1 \leq H_3 \leq 4 $, $H_1 = 2n$. The number of squares is $4n +2 +H_3$.
\item  In the next  eighteen   subfamilies, one has $V_2 = 1$, $V_3 =2$,  $1 \leq H_3 \leq 18$, $H_1 = 6n$. The number of squares is $18n +2 +H_3$.
\item Next, we take $V_2 = 2$, $V_3 = 1$, $H_3 \in \{1,3,5\}$, $H_1 = 2n$. The number of squares is $6n + 3 +H_3$.
\item Then, we take $V_2 =V_3 = 2$, $2 \leq H_3 \leq 16$, $H_3$ even, $H_1 = 4n$. The number of squares is $ 16n + 3 +H_3$.
\item Finally, we take $V_2 =3$, $V_3 =1$,  $1 \leq H_3 \leq 15$,  $H_3$ odd, $H_1 = 4n +3$. The number of squares is $ 16n +16 +H_3$.
\end{itemize}

These 41 families are divided into 5 groups distinguished by their values of $V_2,V_3$.
As $V_1 = H_2 =1$, all origamis in these families are reduced, hence primitive.

\begin{proposition}\label{directionhyper}
For each origami in these  families, the direction $(1,1)$ is a $2$-cylinder direction, and therefore has homological dimension  $2$.
\end{proposition}
\begin{proof}
We label the separatrices in the direction $(1,1)$ in cyclic order by $\Zset_{10}$ as indicated in Figure~\ref{f.H4-hyp-sc}.

\begin{figure}[htb!]
\includegraphics[scale=1]{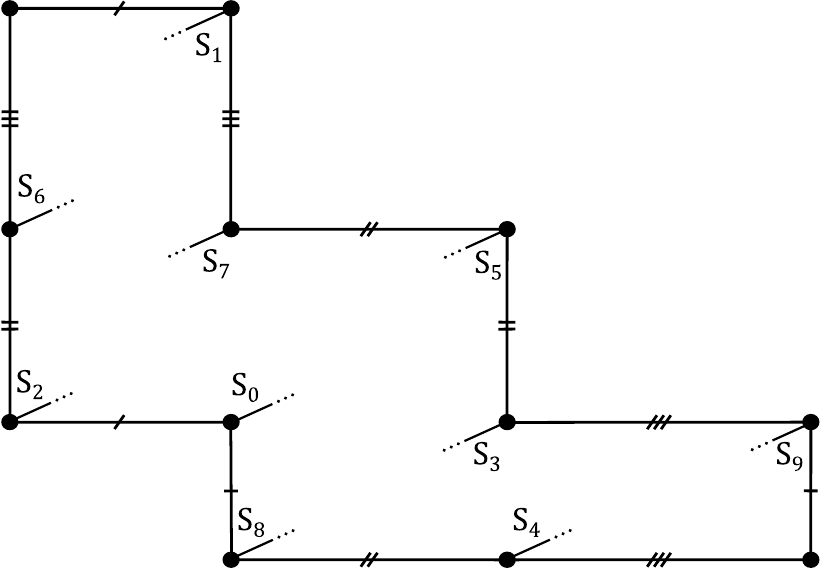}
\caption{Separatrices in the direction $(1,1)$ of a origami in the hyperelliptic component of $\mathcal{H}(4)$.}
\label{f.H4-hyp-sc}
\end{figure}

We  check in each case that there are exactly $3$ balanced saddle connections, which correspond to the $2$-cylinder case according to Proposition \ref{listLelievre}. Actually, the saddle connections from $S_6$ to $S_1$
and $S_2$ to $S_7$ are always there, hence it is sufficient to see that not all saddle connections are balanced.
The separatrix $S_8$ is connected to $S_9$ in the first two groups and also in the fourth. The separatrix $S_4$ is connected to $S_5$ in the last group. Finally, in the third group, there is a  connection  between $S_8$ and $S_9$ if $H_3 = 6n$, between $S_0$ and $S_9$ if $H_3 = 6n+2$, and between $S_8$ and $S_5$ if 
$H_3 = 6n +4$.
\end{proof}

Next, we compute the HLK-invariant of the origamis in these families. We deduce from Subsection~\ref{hyperellipticmodel}
the HLK-invariant of the families under consideration.

\begin{proposition}\label{HLKH4}
The HLK-invariant of the origamis in the families under consideration is equal to
\begin{itemize}
\item $(0,[3,3,1])$ for the families in the first group with $H_3$ odd;
\item $(1,[4,2,0])$ for the families in the first group with $H_3$ even;
\item $(2,[3,1,1])$ for the families in the second group with $H_3$ odd;
\item $(3,[2,2,0])$ for the families in the second group with $H_3$ even;
\item $(1,[2,2,2])$ for the families in the third group;
\item $(4,[1,1,1])$ for the families in the fourth group;
\item $(0,[5,1,1])$ for the families in the last group.
\end{itemize}
\end{proposition}

As in the odd case, we plan to apply the elementary Galois theory of Subsection \ref{Galoistheory} in order to prove that, when $n$
 is large enough, the affine homeomorphism $A$ constructed in Subsection \ref{hyperparabolic} is Galois-pinching. Specializing the formulas of the last subsection gives
 
 \begin{eqnarray*}
 d &=& H_1^2 \bar d, \quad  t  = H_1 \bar t, \quad  \Delta_1 =  H_1^2 \bar \Delta_1, \quad \Delta_2 = H_1^2 \bar \Delta_2,\\
 \bar d &=& V_2 V_3 (1+V_2)(V_2+V_3)H_3(1+H_3) (1+H_1),\\
 \bar t &=& H_1(V_2+V_3)[1 +H_3(1+V_2)] + V_3(1+V_2) + H_3(V_2^2 + 2V_2V_3 + V_3),\\
 \bar\Delta_1 &=& \bar t^2 - 4 \bar d,\\
 \bar \Delta_2 &=& \bar d ( H_1^2 \bar d + 4H_1 \bar t +16).
  \end{eqnarray*}

Observe that $d,t$ are always positive and $\Delta_1$ is positive if $H_1$ is large enough. 
By Proposition \ref{positiveroots}  the eigenvalues of $A$ acting on $H_1^{(0)}$ are simple, real and positive for the 
origamis in these families if $H_1$ is large enough.
\par
We now write $\bar d$, $\bar t$, $\bar \Delta_1$ more explicitly in each of the five groups of families:
\begin{itemize}
\item In the first group, we have $V_2 = V_3 =1$ and
\begin{eqnarray*}
\bar d &=& 4H_3(1+H_3) (H_1 +1), \\
\bar t &=& (2 +4H_3) (H_1+1), \\
\bar \Delta_1 &=& 4 (H_1 +1) [(1+2H_3)^2 H_1 +1].
  \end{eqnarray*}
\item In the third group, we have $V_2 = 2$, $V_3 =1$ and
\begin{eqnarray*}
\bar d &=& 18H_3(1+H_3) (H_1 +1), \\
\bar t &=& (3 +9H_3) (H_1+1), \\
\bar \Delta_1 &=& 9 (H_1 +1) [(1+3H_3)^2 H_1 + (H_3-1)^2].
\end{eqnarray*}
\item In the last group, we have $V_2 =3$, $ V_3 =1$ and
\begin{eqnarray*}
\bar d &=& 48H_3(1+H_3) (H_1 +1), \\
\bar t &=& (4 +16H_3) (H_1+1), \\
\bar \Delta_1 &=& 16 (H_1 +1) [(1+4H_3)^2 H_1 + (2H_3 -1)^2].
  \end{eqnarray*}
 \item In the second group, we have $V_2 = 1$, $V_3 =2$ and
\begin{eqnarray*}
\bar d &=& 12H_3(1+H_3) (H_1 +1), \\
\bar t &=& (3 +6H_3) H_1+ 4 +7H_3, \\
\bar \Delta_1 &=& (3+6H_3)^2 H_1^2 + (36H_3^2 +42 H_3 +24)H_1 + (H_3 +4)^2.
  \end{eqnarray*}
  \item In the fourth group, we have $V_2 = 2$, $V_3 =2$ and
\begin{eqnarray*}
\bar d &=& 48H_3(1+H_3) (H_1 +1), \\
\bar t &=& (4 +12H_3) H_1+ 6 +14H_3, \\
\bar \Delta_1 &=& (4+12H_3)^2 H_1^2 + (144H_3^2 +64 H_3 +48)H_1 + (2H_3 -6)^2.
  \end{eqnarray*}
  \end{itemize}
\begin{proposition}
In each of the 41 families considered above, the quantity $\bar \Delta_1$ is not a square if $H_1$ is large enough.
\end{proposition}
\begin{proof}
For each family, $\bar \Delta_1$  is a degree two polynomial in $H_1$ with integer coefficients. 
Moreover, the leading coefficient is always a square. The only way that the values of such a polynomial may be squares for infinitely many integers $H_1$ is that this polynomial is itself a square. 
\par
This is obviously not the case for families in the first, third or last group (recall that $H_3>0$). For families in the second group, we compute the discriminant of this degree two polynomial to be equal to 
$(24)^2 H_3(2H_3-1)(H_3+1)^2 \ne 0$.  For families in the fourth group, we compute the discriminant  to be equal to 
$6(32)^2 H_3(3H_3-1)(H_3+1)^2 \ne 0$.
\end{proof}

 \begin{lemma} For each of the 41 families considered above, the third degree polynomial $ H_1^2 \bar d + 4H_1 \bar t +16$ is irreducible, except for the family in the first group with $H_3 =2$, where it is equal to
  $8(H_1+2)(3 H_1^2 + 2H_1 +1)$, and the family in the second group with $H_3 =3$, where it is equal to 
  $4(H_1+1)(36 H_1^2 + 21 H_1 +4)$.
 \end{lemma}
 
This is checked using $\Zset$-irreducibility by the lemma 
of Gauss, implemented in any standard mathematical software, e.g. \emph{Mathematica} or \emph{Sage}.

\begin{proposition}
\begin{enumerate} 
\item For each family such that $ H_1^2 \bar d + 4H_1 \bar t +16$ is irreducible, the reduced degrees of $\Delta_2$ and $\bar \Delta_2$, as a polynomials in $H_1$, are equal to $4$. The reduced degrees of $\Delta_1 \Delta_2$ and $\bar\Delta_1 \bar \Delta_2$ are equal to $4$ for families in the first, third and last group, and equal to $6$ for families in the second and fourth group.
\item For the family in the first group with $H_3 =2$, the reduced degrees of $\Delta_2$,$\bar \Delta_2$, 
$\Delta_1 \Delta_2$ and $\bar\Delta_1 \bar \Delta_2$ are equal to $4$. 
\item For the family in the second group with $H_3 =3$, the reduced degrees of $\Delta_1 \Delta_2$ and $\bar\Delta_1 \bar \Delta_2$ are equal to $4$. The reduced degrees of $\Delta_2$ and $\bar \Delta_2$ are equal to $2$ but these quantities are not squares for $H_1$ large enough.
\end{enumerate}
\end{proposition}

\begin{proof}
\begin{enumerate}
\item  With $c= V_2 V_3 (1+V_2)(V_2+V_3)H_3(1+H_3)$, one has
$$  \bar \Delta_2 = c(H_1+1) ( H_1^2 \bar d + 4H_1 \bar t +16).$$
As $ H_1^2 \bar d + 4H_1 \bar t +16$ is irreducible, it has no double root and does not vanish for $H_1 = -1$.
Therefore the reduced degree of $\bar \Delta_2$ is equal to $4$. The roots of $\bar\Delta_1$ are either rational or quadratic, hence cannot be roots of $ H_1^2 \bar d + 4H_1 \bar t +16$. For families in the first, third or last group $H_1=-1$ is both a root of $\bar d$ and $\bar t$. This is not the case for families in the second or fourth group. As we have already seen that $\bar \Delta_1$ does not have double roots, we obtain the assertion on the reduced degree of $\bar\Delta_1 \bar \Delta_2$.
\item In this case, one has 
\begin{eqnarray*}
\bar \Delta_1 &=& 4(H_1+1)(25 H_1+1)\\
\bar \Delta_2 &=& 192(H_1+1)(H_1+2)(3 H_1^2 + 2H_1 +1),
\end{eqnarray*}
which implies the assertion of the proposition.
\item
In this case, one has 
\begin{eqnarray*}
\bar \Delta_1 &=& 21^2  H_1^2 + 474 H_1 + 7^2\\
\bar \Delta_2 &=& 24^2 (H_1+1)^2 (36H_1^2 +21H_1 +4),
\end{eqnarray*}
which implies the assertion on the reduced degrees. Also, as $36$ is a square but the discriminant of $36H_1^2 +21H_1 +4$ does not vanish, $\bar\Delta_2$ is not a square when $H_1$ is large enough.
\end{enumerate}
\end{proof}

Applying the results of Subsection~\ref{Galoistheory}, we obtain the desired result.

\begin{corollary}\label{Galhyper}
In each of the 41 subfamilies, if  $H_1$ is large enough, the affine homeomorphism $A$ constructed in Subsection \ref{hyperparabolic} is Galois-pinching.
\end{corollary}

\subsection {Conclusion for the hyperelliptic  case}\label{conclusionhyper}

We now have all the elements to prove Theorem~\ref{thm:introH4} for origamis of hyperelliptic type. 
By Propositions~\ref{no_automorphism}, \ref{directionhyper} and \ref{Galhyper}, they satisfy for $H_1$ large enough the hypotheses of 
Corollary~\ref{cor:mainsimpcrit}, hence their Lyapunov spectra are simple. 
\par
From proposition~\ref{HLKH4}, families in the first group provide origamis with HLK-invariant equal to 
$(0,[3,3,1])$ or $(1,[4,2,0])$. The number of squares is $N= 4n +2+H_3$, with $1 \leq H_3 \leq 4$,
hence any large number of squares is realizable by some family in the first group.
Proposition~\ref{HLKH4} allows to deal similarly with the other values of the HLK-invariant.
The proof of Theorem~\ref{thm:introH4} is complete in the hyperelliptic case. \hfill $\Box$

\appendix 

\section{A version of Avila-Viana simplicity criterion}\label{a.AVcriterion}

In this appendix we will present a streamlined proof of Theorem~\ref{t.AVcriterion}. Here, we will use  notations and definitions introduced in Subsection~\ref{ss.AVsetting} without further comments. 

\par

We begin by noticing that one may consider the cocycle $A$ over the invertible dynamics $\hat{f}:\hat{\Sigma}\to\hat{\Sigma}$ because the Lyapunov spectrum is not affected by this procedure.

\medskip

A crucial fact from (bi)linear algebra  allows to adapt the proof for $\Gset= \GL(d,\Rset)$ to other matrix groups 
$\Gset$ (symplectic or unitary). Let $k$ be an admissible integer, $A \in \Gset$; denote by 
$\sigma_1(A)\geq \dots\geq \sigma_d(A)$ the singular values of $A$, i.e the lengths of the semi-major axes of 
the ellipsoid $A(\{v:\|v\|=1\})$\footnote {As for Lyapunov exponents, singular values are counted with essential 
multiplicity, see Remark~\ref{multi}.}. Then, the subspace $\xi_A$ spanned by the $k$ largest semi-major axes 
belongs
to $G(k)$\footnote{ This subspace is uniquely defined only when $\sigma_k(A) > \sigma_{k+1}(A)$. When 
$\sigma_k(A) = \sigma_{k+1}(A)$, what is meant is that it is possible to select  the $k$ largest semi-major axes
in order that the subspace spanned by them belongs to $G(k)$.}. This fact is a simple consequence of the polar decomposition of matrices in $\Gset$.

\medskip

In the statement below, we use the following notations:

\begin{itemize}
\item $\underline{\ell}(x,n)$ is the terminal word of $x\in \Sigma_-$ of length $n$;
\item For $\underline{\ell} \in \Omega$, we write $\xi_{\underline \ell}$ for $\xi_{A^{\underline \ell}}$.

\end{itemize}

The main result towards the proof of Theorem \ref{t.AVcriterion} above is:

\begin{theorem}[A. Avila and M. Viana]\label{t.AVcriterion-1} For every  admissible integer $k$, there exists a map $\Sigma_-\to G(k)$, $x\mapsto \xi(x)$ verifying the properties:
\begin{itemize}
\item \emph{Invariance}: the map $\hat{\xi} = \xi\circ p^-$ satisfies $A(x)\hat{\xi}(x) = \hat{\xi}(\hat{f}(x))$;
\item for $\mu_-$-almost every $x\in\Sigma_-$, $\frac{\sigma_k(A^{\underline{\ell}(x,n)})}{\sigma_{k+1}(A^{\underline{\ell}(x,n)})}\to+\infty$ and $\xi_{\underline{\ell}(x,n)}\to\xi(x)$ as $n\to +\infty$;
\item for all $F'\in G(d-k)$, we have $\xi(x)\cap F'=\{0\}$ for a set of positive $\mu_-$-measure.
\end{itemize}

\end{theorem}
This result corresponds to \cite[Theorem A.1]{AVKZ}. As the reader can 
check (see  Subsection A.6 ``Proof of Theorem 7.1'' of \cite{AVKZ}), 
it is not hard to deduce Theorem~\ref{t.AVcriterion} from 
Theorem~\ref{t.AVcriterion-1}.
\par
\begin{proof}[Sketch of proof of Theorem~\ref{t.AVcriterion} assuming 
Theorem~\ref{t.AVcriterion-1}] 
For each admissible integer $k$ and each $x\in \Sigma_-$, Theorem~\ref{t.AVcriterion-1} provides us,
with a subspace $\xi(x)\in G(k)$ verifying the properties above. 
By using the same theorem with the time ``reversed'', one gets  for $y\in\Sigma_+$ a subspace 
$\xi_*(y) \in G(d-k)$ verifying similar properties. From the third property in the theorem, one deduces the transversality property  $\xi(x)\cap\xi_*(y)=\{0\}$ for almost every $(x,y)\in\hat{\Sigma}$. The second property
in the theorem implies that $\xi(x)$ is associated with the $k$ largest exponents and $\xi_*(y)$ is associated with the $d-k$ smallest exponents. Then the transversality property permits to show that, for any  admissible integer $k$, the $k$th Lyapunov exponent is strictly larger than  the $(k+1)$th exponent .  This shows that the Lyapunov spectrum of $A$ is simple  in the sense  defined in Subsection  \ref{ss.AVsetting}.
\end{proof}

This reduces our considerations to the discussion of Theorem~ \ref{t.AVcriterion-1}.
Let $k$ be an admissible integer. We denote by $p', p'' $ the natural projections from $\hat{\Sigma}\times G(k)$
onto $\hat{\Sigma}$ and $G(k)$ respectively.

\begin{definition} A \emph{u-state} is a probability measure $\hat{m}$ on 
$\hat{\Sigma}\times G(k)$ such that $p'_*(\hat{m})=\hat{\mu}$ and there 
exists a constant $C(\hat{m})$ with
$$\frac{\hat{m}(\Sigma_-(\underline{\ell}^0)\times\Sigma(\underline{\ell})\times X)}{\mu(\Sigma(\underline{\ell}))} \leq C(\hat{m})\frac{\hat{m}(\Sigma_-(\underline{\ell}^0)\times\Sigma(\underline{\ell}')\times X)}{\mu(\Sigma(\underline{\ell}'))}$$
for any Borelian $X\subset G(k)$, $\underline{\ell}^0,\underline{\ell},
\underline{\ell}'\in\Omega$.
\end{definition}
Roughly speaking, the previous condition says that u-states are almost product measures.

\begin{example} Given any probability measure $\nu$ on $G(k)$, $\hat{m}:=\hat{\mu}\times\nu$ is a u-state with $C(\hat{m})=C(\mu)^2$, where $C(\mu)$ is the constant appearing in the bounded distortion property (see Definition~\ref{def:boundeddist}).
\end{example}

\begin{proposition}\label{p.a-1} There exists a u-state invariant under $(\hat{f},A)$.
\end{proposition}

\begin{proof} The argument is very classical and we will only sketch its main steps. Even though the space $\hat{\Sigma}$ may not be compact (in the case of an alphabet $\Lambda$ with countably many symbols), the space of probability measures on $\hat{\Sigma}\times G(k)$ projecting to $\hat{\mu}$ is compact in the weak-* topology. In particular, for each $C>0$, it follows that the space of u-states $\hat{m}$ with $C(\hat{\mu})\leq C$ is a convex compact set.

A short direct computation    (\cite[Lemma~A.2]{AVKZ}) shows that,  for any u-state $\hat{m}_0$ 
and any $n>0$ ,  $\hat{m}(n):=(\hat{f},A)_*^n\hat{m}_0$ is a u-state with 
$C(\hat{m}(n))\leq C(\hat{m})C(\mu)^2$. Then the standard Krylov-Bogolyubov argument completes the proof: 
any accumulation point of the  Cesaro averages of $\hat{m}(n)$ is a u-state invariant under $(\hat{f},A)$.
\end{proof}

The following result is a simple application of the martingale convergence theorem 
(see \cite[Lemma~A.4]{AVKZ}  for a short  proof). We recall that $\underline{\ell}(x,n)$ denotes 
the terminal word of  $x\in\Sigma_-$ of length $n$.

\begin{proposition}\label{p.a-2} Let $\hat{m}$ be a probability measure on $\hat{\Sigma}\times G(k)$ with $p'_*\hat{m} = \hat{\mu}$. For any $x\in\Sigma_-$, and any Borelian subset  $X\subset G(k)$ , let
$$\hat{m}_n(x)(X):=\frac{\hat{m}(\Sigma_-(\underline{\ell}(x,n))\times\Sigma\times X)}{\hat{m}(\Sigma_-(\underline{\ell}(x,n))\times\Sigma\times G(k))}$$
Then, for $\mu_-$-almost every $x\in\Sigma_-$, $\hat{m}_n(x)$ converges in the weak$*$ topology
to some $\hat{m}(x)$. 
\end{proposition}

 Let $\hat{m}$ be a $(\hat{f},A)$-invariant u-state given by Proposition \ref{p.a-1}. Define $\hat m_n (x)$ as in Proposition  \ref {p.a-2}. Let also $\nu=p''_*\hat{m}$. 
 For any  $x\in\Sigma_-$,  define a sequence of probability measures on $G(k)$ by
$$\nu_n(x):=A^{\underline{\ell}(x,n)}_*\nu.$$

Let $X$ be a Borelian subset of $G(k)$. As $\hat{m}$ is $(\hat{f},A)$-invariant  , we have
$$\hat{m}_n(x)(X)=\frac{\hat{m}(\Sigma_-(\underline{\ell}(x,n))\times\Sigma\times X)}{\hat{m}(\Sigma_-(\underline{\ell}(x,n))\times\Sigma\times G(k))} = \frac{\hat{m}(\Sigma_-\times\Sigma(\underline{\ell}(x,n))\times A^{-\underline{\ell}(x,n)}(X))}{\hat{m}(\Sigma_-\times\Sigma(\underline{\ell}(x,n))\times G(k))}$$

On the other hand, by definition, we have
$$\nu_n(x)(X) = \nu(A^{-\underline{\ell}(x,n)}(X)) = \hat{m}(\hat{\Sigma}\times A^{-\underline{\ell}(x,n)}(X)).$$

Since $\hat{m}$ is a u-state, we obtain 

$$C(\hat{m})^{-2}\leq \hat{m}_n(x)(X)/\nu_n(x)(X)\leq C(\hat{m})^2.$$

In particular,
\begin{corollary}\label{c.a-1} For $\mu_-$-almost every $x\in\Sigma_-$, the probability measure $\hat{m}(x)=\lim\hat{m}_n(x)$ is equivalent to any accumulation point of the sequence $\nu_n(x)$.
\end{corollary}

The crucial step in the proof of Theorem \ref{t.AVcriterion-1} is given by the

\begin{proposition}\label{p.a-3} For $\mu_-$-almost every $x\in\Sigma_-$, there exists a subsequence $\nu_{n_k}(x)$, $n_k=n_k(x)\to\infty$, converging to a Dirac mass.
\end{proposition}
\begin{proof}[Sketch of proof of Theorem \ref{t.AVcriterion-1} assuming Proposition \ref{p.a-3}]
(see end of Subsection A.4 of \cite{AVKZ}). By Corollary \ref{c.a-1} and Proposition \ref{p.a-3},  $\hat{m}(x)$  is a Dirac mass $ \delta_{\xi(x)}$ for $\mu_-$-almost every $x\in\Sigma_-$.
\par
Then $x\mapsto\xi(x)$ has the desired properties: the  invariance property (first item of Theorem \ref{t.AVcriterion-1}) follows from the $(\hat{f},A)$-invariance of $\hat{m}$; the other two  items are a consequence of the pinching and twisting assumptions on the cocycle $A$.
\end{proof}

\begin{proof}[Proof of  Proposition \ref{p.a-3}]

 Let $\underline {\ell}^* \in \Omega$ be a word such that the matrix $A^{\underline{\ell}^*}$ is pinching. As there 
 are only finitely many $A^{\underline{\ell}^*}$-invariant subspaces in $G(k)$ (see Remark~\ref{r.noether-3}), 
one can use the twisting hypothesis 
 to choose  $\underline{\ell}^0\in\Omega$ such that, for each  admissible integer $k$ and for every pair 
 of $A^{\underline{\ell}^*}$-invariant subspaces $F\in G(k), F'\in G(d-k)$, one has 
 $A^{\underline{\ell}^0}(F)\cap F'=\{0\}$.
 \par
We claim that there exists $m\geq 1$, $\underline{\ell}_1,\dots,\underline{\ell}_m\in\Omega$ and $\delta>0$
 such that, for each  admissible integer $k$ and for every $F'\in G(d-k)$, there exists $\underline{\ell}_i$ 
 with $A^{\underline{\ell}_i}(F_+(A^{\underline{\ell}^*}))\cap F'=\{0\}$ and the angle between 
 $A^{\underline{\ell}_i}(F_+(A^{\underline{\ell}^*}))$ and $F'$ is $\geq\delta$. 
 Here, $F_+(A^{\underline{\ell}^*})$ is the subspace associated to $k$ largest exponents of
  $A^{\underline{\ell}^*}$. Indeed, it is sufficient to prove this for a given admissible integer $k$, 
  in which case it follows from the twisting assumption and the compactness of $G(d-k)$.


\begin{lemma}[Lemma A.6 of \cite{AVKZ}]\label{l.a-1} Let $\varepsilon>0$ and $\rho$ a probability measure on $G(k)$. There exists $n_0=n_0(\rho,\varepsilon)$ and, for each $\widetilde{\underline{\ell}}\in\Omega$, there exists $i=i(\widetilde{\underline{\ell}})\in\{1,\dots,m\}$ such that, for $n\geq n_0$, we have
$$A^{\underline{\ell}}_*(\rho)(B)>1-\varepsilon$$
where $\underline{\ell}:=(\underline{\ell}^*)^n\underline{\ell}^0(\underline{\ell}^*)^n\underline{\ell}_i\widetilde{\underline{\ell}}$
and $B$ is the ball of radius $\varepsilon>0$ centered at $\xi_{\underline{\ell}}$.
\end{lemma}

This lemma is harder to state than to explain: geometrically, it says that, although the word 
$\widetilde{\underline{\ell}}$ may be very long, we can choose an appropriate ``start'' ($(\underline{\ell}^*)^n\underline{\ell}^0(\underline{\ell}^*)^n\underline{\ell}_i$) so that the word $\underline{\ell}$ obtained by the concatenation of $(\underline{\ell}^*)^n\underline{\ell}^0(\underline{\ell}^*)^n\underline{\ell}_i$ and $\widetilde{\underline{\ell}}$ has the property that $A^{\underline{\ell}}$ concentrates  most of the mass of any probability measure $\rho$ on $G(k)$ (given in advance) in a tiny ball $B$.

\par
We defer the proof of the lemma to the end of the appendix and first end the proof of Proposition \ref{p.a-3}. 
The details are slightly different from \cite{AVKZ}.
\par
We will apply Lemma \ref{l.a-1} with $\rho = \nu$. As $\nu_n(x):=A^{\underline{\ell}(x,n)}_*\nu$, 
the conclusion of the lemma will imply the conclusion of Proposition \ref{p.a-3} if we can show that, for any 
$n\geq 0$ and  $\mu_-$-almost every $x$,  there are infinitely many integers $N$ such that
  $(\underline{\ell}^*)^n \underline{\ell}^0(\underline{\ell}^*)^n\underline{\ell}_i\underline{\ell}(x,N)$, with $i=i(\underline{\ell}(x,N))$, is a terminal word of $x$.
\par
Assume that this is not true. Then there exist  integers $n, N_0$ and a positive measure set 
$E \subset \Sigma_-$ such that, for any $x \in E$ , $N \geq N_0$,  the word 
$(\underline{\ell}^*)^n \underline{\ell}^0(\underline{\ell}^*)^n\underline{\ell}_i \underline{\ell}(x,N)$, with $i=i(\underline{\ell}(x,N))$,
is not a terminal word of $x$. 
\par
By the bounded distortion property, there exists $c>0$ such that
\begin{equation} \label{eq.bd}
\mu_-(\Sigma_-((\underline{\ell}^*)^n \underline{\ell}^0(\underline{\ell}^*)^n\underline{\ell}_i 
\widetilde{\underline{\ell}}))\geq c \mu_-(\Sigma_-(\widetilde{\underline{\ell}}))
\end{equation}
for every $1\leq i\leq m$ and $\widetilde{\underline{\ell}}\in\Omega$.
\par
Let $x_0$ be a density point of $E$. There exists $N \geq N_0$ such that
$$\mu_-(\Sigma_-(\underline{\ell}(x_0,N))\cap E^c) < \frac c2 \mu_-(\Sigma_-(\underline{\ell}(x_0,N))),$$
where $E^c$ is the complement of $E$. 
\par
Taking $\widetilde{\underline{\ell}} = \underline{\ell}(x_0,N)$, $i = i(\widetilde{\underline{\ell}})$ in (\ref {eq.bd})
above,
we find a point in $E$ with terminal word\linebreak $ (\underline{\ell}^*)^n \underline{\ell}^0(\underline{\ell}^*)^n\underline{\ell}_i \underline{\ell}(x_0,N)$. This  contradiction to the definition of $E$ proves the claim and ends the proof of Proposition \ref{p.a-3}. 
\end{proof}

\begin{proof}[Proof of Lemma \ref{l.a-1}] 
An elementary calculation shows that, as   $A^{\underline{\ell}^*}$ is pinching, the sequence 
$(A^{\underline{\ell}^*})^n(\xi)$ converges for every $\xi \in G(k)$. The limit is one of the finitely many 
$A^{\underline{\ell}^*}$-invariant subspaces in $G(k)$. Moreover, the limit is the subspace 
$F_+(A^{\underline{\ell}^*})$ associated to the 
$k$ largest exponents whenever $\xi$ is transverse to every $A^{\underline{\ell}^*}$-invariant subspace in $G(d-k)$.
\par
For any probability measure on $G(k)$, the sequence $(A^{\underline{\ell}^*})_*^n(\rho)$ converges,
as $n$ goes to $+ \infty$, to a limit 
which is a convex combination of Dirac masses at the $A^{\underline{\ell}^*}$-invariant subspaces in $G(k)$.
By definition of $\underline{\ell}^0$, the images under $A^{\underline{\ell}^0}$ of these invariant subspaces are transverse to every $A^{\underline{\ell}^*}$-invariant subspace in $G(d-k)$. We conclude that 
$(A^{\underline{\ell}^*})_*^n (A^{\underline{\ell}^0})_* (A^{\underline{\ell}^*})_*^n (\rho)$ converges to the Dirac mass
at $F_+(A^{\underline{\ell}^*})$. 
\par
Let $\widetilde{\underline{\ell}} \in \Omega$ be given. Denote by $\xi^*_{\widetilde{\underline{\ell}}}$ 
the $(d-k)$-dimensional subspace which is least dilated by $A^{\widetilde{\underline{\ell}}}$, i.e whose
image is spanned by the $(d-k)$ shortest semi-major axes of the ellipsoid $A^{\widetilde{\underline{\ell}}}
(\{||v|| =1\})$. Taking $F' = \xi^*_{\widetilde{\underline{\ell}}}$ in the defining property of $\ell_1,\ldots,\ell_m$,
we find $i$ such that $A^{\underline \ell_i} ( F_+(A^{\underline{\ell}^*}))$ is transverse to 
$\xi^*_{\widetilde{\underline{\ell}}}$, the angle between these subspaces being $\geq \delta$.
From the claim below, we conclude that for large $n$ (independently of $\widetilde{\underline{\ell}}$) , most of the mass of the probability measure
$$(A^{\widetilde{\underline{\ell}}})_*  (A^{\underline \ell_i})_*  (A^{\underline{\ell}^*})_*^n 
(A^{\underline{\ell}^0})_*(A^{\underline{\ell}^*})_*^n (\rho)$$
is concentrated in a small ball in $G(k)$ around 
$A^{\widetilde{\underline{\ell}}} A^{\underline \ell_i}  ( F_+(A^{\underline{\ell}^*}))$. 
Considering the case where $\rho$ is the Lebesgue measure on $G(k)$ , we conclude that this small ball is contained in a small ball around $\xi_{\underline\ell}$, where  $\underline{\ell}:=(\underline{\ell}^*)^n\underline{\ell}^0(\underline{\ell}^*)^n\underline{\ell}_i\widetilde{\underline{\ell}}$.
\end{proof}

\smallskip

{\bf Claim:} Let $A \in \GL_d(\Kset)$ act on the Grassmannian of $k$-dimensional subspaces. Denote by
 $\xi^*_A$ the $(d-k)$ dimensional subspace which is least dilated by $A$, and by $K_{\delta}(A) $ the set of 
$k$-dimensional subspaces which form an angle $\geq \delta$ with $\xi^*_A$. Then the modulus of continuity of the restriction of $A$ to $K_{\delta}(A) $ is controlled by $\delta$ only, independently of $A$.

\begin{proof}
This is an elementary computation: write any subspace in $K_{\delta}(A) $ as the graph of a linear map from
  $(\xi^*_A)^{\perp}$ to $\xi^*_A$, whose norm  is bounded in terms of $\delta$. After composing if necessary by an isometry, the action of $A$ on the matrix of this linear map is given by the multiplication of  each coefficient by a number $\in (0, 1)$ (the ratio of two singular values of $A$). 
 \end{proof}

\section{Twisting properties}\label{a.twisting}

In Appendix \ref{a.AVcriterion} above, we studied a version of Avila-Viana simplicity criterion in the context of locally constant cocycles with values on
\begin{equation}\label{e.g-glspu}
\mathbb G=\GL(d,\mathbb{R}), \Sp(d,\mathbb{R}), U_{\mathbb{K}}(p,q), \Kset = \Rset,\, \Cset,  \textrm{ or } \Hset
\end{equation}
over shifts on at most countably many symbols. 

In this way, based on the setting of Section \ref{sec:cfalgo} above, we can \emph{already} get a simplicity 
criterion for the Kontsevich-Zorich cocycle over $\SL_2(\mathbb{R})$-orbits of square-tiled surfaces based on
 pinching and a strong form of twisting. However, such a simplicity criterion is not easy to apply directly, 
 so it is desirable to replace the strong form of twisting by the relative form, with respect to some pinching 
 matrix. This lead us to the statement of Proposition \ref{p.twist/strongtwist} whose proof is the main purpose 
 of this appendix. But, before explaining the proof of Proposition \ref{p.twist/strongtwist}, it is convenient to 
 revisit a little bit the features of Noetherian topological spaces.
 
\subsection{Noetherian spaces}\label{ss.noether}

Let $\Gset$ be as in \eqref{e.g-glspu}. We  use  the notations and definitions introduced in 
Subsection \ref{ss.AVsetting}. Let $k$ be an admissible integer. For each $F'\in G(d-k)$, we define 
an hyperplane section as
$$\{F\in G(k): F\cap F'\neq\{0\}\}$$
 We consider then the coarsest topology on $G(k)$  such that the hyperplane sections are closed. 
 The closed sets are the (arbitrary) intersections of finite unions of hyperplane sections. For sake of 
 convenience, we will refer to this ``pseudo-Zariski'' topology as the \emph{Schubert topology}.

Notice that the Schubert topology is coarser than the Zariski topology:  hyperplane sections are defined by degree one (linear) equations while Zariski topology involves taking equations of arbitrary degree. In particular, this topology is not Hausdorff as the same is true for the Zariski topology. 
\begin{definition}\label{defNoether}
 A topological space $X$ is \emph{Noetherian} if one of the following equivalent conditions is satisfied:
\begin{itemize}
\item[(i)] any decreasing sequence $F_1\supset F_2\supset\dots$ of closed sets is \emph{stationary} (in the sense that there exists $m\in\mathbb{N}$ such that $F_i=F_m$ for all $i\geq m$).
\item[(ii)] any increasing sequence of open sets is stationary.
\item[(iii)] every intersection of a family $(F_{\alpha})$ of closed sets is the intersection of a finite subfamily $F_1,\dots, F_m$.
\item[(iv)] every union of a family $(U_{\alpha})$ of open sets is the union of a finite subfamily $U_1,\dots,U_m$.
\end{itemize}
\end{definition}
Observe that any subspace of a Noetherian space is also Noetherian. A topology which is coarser than a Noetherian topology is also Noetherian.
\begin{example}\label{ex.noether1} It is a classical fact that the Zariski topology is Noetherian. Therefore the Schubert topology is also Noetherian.
\end{example}
\begin{definition} A Noetherian (topological) space $X$ is irreducible if $X$ is not the union of two proper closed sets. 
\end{definition}

\smallskip
The Grassmannian $G(k)$, equipped with the Zariski topology, is irreducible. It is a fortiori irreducible when it is equipped with the coarser Schubert topology.
\par
We will need the following properties of Noetherian spaces. 
\begin{proposition}[Proposition 1.5 in \cite{Hartshorne}]\label{p.noether-3}A Noetherian space $X$ can be written  as a finite union $X=X_1\cup\dots\cup X_m$ of irreducible closed subsets $X_i$, 
$1\leq i\leq m$. Moreover, this decomposition is unique (up to a permutation of the $X_i$'s) if we ask that $X_i\not\subset X_j$ for $i\neq j$. 
\end{proposition}
\begin{proposition}\label{p.noether-4}
Let $X_1, \ldots, X_n$ be Noetherian spaces.
\begin{itemize}
\item [(i)] The product space $X= X_1 \times \ldots  \times X_n$ is Noetherian.
\item [(ii)] It is irreducible iff each $X_i$ is irreducible.
\item [(iii)] Open subsets of $X$ are exactly the finite unions of products of open subsets of the $X_i$.
\item [(iv)]  Closed subsets of $X$ are exactly the finite unions of products of closed subsets of the $X_i$.
\item  [(v)] A closed subset of $X$ is irreducible iff it is the product of closed irreducible subsets of the $X_i$.
\end{itemize}
\end{proposition}
\begin{proof}
The first assertion is Exercise~8,  p.~142 of \cite{Bourbaki-CA}. 
Item $(iii)$ is an immediate consequence
of item $(i)$ and Definition~\ref{defNoether}, item $(iv)$. Then item $(iv)$ follows from some Boolean manipulations. From item $(iv)$, it follows that a closed irreducible subset of $X$ is the product of closed subsets of the $X_i$. It is also clear that if $X$ is irreducible, then each $X_i$ must be irreducible.
Finally we show that a product of irreducible spaces is irreducible. Let 
$$ X=F_1\cup\dots\cup F_m$$
be the minimal decomposition of $X$ into irreducible subsets. Each $F_j$ is a product
$$F_j=F_j^{(1)}\times\dots\times F_j^{(n)}$$
where each $F_j^{(i)}$ is an  irreducible closed subset of $X_i$. For each $1 \leq i \leq n$, define
$$F^{(i)} := \bigcup\limits_{F_j^{(i)}\neq X} F_j^{(i)}.$$
As $X_i $ is irreducible, one has $F^{(i)} = X_i$ iff  $F_j^{(i)} = X_i$ for all $1 \leq j \leq m$. 
For $1 \leq i \leq n$, choose  $x_i \in X_i - F^{(i)}$ if $F^{(i)} \ne X_i$  and $x_i \in X_i$   otherwise. Let $j$ be an index such that $x:=(x_1, \ldots ,x_n) \in F_j$. One must have $F_j^{(i)} = X_i$ for all $1 \leq i \leq n$, hence 
$F_j = X$. \end{proof}

\subsection{Twisting monoids}\label{ss.twist-monoid}

Let $\mathcal{M}$ be a monoid acting on a Noetherian space $X$ by homeomorphisms. Here, of course, our main example is:
\begin{example}\label{ex.noether-2} Given a countable family of matrices $A_{\ell}\in \Gset$, $\ell\in\Lambda$, 
 we  consider the natural action of the monoid $\mathcal{M}$ generated by $A_{\ell}$ acting on the Grassmanian $X_k=G(k)$ equipped with the Schubert topology. 
\end{example}
\begin{proposition}\label{p.noether-6} If $g\in\mathcal{M}$, $F\subset X$ is closed and $gF\subset F$, then $gF=F$. 
\end{proposition}
\begin{proof} Otherwise, $(g^nF)_{n\geq 0}$ would be a strictly decreasing infinite sequence of closed subsets of the Noetherian space $X$. 
\end{proof}

\begin{proposition}\label{p.noether-7} Let $g\in\mathcal{M}$, let $F=F_1\cup\dots\cup F_n$ be the (minimal) decomposition of the closed subset $F\subset X$ into irreducible 
closed subsets $F_i\subset X$. If  $gF=F$, then $g$ permutes the irreducible pieces $F_i$. 
\end{proposition}
\begin{proof} 
This follows from the uniqueness part of Proposition \ref{p.noether-3}. 
\end{proof}

\begin{proposition}\label{p.noether-8} Assume that $X$ is irreducible. Then, the following properties are equivalent:
\begin{itemize}
\item[(i)] there exists no proper closed $\mathcal{M}$-invariant subset of $X$;
\item[(ii)] for every $x\in X$ and every non empty open subset $ U\subset X$, there exists $g\in\mathcal{M}$ such that $gx\in U$;
\item[(iii)] for every $N\geq 1$, $x_1,\dots, x_N\in X$ and every non empty open subsets $ U_1,\dots,U_N\subset X$, there exists $g\in\mathcal{M}$ such that $gx_i\in U_i$ for all $1\leq i\leq N$.
\end{itemize}
\end{proposition}
\begin{proof} 
It is clear that $(iii) \implies (ii)\implies (i)$. We will  prove by contradiction that $(i)$ implies $(iii)$.
We let $\mathcal{M}$ act diagonally on $X^N$, for any $N \geq 1$.
\par
Suppose that there exist $N \geq 1$,  $x= (x_1,\dots,x_N)\in X^N$ and  non empty open subsets 
 $ U_1,\dots,U_N\subset X$ such that

$$ g.x \notin U^{(N)} := U_1\times\dots\times U_N$$
for all $g\in\mathcal{M}$.

\par
Consider the closed set

$$ F:= \bigcap\limits_{g\in\mathcal{M}}g^{-1}(X^N - U^{(N)}).$$

It is distinct from $X^N$ and non empty because it contains $x$. It satisfies $gF\subset F$ for all 
$g\in\mathcal{M}$, hence $gF = F$ for all $g\in\mathcal{M}$ (Proposition \ref{p.noether-6}).


Let $F=F_1\cup\dots\cup F_m$ be the decomposition  of $F$ into irreducible closed sets $F_i\subset X^N$. By Proposition \ref{p.noether-4}, one has
$$F_i=F_i^{(1)}\times\dots\times F_i^{(N)}$$
where each $F_i^{(l)}$ is an  irreducible closed subset of $X$. 
\par
Since $gF=F$ for all $g\in\mathcal{M}$,  Proposition \ref{p.noether-7} implies that  every $g\in\mathcal{M}$ permutes the  subsets $F_i^{(l)}$.

Define the closed subset
$$ F^*:=\bigcup\limits_{F_i^{(l)}\neq X} F_i^{(l)}.$$
As $\emptyset\neq F\neq X^N$, the subset $ F^*$ is not empty. As  $X$ is irreducible, one has $ F^* \ne X$.
As every $g\in\mathcal{M}$ permutes the  subsets $F_i^{(l)}$, one has $gF^*=F^*$ for every $g\in\mathcal{M}$,
so item $(i)$ does not hold.
\end{proof}

In view of Example \ref{ex.noether-2} and the discussion in Subsection \ref{ss.AVsetting} (related to Avila-Viana simplicity criterion), it is natural to call $(iii)$ a (``strong form'' of) \emph{twisting condition} for an abstract monoid $\mathcal{M}$ acting by homeomorphisms on a Noetherian space $X$.

\begin{remark}\label{r.noether-2} The equivalent conditions of the proposition are satisfied by the monoid $\mathcal{M}$ if and only if they are  satisfied by the 
group $\mathcal{G}=\langle g, g^{-1}:g\in\mathcal{M}\rangle$ generated by $\mathcal{M}$:  this follows immediately from the statement of item $(i)$.
\end{remark} 

\subsection{Twisting with respect to pinching matrices}

In the context of Example \ref{ex.noether-2} and aiming to the proof of Proposition \ref{p.twist/strongtwist}, 
consider a word $\underline{\ell}^*\in\Omega=\bigcup\limits_{n\geq 0}\Lambda^n$ such that $A_*:=A^{\underline{\ell}^*}$ has simple spectrum (i.e., the cocycle is pinching). In this notation, Proposition \ref{p.twist/strongtwist} can be restated as:

\begin{proposition}\label{p.noether-9}The (strong form of the) twisting condition is realized for the cocycle $A$ if and only if, for each  admissible integer $k$, there exists a word $\underline{\ell}(k)\in \Omega$ such that the matrix $B_k:=A^{\underline{\ell}(k)}$ satisfies
$$B_k(F)\cap F'=\{0\}$$
for every $A_*$-invariant subspaces $F\in G(k)$ and $F'\in G(d-k)$.
\end{proposition}

\begin{proof} The condition is clearly necessary. 
\par
 Conversely, assume that the condition in the proposition is 
satisfied. Let $\mathcal{M}$ denote the monoid generated by the matrices $A_\ell$, $\ell\in\Lambda$. Recall that each $G(k)$ is irreducible.
\begin{lemma}\label{last}
For each admissible integer $k$, the action of $\mathcal{M}$ on $G(k)$ satisfies the equivalent conditions of Proposition \ref {p.noether-8}.
\end{lemma}
\begin{proof}[Proof of Lemma] We check that item (ii) in Proposition  \ref {p.noether-8} is satisfied. It is sufficient 
to show that, given $F \in G(k)$ and $F'_1, \ldots, F'_m \in G(d-k)$ , there exists $C \in \mathcal M$ such that
$C(F)$ is transverse to each $F'_i$. We claim that $C = A_*^n B_k A_*^n$ is an appropriate choice if $n$ 
is large enough. Indeed, when $n$ goes to $+\infty$, the sequence $(A_*^n(F))$ converges to some 
$A_*$-invariant subspace in $G(k)$, and each sequence $(A_*^{-n}(F'_i))$ converges to some 
$A_*$-invariant subspace in $G(d-k)$. As transversality is an open property, the claim follows from the property of $B_k$.
\end{proof}

We now finish the proof of the proposition.
Consider the diagonal action of $\mathcal{M}$ on the irreducible Noetherian space $X=\prod\limits_{k \textrm{ admissible}}G(k)$. The strong form of the twisting condition will be satisfied if the action of 
$\mathcal{M}$ on $X$ satisfies item (iii) in Proposition  \ref {p.noether-8}. We check the equivalent item (i).
Let $F$ be a non-empty closed subset $F \subset X$ invariant under $\mathcal{M}$.
Let  $F=F_1\cup\dots\cup F_m$ be the minimal decomposition  of $F$ into irreducible closed sets $F_i\subset X$. By Proposition \ref{p.noether-4}, one has
$$F_i=  \prod\limits_{k \textrm{ admissible}} F_i^{(k)}$$
where each $F_i^{(k)}$ is an  irreducible closed subset of $G(k)$. Define, for each admissible integer $k$,

$$ F^{(k)}:=\bigcup\limits_{F_i^{(k)}\neq G(k)} F_i^{(k)}.$$
By Proposition \ref{p.noether-7}, the closed subset $ F^{(k)}$ is invariant under $\mathcal M$.
From Lemma \ref{last}, $F^{(k)}$ must be either empty or equal to $G(k)$ for each admissible integer $k$.
The second case cannot occur (since $G(k)$ is irreducible), and the first means that $F_i^{(k)} = G(k)$ for all $i$. We conclude that $F = X$, and the proof of the proposition is complete.
\end{proof}

\section{Completely periodic configurations in $\mathcal{H}(4)$ \\ by Samuel Leli\`evre}\label{a.Samuel}

A surface in the stratum $\cH(4)$ has one singularity of angle
$10\pi$. At this singularity, $5$ outgoing separatrices start
and $5$ incoming separatrices end (see Figure~1). 
We label the outgoing separatrices from $1$ to $5$ (see Figure~1).

\begin{longtable}{@{}>{\centering}p{0.33\textwidth}@{}>{\centering}p{0.33\textwidth}@{}p{0.33\textwidth}<{\centering}@{}}
{\scriptsize
\begin{tikzpicture}[auto,scale=1.5]
  \pgfsetbaseline{0cm}
  \fill (0,0) circle (2pt);
  \foreach \a in {1,...,5}
    {
    \draw [-triangle 45] (0,0) -- (\a*72:0.8cm);
    \draw (\a*72-72:0.8cm) -- node [swap,very near start] {$\a$} (\a*72-72:1cm);
    \draw [-open triangle 45] (\a*72-72+36:1cm) --
      (\a*72-72+36:0.6cm);
    \draw (\a*72-72+36:0.6cm) -- (0,0);
    }
    \useasboundingbox (-1.2cm,-1.2cm) rectangle (1.2cm,1.2cm);
\end{tikzpicture}%
}
\newline
Fig.~ 1. Outgoing and\newline
incoming separatrices,\newline
and a numbering of\newline
outgoing separatrices.\newline
&
{\scriptsize
\begin{tikzpicture}[auto,scale=1.15]
  \pgfsetbaseline{0cm}
  \fill (0,0) circle (2pt);
  \foreach \a in {1,...,5}
    {
    \draw [-triangle 45] (0,0) -- (\a*72-72:0.8cm);
    \draw (\a*72-72:0.8cm) -- node [swap] {$\a$} (\a*72-72:1cm);
    \draw (\a*72-72:1cm)
      .. controls +(\a*72-72:0.75cm) and +(\a*72-72+36:0.75cm) ..
      (\a*72-72+36:1cm);
    \draw [-open triangle 45] (\a*72-72+36:1cm) --
      (\a*72-72+36:0.6cm);
    \draw (\a*72-72+36:0.6cm) -- (0,0);
    }
    \useasboundingbox (-1.5cm,-1.25cm) rectangle (1.5cm,1.5cm);
\end{tikzpicture}%
}
\newline
Fig.~ 2. There is no way\newline
to glue cylinders using this\newline
pairing of separatrices.\newline
&
{\scriptsize
\begin{tikzpicture}[auto,scale=1.15]
  \pgfsetbaseline{0cm}
  \fill (0,0) circle (2pt);
  \foreach \a in {1,2}
    {
    \draw [-triangle 45] (0,0) -- (\a*72-72:0.8cm);
    \draw (\a*72-72:0.8cm) -- node [swap] {$\a$} (\a*72-72:1cm);
    \draw (\a*72-72:1cm)
      .. controls +(\a*72-72:0.75cm) and +(\a*72-72+36:0.75cm) ..
      (\a*72-72+36:1cm);
    \draw [-open triangle 45] (\a*72-72+36:1cm) --
      (\a*72-72+36:0.6cm);
    \draw (\a*72-72+36:0.6cm) -- (0,0);
    }
  \foreach \a in {4,5}
    {
    \draw [-triangle 45] (0,0) -- (\a*72-72:0.8cm);
    \draw (\a*72-72:0.8cm) -- node [swap] {$\a$} (\a*72-72:1cm);
    \draw (\a*72-72:1cm)
      .. controls +(\a*72-72:0.75cm) and +(\a*72-72-36:0.75cm) ..
      (\a*72-72-36:1cm);
    \draw [-open triangle 45] (\a*72-72-36:1cm) --
      (\a*72-72-36:0.6cm);
    \draw (\a*72-72-36:0.6cm) -- (0,0);
    }
  \foreach \a in {3}
    {
    \draw [-triangle 45] (0,0) -- (\a*72-72:0.8cm);
    \draw (\a*72-72:0.8cm) -- node [swap] {$\a$} (\a*72-72:1cm);
    \draw (\a*72-72:1cm)
      .. controls +(\a*72-72:1cm) and +(\a*72-72:1.5cm) ..
      (\a*72-72+90:1.8cm)
      .. controls +(\a*72-72:-1.5cm) and +(\a*72-72:-1cm) ..
      (\a*72-72:-1cm);
     \draw [-open triangle 45] (\a*72-72+180:1cm) --
       (\a*72-72+180:0.6cm);
    \draw (\a*72-72+180:0.6cm) -- (0,0);
    }
    \useasboundingbox (-1.75cm,-1.8cm) rectangle (1.5cm,1.5cm);
\end{tikzpicture}%
}
\vspace*{-12pt}\newline
Fig.~ 3. With this pairing\newline
of separatrices there are\newline
two ways to glue cylinders.\newline
\\
\end{longtable}
In a completely periodic direction, outgoing separatrices pair with
incoming separatrices.  Define a permutation $\sigma$ of $\{ 1, 2, 3,
4, 5 \}$ by setting $\sigma(i) = j$ if the outgoing separatrix $i$
comes back between the outgoing separatrices $j$ and $j+1 \bmod 5$.
We can enumerate permutations and draw the corresponding separatrix
diagrams. Since the separatrix at which we start the labelling is
arbitrary, we only need to enumerate permutations up to conjugation by
cyclic permutations.

\vspace{5pt}

A diagram makes sense if the ribbons which follow unions of
separatrices above or below form compatible bottoms and tops of
cylinders.  So, given a permutation, we look for the cycles of
$\sigma$ and of $\sigma'$ defined by $\sigma'(i) = \sigma(i)
+ 1 \bmod 5$ and then look for all the ways to match them.

\vspace{5pt}

A first example: $\sigma = \textrm{id}$ gives the diagram on Figure~2.
In this example, $\sigma$ has cycles $(1)(2)(3)(4)(5)$, while
$\sigma'$ has only one cycle $(12345)$, therefore there is no 
possible pairing. We can't glue any cylinders on this separatrix 
diagram, there would need to be five bottoms of cylinders and
only one top of cylinder.

\vspace{5pt}

Another example: $\sigma = (354)$ gives the diagram on Figure~3.
In this example, $\sigma$ has cycles $(1)(2)(354)$, while
$\sigma'$ has cycles $(123)(4)(5)$, and two pairings are possible:
cylinders can fit on this separatrix diagram in two different ways.

\vspace{5pt}

As checked with Sage, there are exactly 16 permutations $\sigma$ of
$\{1, 2, 3, 4, 5\}$ (up to cyclic permutation) for which $\sigma$ and
$\sigma'$ have the same number of cycles.  The corresponding pairs
$(\sigma,\sigma')$ are listed below, expressed as products of
nontrivial cycles.  It turns out that each of them gives one, two
or three cylinder diagrams, and we get 22 cylinder diagrams in all.

$\begin{array}{cccc}
(354), (123)
&
(23)(45), (124)
&
(2345), (124)(35)
&
(2354), (1243)
\\
(2453), (1254)
&
(2435), (1253)
&
(253), (12)(45)
&
(2534), (12)(354)
\\
(25)(34), (12)(35)
&
(12)(345), (1354)
&
(12345), (13524)
&
(12453), (13254)
\\
(124)(35), (13)(254)
&
(13524), (14253)
&
(13)(254), (1432)
&
(14253), (15432)
\end{array}$

\vspace{5pt}

Below, we list next to each of these 16 pairs $(\sigma,\sigma')$ the associated cylinder diagrams and, furthermore, for each cylinder diagram, we put the letter $H$, resp. $O$, when the corresponding translation surfaces belong the the hyperelliptic, resp., odd, connected component of $\mathcal{H}(4)$.



\setlength{\parindent}{0pt}


\begin{minipage}{0.24\textwidth}
\centering
\begin{tikzpicture}[auto,scale=0.75]
  \pgfsetbaseline{0cm}
  \fill (0,0) circle (2pt);
  \foreach \a in {0,1}
    {
    \draw [-triangle 45] (0,0) -- (\a*72:0.8cm);
    \draw (\a*72:0.8cm) --
      (\a*72:1cm);
    \draw (\a*72:1cm)
      .. controls +(\a*72:0.75cm) and +(\a*72+36:0.75cm) ..
      (\a*72+36:1cm);
    \draw [-open triangle 45] (\a*72+36:1cm) --
      (\a*72+36:0.6cm);
    \draw (\a*72+36:0.6cm) -- (0,0);
    }
  \foreach \a in {3,4}
    {
    \draw [-triangle 45] (0,0) -- (\a*72:0.8cm);
    \draw (\a*72:0.8cm) --
      (\a*72:1cm);
    \draw (\a*72:1cm)
      .. controls +(\a*72:0.75cm) and +(\a*72-36:0.75cm) ..
      (\a*72-36:1cm);
    \draw [-open triangle 45] (\a*72-36:1cm) --
      (\a*72-36:0.6cm);
    \draw (\a*72-36:0.6cm) -- (0,0);
    }
  \foreach \a in {2}
    {
    \draw [-triangle 45] (0,0) -- (\a*72:0.8cm);
    \draw (\a*72:0.8cm) --
      (\a*72:1cm);
    \draw (\a*72:1cm)
      .. controls +(\a*72:1cm) and +(\a*72:1.5cm) ..
      (\a*72+90:1.8cm)
      .. controls +(\a*72:-1.5cm) and +(\a*72:-1cm) ..
      (\a*72:-1cm);
     \draw [-open triangle 45] (\a*72+180:1cm) --
       (\a*72+180:0.6cm);
    \draw (\a*72+180:0.6cm) -- (0,0);
    }
    \useasboundingbox (-1.75cm,-1.8cm) rectangle (1.5cm,1.5cm);
\end{tikzpicture}%
\newline
$\sigma: (354)$ \newline $\sigma': (123)$\newline
$b\!: 1, 2, 354$ \newline $t\!: 123, 4, 5$
\end{minipage}
\hfill
\begin{minipage}{0.24\textwidth}
{\scriptsize
\begin{tikzpicture}[auto,scale=0.7]
  \node at (3,1.5) {\bfseries\sffamily H};
  \draw (0,0)
    -- node {3} ++(1,0)
    -- node {5} ++(1,0)
    -- node {4} ++(1,0)
    -- ++ (0,1)
    -- node {3} ++(-1,0)
    -- node {2} node [swap] {2} ++(-1,0)
    -- node {1} node [swap] {1} ++(-1,0)
    -- cycle;
  \draw (0,0) ++(0,1) -- ++(-0.3,1) -- node [swap] {4} ++(1,0)
    -- ++(0.3,-1);
  \draw (0,0) ++(0,1) ++(1,0) -- ++(0.3,1) -- node [swap] {5} ++(1,0)
    -- ++(-0.3,-1);
  \fill (0,0) circle (2pt);
  \fill (1,0) circle (2pt);
  \fill (2,0) circle (2pt);
  \fill (3,0) circle (2pt);
  \fill (3,1) circle (2pt);
  \fill (2,1) circle (2pt);
  \fill (2.3,2) circle (2pt);
  \fill (1.3,2) circle (2pt);
  \fill (1,1) circle (2pt);
  \fill (0.7,2) circle (2pt);
  \fill (-0.3,2) circle (2pt);
  \fill (0,1) circle (2pt);
\end{tikzpicture}}\newline
1-4, 2-5, 354-123\par\medskip
{\scriptsize
\begin{tikzpicture}[auto,scale=0.7]
  \node at (3,1.5) {\bfseries\sffamily O};
  \draw (0,0)
    -- node {3} ++(1,0)
    -- node {5} ++(1,0)
    -- node {4} ++(1,0)
    -- ++ (0,1)
    -- node {3} ++(-1,0)
    -- node {2} node [swap] {2} ++(-1,0)
    -- node {1} node [swap] {1} ++(-1,0)
    -- cycle;
  \draw (0,0) ++(0,1) -- ++(-0.3,1) -- node [swap] {5} ++(1,0)
    -- ++(0.3,-1);
  \draw (0,0) ++(0,1) ++(1,0) -- ++(0.3,1) -- node [swap] {4} ++(1,0)
    -- ++(-0.3,-1);
  \fill (0,0) circle (2pt);
  \fill (1,0) circle (2pt);
  \fill (2,0) circle (2pt);
  \fill (3,0) circle (2pt);
  \fill (3,1) circle (2pt);
  \fill (2,1) circle (2pt);
  \fill (2.3,2) circle (2pt);
  \fill (1.3,2) circle (2pt);
  \fill (1,1) circle (2pt);
  \fill (0.7,2) circle (2pt);
  \fill (-0.3,2) circle (2pt);
  \fill (0,1) circle (2pt);
\end{tikzpicture}}\newline
1-5, 2-4, 354-123\newline
\end{minipage}
\hfill
\begin{minipage}{0.24\textwidth}
\centering
\begin{tikzpicture}[auto,scale=0.65]
  \pgfsetbaseline{0cm}
  \fill (0,0) circle (2pt);
  \foreach \a in {0}
    {
    \draw [-triangle 45] (0,0) -- (\a*72:0.8cm);
    \draw (\a*72:0.8cm) --
      (\a*72:1cm);
    \draw (\a*72:1cm)
      .. controls +(\a*72:0.75cm) and +(\a*72+36:0.75cm) ..
      (\a*72+36:1cm);
    \draw [-open triangle 45] (\a*72+36:1cm) --
      (\a*72+36:0.6cm);
    \draw (\a*72+36:0.6cm) -- (0,0);
    }
  \foreach \a in {2,4}
    {
    \draw [-triangle 45] (0,0) -- (\a*72:0.8cm);
    \draw (\a*72:0.8cm) --
      (\a*72:1cm);
    \draw (\a*72:1cm)
      .. controls +(\a*72:0.75cm) and +(\a*72-36:0.75cm) ..
      (\a*72-36:1cm);
    \draw [-open triangle 45] (\a*72-36:1cm) --
      (\a*72-36:0.6cm);
    \draw (\a*72-36:0.6cm) -- (0,0);
    }
  \foreach \a in {1,3}
    {
    \draw [-triangle 45] (0,0) -- (\a*72:0.8cm);
    \draw (\a*72:0.8cm) --
      (\a*72:0.9cm);
    \draw (\a*72:0.9cm)
      .. controls +(\a*72:0.75cm) and +(\a*72+54-90:1cm) ..
      (\a*72+54:2cm)
      .. controls +(\a*72+54+90:1cm) and +(\a*72+108:0.75cm) ..
      (\a*72+108:0.9cm);
     \draw [-open triangle 45] (\a*72+108:0.9cm) --
       (\a*72+108:0.6cm);
    \draw (\a*72+108:0.6cm) -- (0,0);
    }
    \useasboundingbox (-1.75cm,-2cm) rectangle (1.5cm,1.85cm);
\end{tikzpicture}%
\par\vspace*{0.25cm}
$\sigma$: $(23)(45)$ \newline $\sigma'$: $(124)$\newline
$b$: $1, 23, 45$ \newline $t\!: 124, 3, 5$
\end{minipage}
\hfill
\begin{minipage}{0.24\textwidth}
{\scriptsize
\begin{tikzpicture}[auto,scale=0.6]
  \node at (3,1.5) {\bfseries\sffamily O};
  \draw (0,0)
    -- node {4} ++(1,0)
    -- node {5} node [swap] {5} ++(2,0)
    -- ++ (0,1)
    -- node {4} ++(-1,0)
    -- node {2} ++(-1,0)
    -- node {1} node [swap] {1} ++(-1,0)
    -- cycle;
  \draw (0,0) ++(0,1)
    -- ++(0,1)
    -- node [swap] {3} ++(1,0)
    -- ++(0,-1);
  \draw (0,0) ++(1,0)
    -- ++(0,-1)
    -- node {2} ++(1,0)
    -- node {3} ++(1,0)
    -- ++(0,1);
  \fill (0,0) circle (2pt);
  \fill (1,0) circle (2pt);
  \fill (3,0) circle (2pt);
  \fill (3,1) circle (2pt);
  \fill (2,1) circle (2pt);
  \fill (1,1) circle (2pt);
  \fill (1,2) circle (2pt);
  \fill (0,2) circle (2pt);
  \fill (0,1) circle (2pt);
  \fill (1,-1) circle (2pt);
  \fill (2,-1) circle (2pt);
  \fill (3,-1) circle (2pt);
\end{tikzpicture}}\newline
1-3, 23-5, 45-124\par\medskip
{\scriptsize
\begin{tikzpicture}[auto,scale=0.6]
  \node at (3,1.5) {\bfseries\sffamily O};
  \draw (0,0)
    -- node {2} ++(1,0)
    -- node {3} node [swap] {3} ++(2,0)
    -- ++ (0,1)
    -- node {4} ++(-1,0)
    -- node {2} ++(-1,0)
    -- node {1} node [swap] {1} ++(-1,0)
    -- cycle;
  \draw (0,0) ++(0,1)
    -- ++(0,1)
    -- node [swap] {5} ++(1,0)
    -- ++(0,-1);
  \draw (0,0) ++(1,0)
    -- ++(0,-1)
    -- node {4} ++(1,0)
    -- node {5} ++(1,0)
    -- ++(0,1);
  \fill (0,0) circle (2pt);
  \fill (1,0) circle (2pt);
  \fill (3,0) circle (2pt);
  \fill (3,1) circle (2pt);
  \fill (2,1) circle (2pt);
  \fill (1,1) circle (2pt);
  \fill (1,2) circle (2pt);
  \fill (0,2) circle (2pt);
  \fill (0,1) circle (2pt);
  \fill (1,-1) circle (2pt);
  \fill (2,-1) circle (2pt);
  \fill (3,-1) circle (2pt);
\end{tikzpicture}}\newline
1-5, 23-124, 45-3\newline
\end{minipage}
\hfill
\begin{minipage}{0.24\textwidth}
\centering
\begin{tikzpicture}[auto,scale=0.7]
  \pgfsetbaseline{0cm}
  \fill (0,0) circle (2pt);
  \foreach \a in {0} 
    {
    \draw [-triangle 45] (0,0) -- (\a*72:0.8cm);
    \draw (\a*72:0.8cm) --
      (\a*72:1cm);
    \draw (\a*72:1cm)
      .. controls +(\a*72:0.75cm) and +(\a*72+36:0.75cm) ..
      (\a*72+36:1cm);
    \draw [-open triangle 45] (\a*72+36:1cm) --
      (\a*72+36:0.6cm);
    \draw (\a*72+36:0.6cm) -- (0,0);
    }
  \foreach \a in {1,2,3}
    {
    \draw [-triangle 45] (0,0) -- (\a*72:0.8cm);
    \draw (\a*72:0.8cm) --
      (\a*72:0.9cm);
    \draw (\a*72:0.9cm)
      .. controls +(\a*72:0.75cm) and +(\a*72+54-90:1cm) ..
      (\a*72+54:2cm)
      .. controls +(\a*72+54+90:1cm) and +(\a*72+108:0.75cm) ..
      (\a*72+108:0.9cm);
     \draw [-open triangle 45] (\a*72+108:0.9cm) --
       (\a*72+108:0.6cm);
    \draw (\a*72+108:0.6cm) -- (0,0);
    }
    \useasboundingbox (-1.75cm,-2cm) rectangle (1.5cm,1.85cm);
  \foreach \a in {4}
    {
    \draw [-triangle 45] (0,0) -- (\a*72:0.8cm);
    \draw (\a*72:0.8cm) --
      (\a*72:1cm);
    \draw (\a*72:1cm)
      .. controls +(\a*72:1cm) and +(\a*72:1.5cm) ..
      (\a*72+90:1.8cm)
      .. controls +(\a*72:-1.5cm) and +(\a*72:-1cm) ..
      (\a*72:-1cm);
     \draw [-open triangle 45] (\a*72+180:1cm) --
       (\a*72+180:0.6cm);
    \draw (\a*72+180:0.6cm) -- (0,0);
    }
\end{tikzpicture}%
\par\vspace*{0.25cm}
$\sigma$: $(2345)$ \newline $\sigma'$: $(124)(35)$\newline
$b\!: 1, 2345$ \newline $t\!: 124, 35$\newline
\end{minipage}
\hfill
\begin{minipage}{0.24\textwidth}
{\scriptsize
\begin{tikzpicture}[auto,scale=0.7]
  \node at (3,1.5) {\bfseries\sffamily O};
  \draw (0,0)
    -- node {2} ++(1,0)
    -- node {3} ++(1,0)
    -- node {4} ++(1,0)
    -- node {5} ++(1,0)
    -- ++ (0,1)
    -- node {4} ++(-1,0)
    -- node {2} ++(-1,0)
    -- node {1} node [swap] {1} ++(-2,0)
    -- cycle;
  \draw (0,0) ++(0,1)
    -- ++(0,1)
    -- node [swap] {3} ++(1,0)
    -- node [swap] {5} ++(1,0)
    -- ++(0,-1);
  \fill (0,0) circle (2pt);
  \fill (1,0) circle (2pt);
  \fill (2,0) circle (2pt);
  \fill (3,0) circle (2pt);
  \fill (4,0) circle (2pt);
  \fill (4,1) circle (2pt);
  \fill (3,1) circle (2pt);
  \fill (2,1) circle (2pt);
  \fill (2,2) circle (2pt);
  \fill (1,2) circle (2pt);
  \fill (0,2) circle (2pt);
  \fill (0,1) circle (2pt);
\end{tikzpicture}}\newline
1-35, 2345-124\newline
\end{minipage}
\hfill
\begin{minipage}{0.24\textwidth}
\centering
\begin{tikzpicture}[auto,scale=0.7]
  \pgfsetbaseline{0cm}
  \fill (0,0) circle (2pt);
  \foreach \a in {0}
    {
    \draw [-triangle 45] (0,0) -- (\a*72:0.8cm);
    \draw (\a*72:0.8cm) --
      (\a*72:1cm);
    \draw (\a*72:1cm)
      .. controls +(\a*72:0.75cm) and +(\a*72+36:0.75cm) ..
      (\a*72+36:1cm);
    \draw [-open triangle 45] (\a*72+36:1cm) --
      (\a*72+36:0.6cm);
    \draw (\a*72+36:0.6cm) -- (0,0);
    }
  \foreach \a in {4}
    {
    \draw [-triangle 45] (0,0) -- (\a*72:0.8cm);
    \draw (\a*72:0.8cm) --
      (\a*72:1cm);
    \draw (\a*72:1cm)
      .. controls +(\a*72:0.75cm) and +(\a*72-36:0.75cm) ..
      (\a*72-36:1cm);
    \draw [-open triangle 45] (\a*72-36:1cm) --
      (\a*72-36:0.6cm);
    \draw (\a*72-36:0.6cm) -- (0,0);
    }
  \foreach \a in {1}
    {
    \draw [-triangle 45] (0,0) -- (\a*72:0.8cm);
    \draw (\a*72:0.8cm) --
      (\a*72:0.9cm);
    \draw (\a*72:0.9cm)
      .. controls +(\a*72:0.75cm) and +(\a*72+54-90:1cm) ..
      (\a*72+54:2cm)
      .. controls +(\a*72+54+90:1cm) and +(\a*72+108:0.75cm) ..
      (\a*72+108:0.9cm);
     \draw [-open triangle 45] (\a*72+108:0.9cm) --
       (\a*72+108:0.6cm);
    \draw (\a*72+108:0.6cm) -- (0,0);
    }
  \foreach \a in {3}
    {
    \draw [-triangle 45] (0,0) -- (\a*72:0.8cm);
    \draw (\a*72:0.8cm) --
      (\a*72:0.9cm);
    \draw (\a*72:0.9cm)
      .. controls +(\a*72:0.75cm) and +(\a*72-54+90:1cm) ..
      (\a*72-54:2cm)
      .. controls +(\a*72-54-90:1cm) and +(\a*72-108:0.75cm) ..
      (\a*72-108:0.9cm);
     \draw [-open triangle 45] (\a*72-108:0.9cm) --
       (\a*72-108:0.6cm);
    \draw (\a*72-108:0.6cm) -- (0,0);
    }
  \foreach \a in {2}
    {
    \draw [-triangle 45] (0,0) -- (\a*72:0.8cm);
    \draw (\a*72:0.8cm) --
      (\a*72:1cm);
    \draw (\a*72:1cm)
      .. controls +(\a*72:1cm) and +(\a*72:1.5cm) ..
      (\a*72+90:1.8cm)
      .. controls +(\a*72:-1.5cm) and +(\a*72:-1cm) ..
      (\a*72:-1cm);
     \draw [-open triangle 45] (\a*72+180:1cm) --
       (\a*72+180:0.6cm);
    \draw (\a*72+180:0.6cm) -- (0,0);
    }
    \useasboundingbox (-1.75cm,-2cm) rectangle (1.5cm,1.85cm);
\end{tikzpicture}%
\newline
$\sigma$: $(2354)$ \newline $\sigma'$: $(1243)$\newline
$b\!: 1, 2354$ \newline $t\!: 1243, 5$\newline
\end{minipage}
\hfill
\begin{minipage}{0.24\textwidth}
{\scriptsize
\begin{tikzpicture}[auto,scale=0.7]
  \node at (3,1.5) {\bfseries\sffamily O};
  \draw (0,0)
    -- node {2} ++(1,0)
    -- node {3} ++(1,0)
    -- node {5} ++(1,0)
    -- node {4} ++(1,0)
    -- ++ (0,1)
    -- node {3} ++(-1,0)
    -- node {4} ++(-1,0)
    -- node {2} ++(-1,0)
    -- node {1} node [swap] {1} ++(-1,0)
    -- cycle;
  \draw (0,0) ++(0,1)
    -- ++(0,1)
    -- node [swap] {5} ++(1,0)
    -- ++(0,-1);
  \fill (0,0) circle (2pt);
  \fill (1,0) circle (2pt);
  \fill (2,0) circle (2pt);
  \fill (3,0) circle (2pt);
  \fill (4,0) circle (2pt);
  \fill (4,1) circle (2pt);
  \fill (3,1) circle (2pt);
  \fill (2,1) circle (2pt);
  \fill (1,1) circle (2pt);
  \fill (1,2) circle (2pt);
  \fill (0,2) circle (2pt);
  \fill (0,1) circle (2pt);
\end{tikzpicture}}\newline
1-5, 2354-1243\newline
\end{minipage}
\hfill
\begin{minipage}{0.24\textwidth}
\centering
\begin{tikzpicture}[auto,scale=0.65]
  \pgfsetbaseline{0cm}
  \fill (0,0) circle (2pt);
  \foreach \a in {0}
    {
    \draw [-triangle 45] (0,0) -- (\a*72:0.8cm);
    \draw (\a*72:0.8cm) --
      (\a*72:1cm);
    \draw (\a*72:1cm)
      .. controls +(\a*72:0.75cm) and +(\a*72+36:0.75cm) ..
      (\a*72+36:1cm);
    \draw [-open triangle 45] (\a*72+36:1cm) --
      (\a*72+36:0.6cm);
    \draw (\a*72+36:0.6cm) -- (0,0);
    }
  \foreach \a in {2}
    {
    \draw [-triangle 45] (0,0) -- (\a*72:0.8cm);
    \draw (\a*72:0.8cm) --
      (\a*72:1cm);
    \draw (\a*72:1cm)
      .. controls +(\a*72:0.75cm) and +(\a*72-36:0.75cm) ..
      (\a*72-36:1cm);
    \draw [-open triangle 45] (\a*72-36:1cm) --
      (\a*72-36:0.6cm);
    \draw (\a*72-36:0.6cm) -- (0,0);
    }
  \foreach \a in {3}
    {
    \draw [-triangle 45] (0,0) -- (\a*72:0.8cm);
    \draw (\a*72:0.8cm) --
      (\a*72:0.9cm);
    \draw (\a*72:0.9cm)
      .. controls +(\a*72:0.75cm) and +(\a*72+54-90:1cm) ..
      (\a*72+54:2cm)
      .. controls +(\a*72+54+90:1cm) and +(\a*72+108:0.75cm) ..
      (\a*72+108:0.9cm);
     \draw [-open triangle 45] (\a*72+108:0.9cm) --
       (\a*72+108:0.6cm);
    \draw (\a*72+108:0.6cm) -- (0,0);
    }
  \foreach \a in {4}
    {
    \draw [-triangle 45] (0,0) -- (\a*72:0.8cm);
    \draw (\a*72:0.8cm) --
      (\a*72:0.9cm);
    \draw (\a*72:0.9cm)
      .. controls +(\a*72:0.75cm) and +(\a*72-54+90:1cm) ..
      (\a*72-54:2cm)
      .. controls +(\a*72-54-90:1cm) and +(\a*72-108:0.75cm) ..
      (\a*72-108:0.9cm);
     \draw [-open triangle 45] (\a*72-108:0.9cm) --
       (\a*72-108:0.6cm);
    \draw (\a*72-108:0.6cm) -- (0,0);
    }
  \foreach \a in {1}
    {
    \draw [-triangle 45] (0,0) -- (\a*72:0.8cm);
    \draw (\a*72:0.8cm) --
      (\a*72:1cm);
    \draw (\a*72:1cm)
      .. controls +(\a*72:1cm) and +(\a*72:1.5cm) ..
      (\a*72+90:1.8cm)
      .. controls +(\a*72:-1.5cm) and +(\a*72:-1cm) ..
      (\a*72:-1cm);
     \draw [-open triangle 45] (\a*72+180:1cm) --
       (\a*72+180:0.6cm);
    \draw (\a*72+180:0.6cm) -- (0,0);
    }
    \useasboundingbox (-1.75cm,-2cm) rectangle (1.5cm,1.85cm);
\end{tikzpicture}%
\newline
$\sigma$: $(2453)$ \newline $\sigma'$: $(1254)$\newline
$b$: $1, 2453$ \newline $t$: $1254, 3$\newline
\end{minipage}
\hfill
\begin{minipage}{0.24\textwidth}
{\scriptsize
\begin{tikzpicture}[auto,scale=0.7]
  \node at (3,1.5) {\bfseries\sffamily O};
  \draw (0,0)
    -- node {2} ++(1,0)
    -- node {4} ++(1,0)
    -- node {5} ++(1,0)
    -- node {3} ++(1,0)
    -- ++ (0,1)
    -- node {4} ++(-1,0)
    -- node {5} ++(-1,0)
    -- node {2} ++(-1,0)
    -- node {1} node [swap] {1} ++(-1,0)
    -- cycle;
  \draw (0,0) ++(0,1) 
    -- ++(0,1)
    -- node [swap] {3} ++(1,0)
    -- ++(0,-1);
  \fill (0,0) circle (2pt);
  \fill (1,0) circle (2pt);
  \fill (2,0) circle (2pt);
  \fill (3,0) circle (2pt);
  \fill (4,0) circle (2pt);
  \fill (4,1) circle (2pt);
  \fill (3,1) circle (2pt);
  \fill (2,1) circle (2pt);
  \fill (1,1) circle (2pt);
  \fill (1,2) circle (2pt);
  \fill (0,2) circle (2pt);
  \fill (0,1) circle (2pt);
\end{tikzpicture}}\newline
1-3, 2453-1254\newline
\end{minipage}
\hfill
\begin{minipage}{0.24\textwidth}
\centering
\begin{tikzpicture}[auto,scale=0.65]
  \pgfsetbaseline{0cm}
  \fill (0,0) circle (2pt);
  \foreach \a in {0}
    {
    \draw [-triangle 45] (0,0) -- (\a*72:0.8cm);
    \draw (\a*72:0.8cm)
      .. controls +(\a*72:0.75cm) and +(\a*72+36:0.75cm) ..
      (\a*72+36:0.8cm);
    \draw [-open triangle 45] (\a*72+36:0.8cm) --
      (\a*72+36:0.6cm);
    \draw (\a*72+36:0.6cm) -- (0,0);
    }
  \foreach \a in {3}
    {
    \draw [-triangle 45] (0,0) -- (\a*72:0.8cm);
    \draw (\a*72:0.8cm)
      .. controls +(\a*72:0.7cm) and +(\a*72-36:0.7cm) ..
      (\a*72-36:0.8cm);
    \draw [-open triangle 45] (\a*72-36:0.8cm) --
      (\a*72-36:0.6cm);
    \draw (\a*72-36:0.6cm) -- (0,0);
    }
  \foreach \a in {1,2,4}
    {
    \draw [-triangle 45] (0,0) -- (\a*72:0.8cm);
    \draw [-open triangle 45] (\a*72:0.8cm)
      .. controls +(\a*72:0.8cm) and +(\a*72:1.35cm) ..
      (\a*72+90:1.75cm)
      .. controls +(\a*72:-1.35cm) and +(\a*72:-0.8cm) ..
      (\a*72+180:0.6cm);
    \draw (\a*72+180:0.6cm) -- (0,0);
    }
    \useasboundingbox (-1.75cm,-2cm) rectangle (1.5cm,1.85cm);
\end{tikzpicture}%
\newline
$\sigma$: $(2435)$ \newline $\sigma'$: $(1253)$\newline
$b$: $1, 2435$ \newline $t$: $1253, 4$\newline
\end{minipage}
\hfill
\begin{minipage}{0.24\textwidth}
{\scriptsize
\begin{tikzpicture}[auto,scale=0.7]
  \node at (3,1.5) {\bfseries\sffamily H};
  \draw (0,0)
    -- node {2} ++(1,0)
    -- node {4} ++(1,0)
    -- node {3} ++(1,0)
    -- node {5} ++(1,0)
    -- ++ (0,1)
    -- node {3} ++(-1,0)
    -- node {5} ++(-1,0)
    -- node {2} ++(-1,0)
    -- node {1} node [swap] {1} ++(-1,0)
    -- cycle;
  \draw (0,0) ++(0,1)
    -- ++(0,1)
    -- node [swap] {4} ++(1,0)
    -- ++(0,-1);
  \fill (0,0) circle (2pt);
  \fill (1,0) circle (2pt);
  \fill (2,0) circle (2pt);
  \fill (3,0) circle (2pt);
  \fill (4,0) circle (2pt);
  \fill (4,1) circle (2pt);
  \fill (3,1) circle (2pt);
  \fill (2,1) circle (2pt);
  \fill (1,1) circle (2pt);
  \fill (1,2) circle (2pt);
  \fill (0,2) circle (2pt);
  \fill (0,1) circle (2pt);
\end{tikzpicture}}\newline
1-4, 2435-1253\newline
\end{minipage}
\hfill
\begin{minipage}{0.24\textwidth}
\centering
\begin{tikzpicture}[auto,scale=0.7]
  \pgfsetbaseline{0cm}
  \fill (0,0) circle (2pt);
  \foreach \a in {0,3}
    {
    \draw [-triangle 45] (0,0) -- (\a*72:0.8cm);
    \draw (\a*72:0.8cm) --
      (\a*72:1cm);
    \draw (\a*72:1cm)
      .. controls +(\a*72:0.75cm) and +(\a*72+36:0.75cm) ..
      (\a*72+36:1cm);
    \draw [-open triangle 45] (\a*72+36:1cm) --
      (\a*72+36:0.6cm);
    \draw (\a*72+36:0.6cm) -- (0,0);
    }
  \foreach \a in {2}
    {
    \draw [-triangle 45] (0,0) -- (\a*72:0.8cm);
    \draw (\a*72:0.8cm) --
      (\a*72:1cm);
    \draw (\a*72:1cm)
      .. controls +(\a*72:0.75cm) and +(\a*72-36:0.75cm) ..
      (\a*72-36:1cm);
    \draw [-open triangle 45] (\a*72-36:1cm) --
      (\a*72-36:0.6cm);
    \draw (\a*72-36:0.6cm) -- (0,0);
    }
  \foreach \a in {1,4}
    {
    \draw [-triangle 45] (0,0) -- (\a*72:0.8cm);
    \draw (\a*72:0.8cm) --
      (\a*72:0.9cm);
    \draw (\a*72:0.9cm)
      .. controls +(\a*72:0.75cm) and +(\a*72-54+90:1cm) ..
      (\a*72-54:2cm)
      .. controls +(\a*72-54-90:1cm) and +(\a*72-108:0.75cm) ..
      (\a*72-108:0.9cm);
     \draw [-open triangle 45] (\a*72-108:0.9cm) --
       (\a*72-108:0.6cm);
    \draw (\a*72-108:0.6cm) -- (0,0);
    }
    \useasboundingbox (-1.75cm,-2cm) rectangle (1.5cm,1.85cm);
\end{tikzpicture}%
\newline
$\sigma$: $(253)$ \newline $\sigma'$: $(12)(45)$\newline
$b$: $1, 253, 4$ \newline $t$: $12, 3, 45$\newline
\end{minipage}
\hfill
\begin{minipage}{0.24\textwidth}
{\scriptsize
\begin{tikzpicture}[auto,scale=0.57]
  \node at (3,1.5) {\bfseries\sffamily O};
  \draw (0,0)
    -- node {2} ++(1,0)
    -- node {5} ++(1,0)
    -- node {3} ++(1,0)
    -- ++ (0,1)
    -- node {5} ++(-1,0)
    -- node {4} node [swap] {4} ++(-2,0)
    -- cycle;
  \draw (0,0) ++(0,1)
    -- ++(0,1)
    -- node [swap] {1} node {1} ++(1,0)
    -- node [swap] {2} ++(1,0)
    -- ++(0,-1);
  \draw (0,0) ++(0,2) -- ++(0,1) -- node [swap] {3} ++(1,0)
    -- ++(0,-1);
  \fill (0,0) circle (2pt);
  \fill (1,0) circle (2pt);
  \fill (2,0) circle (2pt);
  \fill (3,0) circle (2pt);
  \fill (3,1) circle (2pt);
  \fill (2,1) circle (2pt);
  \fill (0,1) circle (2pt);
  \fill (2,2) circle (2pt);
  \fill (1,2) circle (2pt);
  \fill (0,2) circle (2pt);
  \fill (1,3) circle (2pt);
  \fill (0,3) circle (2pt);
\end{tikzpicture}}\newline
1-3, 253-45, 4-12\par\medskip
{\scriptsize
\begin{tikzpicture}[auto,scale=0.57]
  \node at (3,1.5) {\bfseries\sffamily O};
  \draw (0,0)
    -- node {2} ++(1,0)
    -- node {5} ++(1,0)
    -- node {3} ++(1,0)
    -- ++ (0,1)
    -- node {2} ++(-1,0)
    -- node {1} node [swap] {1} ++(-2,0)
    -- cycle;
  \draw (0,0) ++(0,1)
    -- ++(0,1)
    -- node [swap] {4} node {4} ++(1,0)
    -- node [swap] {5} ++(1,0)
    -- ++(0,-1);
  \draw (0,0) ++(0,2)
    -- ++(0,1)
    -- node [swap] {3} ++(1,0)
    -- ++(0,-1);
  \fill (0,0) circle (2pt);
  \fill (2,0) circle (2pt);
  \fill (3,0) circle (2pt);
  \fill (3,1) circle (2pt);
  \fill (2,1) circle (2pt);
  \fill (0,1) circle (2pt);
  \fill (2,2) circle (2pt);
  \fill (1,2) circle (2pt);
  \fill (0,2) circle (2pt);
  \fill (1,3) circle (2pt);
  \fill (0,3) circle (2pt);
\end{tikzpicture}}\newline
1-45, 253-12, 4-3\newline
\end{minipage}
\hfill
\begin{minipage}{0.24\textwidth}
\centering
\begin{tikzpicture}[auto,scale=0.7]
  \pgfsetbaseline{0cm}
  \fill (0,0) circle (2pt);
  \foreach \a in {0}
    {
    \draw [-triangle 45] (0,0) -- (\a*72:0.8cm);
    \draw (\a*72:0.8cm) --
      (\a*72:1cm);
    \draw (\a*72:1cm)
      .. controls +(\a*72:0.75cm) and +(\a*72+36:0.75cm) ..
      (\a*72+36:1cm);
    \draw [-open triangle 45] (\a*72+36:1cm) --
      (\a*72+36:0.6cm);
    \draw (\a*72+36:0.6cm) -- (0,0);
    }
  \foreach \a in {2}
    {
    \draw [-triangle 45] (0,0) -- (\a*72:0.8cm);
    \draw (\a*72:0.8cm) --
      (\a*72:0.9cm);
    \draw (\a*72:0.9cm)
      .. controls +(\a*72:0.75cm) and +(\a*72+54-90:1cm) ..
      (\a*72+54:2cm)
      .. controls +(\a*72+54+90:1cm) and +(\a*72+108:0.75cm) ..
      (\a*72+108:0.9cm);
     \draw [-open triangle 45] (\a*72+108:0.9cm) --
       (\a*72+108:0.6cm);
    \draw (\a*72+108:0.6cm) -- (0,0);
    }
  \foreach \a in {1,3,4}
    {
    \draw [-triangle 45] (0,0) -- (\a*72:0.8cm);
    \draw (\a*72:0.8cm) --
      (\a*72:0.9cm);
    \draw (\a*72:0.9cm)
      .. controls +(\a*72:0.75cm) and +(\a*72-54+90:1cm) ..
      (\a*72-54:2cm)
      .. controls +(\a*72-54-90:1cm) and +(\a*72-108:0.75cm) ..
      (\a*72-108:0.9cm);
     \draw [-open triangle 45] (\a*72-108:0.9cm) --
       (\a*72-108:0.6cm);
    \draw (\a*72-108:0.6cm) -- (0,0);
    }
    \useasboundingbox (-1.75cm,-2cm) rectangle (1.5cm,1.85cm);
\end{tikzpicture}%
\newline
$\sigma$: $(2534)$ \newline $\sigma'$: $(12)(354)$\newline
$b$: $1, 2534$ \newline $t$: $12, 354$\newline
\end{minipage}
\hfill
\begin{minipage}{0.24\textwidth}
{\scriptsize
\begin{tikzpicture}[auto,scale=0.7]
  \node at (4,1.5) {\bfseries\sffamily O};
  \draw (0,0)
    -- node {2} ++(1,0)
    -- node {5} ++(1,0)
    -- node {3} ++(1,0)
    -- node {4} ++(1,0)
    -- ++ (0,1)
    -- node {2} ++(-1,0)
    -- node {1} node [swap] {1} ++(-3,0)
    -- cycle;
  \draw (0,0) ++(0,1) -- ++(0,1) 
    -- node [swap] {3} ++(1,0)
    -- node [swap] {5} ++(1,0)
    -- node [swap] {4} ++(1,0)
    -- ++(0,-1);
  \fill (0,0) circle (2pt);
  \fill (1,0) circle (2pt);
  \fill (2,0) circle (2pt);
  \fill (3,0) circle (2pt);
  \fill (4,0) circle (2pt);
  \fill (4,1) circle (2pt);
  \fill (3,1) circle (2pt);
  \fill (3,2) circle (2pt);
  \fill (2,2) circle (2pt);
  \fill (1,2) circle (2pt);
  \fill (0,2) circle (2pt);
  \fill (0,1) circle (2pt);
\end{tikzpicture}}\newline
1-354, 2534-12\newline
\end{minipage}
\hfill


\begin{minipage}{0.28\textwidth}
\centering
\begin{tikzpicture}[auto,scale=0.7]
  \pgfsetbaseline{0cm}
  \fill (0,0) circle (2pt);
  \foreach \a in {0}
    {
    \draw [-triangle 45] (0,0) -- (\a*72:0.8cm);
    \draw (\a*72:0.8cm) --
      (\a*72:1cm);
    \draw (\a*72:1cm)
      .. controls +(\a*72:0.75cm) and +(\a*72+36:0.75cm) ..
      (\a*72+36:1cm);
    \draw [-open triangle 45] (\a*72+36:1cm) --
      (\a*72+36:0.6cm);
    \draw (\a*72+36:0.6cm) -- (0,0);
    }
  \foreach \a in {3}
    {
    \draw [-triangle 45] (0,0) -- (\a*72:0.8cm);
    \draw (\a*72:0.8cm) --
      (\a*72:1cm);
    \draw (\a*72:1cm)
      .. controls +(\a*72:0.75cm) and +(\a*72-36:0.75cm) ..
      (\a*72-36:1cm);
    \draw [-open triangle 45] (\a*72-36:1cm) --
      (\a*72-36:0.6cm);
    \draw (\a*72-36:0.6cm) -- (0,0);
    }
  \foreach \a in {2}
    {
    \draw [-triangle 45] (0,0) -- (\a*72:0.8cm);
    \draw (\a*72:0.8cm) --
      (\a*72:0.9cm);
    \draw (\a*72:0.9cm)
      .. controls +(\a*72:0.75cm) and +(\a*72+54-90:1cm) ..
      (\a*72+54:2cm)
      .. controls +(\a*72+54+90:1cm) and +(\a*72+108:0.75cm) ..
      (\a*72+108:0.9cm);
     \draw [-open triangle 45] (\a*72+108:0.9cm) --
       (\a*72+108:0.6cm);
    \draw (\a*72+108:0.6cm) -- (0,0);
    }
  \foreach \a in {1}
    {
    \draw [-triangle 45] (0,0) -- (\a*72:0.8cm);
    \draw (\a*72:0.8cm) --
      (\a*72:0.9cm);
    \draw (\a*72:0.9cm)
      .. controls +(\a*72:0.75cm) and +(\a*72-54+90:1cm) ..
      (\a*72-54:2cm)
      .. controls +(\a*72-54-90:1cm) and +(\a*72-108:0.75cm) ..
      (\a*72-108:0.9cm);
     \draw [-open triangle 45] (\a*72-108:0.9cm) --
       (\a*72-108:0.6cm);
    \draw (\a*72-108:0.6cm) -- (0,0);
    }
  \foreach \a in {4}
    {
    \draw [-triangle 45] (0,0) -- (\a*72:0.8cm);
    \draw [-open triangle 45] (\a*72:0.8cm)
      .. controls +(\a*72:1.25cm) and +(\a*72:2cm) ..
      (\a*72+90:2.5cm)
      .. controls +(\a*72:-2cm) and +(\a*72:-1.25cm) ..
      (\a*72+180:0.6cm);
    \draw (\a*72+180:0.6cm) -- (0,0);
    }
    \useasboundingbox (-1.75cm,-2cm) rectangle (1.5cm,1.85cm);
\end{tikzpicture}%
\newline
$\sigma$: $(25)(34)$ \newline $\sigma'$: $(12)(35)$\newline
$b$: $1, 25, 34$ \newline $t$: $12, 35, 4$\newline
\end{minipage}
\hfill
\begin{minipage}{0.20\textwidth}
{\scriptsize
\begin{tikzpicture}[auto,scale=0.6]
  \node at (3,1.5) {\bfseries\sffamily O};
  \draw (0,0)
    -- node {3} ++(1,0)
    -- node {4} node [swap] {4} ++(2,0)
    -- ++ (0,1)
    -- node {2} ++(-1,0)
    -- node {1} node [swap] {1} ++(-2,0)
    -- cycle;
  \draw (0,0) ++(0,1)
    -- ++(0,1)
    -- node [swap] {3} ++(1,0)
    -- node [swap] {5} ++(1,0)
    -- ++(0,-1);
  \draw (0,0) ++(1,0)
    -- ++(0,-1)
    -- node {2} ++(1,0)
    -- node {5} ++(1,0)
    -- ++(0,1);
  \fill (0,0) circle (2pt);
  \fill (1,0) circle (2pt);
  \fill (3,0) circle (2pt);
  \fill (3,1) circle (2pt);
  \fill (2,1) circle (2pt);
  \fill (0,1) circle (2pt);
  \fill (2,2) circle (2pt);
  \fill (1,2) circle (2pt);
  \fill (0,2) circle (2pt);
  \fill (1,-1) circle (2pt);
  \fill (2,-1) circle (2pt);
  \fill (3,-1) circle (2pt);
\end{tikzpicture}}\newline
1-35, 25-4, 34-12\par\medskip
{\scriptsize
\begin{tikzpicture}[auto,scale=0.6]
  \node at (3,1.5) {\bfseries\sffamily O};
  \draw (1,0)
    -- node {2} ++(1,0)
    -- node {5} ++(1,0)
    -- ++ (0,1)
    -- node {2} ++(-1,0)
    -- node {1} node [swap] {1} ++(-1,0)
    -- cycle;
  \draw (1,0) ++(0,1)
    -- ++(0,1)
    -- node [swap] {4} node {4} ++(1,0)
    -- ++(0,-1);
  \draw (1,2)
    -- node [swap] {3} ++(-1,0)
    -- ++(0,1)
    -- node [swap] {3} ++(1,0)
    -- node [swap] {5} ++(1,0)
    -- ++(0,-1);
  \fill (1,0) circle (2pt);
  \fill (2,0) circle (2pt);
  \fill (3,0) circle (2pt);
  \fill (3,1) circle (2pt);
  \fill (2,1) circle (2pt);
  \fill (1,1) circle (2pt);
  \fill (2,2) circle (2pt);
  \fill (1,2) circle (2pt);
  \fill (0,2) circle (2pt);
  \fill (2,3) circle (2pt);
  \fill (1,3) circle (2pt);
  \fill (0,3) circle (2pt);
\end{tikzpicture}}\newline
1-4, 25-12, 34-35\par\medskip
{\scriptsize
\begin{tikzpicture}[auto,scale=0.6]
  \node at (3,0.5) {\bfseries\sffamily H};
  \draw (0,0)
    -- node {3} ++(1,0)
    -- node {4} ++(1,0)
    -- ++ (0,1)
    -- node {2} node [swap] {2} ++(-1,0)
    -- node {1} node [swap] {1} ++(-1,0)
    -- cycle;
  \draw (0,0) ++(0,1)
    -- ++(-0.2,1)
    -- node [swap] {4} ++(1,0)
    -- ++(0.2,-1);
  \draw (0,0) ++(0,1) ++(1,0)
    -- ++(0.2,1)
    -- node [swap] {3} ++(1,0)
    -- node [swap] {5} ++(1,0)
    -- ++(-0.2,-1) -- node [swap] {5} ++(-1,0);
  \fill (0,0) circle (2pt);
  \fill (1,0) circle (2pt);
  \fill (2,0) circle (2pt);
  \fill (3,1) circle (2pt);
  \fill (2,1) circle (2pt);
  \fill (3.2,2) circle (2pt);
  \fill (2.2,2) circle (2pt);
  \fill (1.2,2) circle (2pt);
  \fill (1,1) circle (2pt);
  \fill (0.8,2) circle (2pt);
  \fill (-0.2,2) circle (2pt);
  \fill (0,1) circle (2pt);
\end{tikzpicture}}\newline
1-4, 25-35, 34-12\newline
\end{minipage}
\hfill
\begin{minipage}{0.24\textwidth}
\centering
\begin{tikzpicture}[auto,scale=0.7]
  \pgfsetbaseline{0cm}
  \foreach \a in {1}
    {
    \draw [-triangle 45] (0,0) -- (\a*72:0.8cm);
    \draw (\a*72:0.8cm) --
      (\a*72:1cm);
    \draw (\a*72:1cm)
      .. controls +(\a*72:0.75cm) and +(\a*72-36:0.75cm) ..
      (\a*72-36:1cm);
    \draw [-open triangle 45] (\a*72-36:1cm) --
      (\a*72-36:0.6cm);
    \draw (\a*72-36:0.6cm) -- (0,0);
    }
  \foreach \a in {0,2,3}
    {
    \draw [-triangle 45] (0,0) -- (\a*72:0.8cm);
    \draw (\a*72:0.8cm) --
      (\a*72:0.9cm);
    \draw (\a*72:0.9cm)
      .. controls +(\a*72:0.75cm) and +(\a*72+54-90:1cm) ..
      (\a*72+54:2cm)
      .. controls +(\a*72+54+90:1cm) and +(\a*72+108:0.75cm) ..
      (\a*72+108:0.9cm);
     \draw [-open triangle 45] (\a*72+108:0.9cm) --
       (\a*72+108:0.6cm);
    \draw (\a*72+108:0.6cm) -- (0,0);
    }
  \foreach \a in {4}
    {
    \draw [-triangle 45] (0,0) -- (\a*72:0.8cm);
    \draw (\a*72:0.8cm) --
      (\a*72:0.9cm);
    \draw (\a*72:0.9cm)
      .. controls +(\a*72:0.75cm) and +(\a*72-54+90:1cm) ..
      (\a*72-54:2cm)
      .. controls +(\a*72-54-90:1cm) and +(\a*72-108:0.75cm) ..
      (\a*72-108:0.9cm);
     \draw [-open triangle 45] (\a*72-108:0.9cm) --
       (\a*72-108:0.6cm);
    \draw (\a*72-108:0.6cm) -- (0,0);
    }
    \useasboundingbox (-1.75cm,-2cm) rectangle (1.5cm,1.85cm);
\end{tikzpicture}%
\par\vspace*{0.25cm}
$\sigma$: $(12)(345)$ \newline $\sigma'$: $(1354)$\newline
$b$: $12, 345$ \newline $t$: $1354, 2$\newline
\end{minipage}
\hfill
\begin{minipage}{0.24\textwidth}
{\scriptsize
\begin{tikzpicture}[auto,scale=0.7]
  \node at (3,2.5) {\bfseries\sffamily O};
  \draw (1,0)
    -- node {3} ++(1,0)
    -- node {4} ++(1,0)
    -- node {5} ++(1,0)
    -- ++ (0,1)
    -- node {2} node [swap] {2} ++(-3,0)
    -- cycle;
  \draw (1,1)
    -- node [swap] {1} ++(-1,0)
    -- ++(0,1)
    -- node [swap] {1} ++(1,0)
    -- node [swap] {3} ++(1,0)
    -- node [swap] {5} ++(1,0)
    -- node [swap] {4} ++(1,0)
    -- ++(0,-1);
  \fill (1,0) circle (2pt);
  \fill (2,0) circle (2pt);
  \fill (3,0) circle (2pt);
  \fill (4,0) circle (2pt);
  \fill (0,1) circle (2pt);
  \fill (1,1) circle (2pt);
  \fill (4,1) circle (2pt);
  \fill (0,2) circle (2pt);
  \fill (1,2) circle (2pt);
  \fill (2,2) circle (2pt);
  \fill (3,2) circle (2pt);
  \fill (4,2) circle (2pt);
\end{tikzpicture}}\newline
12-1354, 345-2\newline
\end{minipage}
\hfill\strut

\vspace*{-0.5cm}

\strut
\hfill 
\begin{minipage}{0.24\textwidth}
\centering
\begin{tikzpicture}[auto,scale=0.7]
  \pgfsetbaseline{0cm}
  \fill (0,0) circle (2pt);
  \foreach \a in {0,1,2,3,4}
    {
    \draw [-triangle 45] (0,0) -- (\a*72:0.8cm);
    \draw (\a*72:0.8cm) --
      (\a*72:0.9cm);
    \draw (\a*72:0.9cm)
      .. controls +(\a*72:0.75cm) and +(\a*72+54-90:1cm) ..
      (\a*72+54:2cm)
      .. controls +(\a*72+54+90:1cm) and +(\a*72+108:0.75cm) ..
      (\a*72+108:0.9cm);
     \draw [-open triangle 45] (\a*72+108:0.9cm) --
       (\a*72+108:0.6cm);
    \draw (\a*72+108:0.6cm) -- (0,0);
    }
    \useasboundingbox (-1.75cm,-2cm) rectangle (1.5cm,1.85cm);
\end{tikzpicture}%
\par\vspace*{0.5cm}
$\sigma$: $(12345)$ \newline $\sigma'$: $(13524)$\newline
$b$: $12345$ \newline $t$: $13524$\newline
\end{minipage}
\hfill
\begin{minipage}{0.24\textwidth}
{\scriptsize
\begin{tikzpicture}[auto,scale=0.6]
  \node at (3,1.5) {\bfseries\sffamily O};
  \draw (0,0)
    -- node {1} ++(1,0)
    -- node {2} ++(1,0)
    -- node {3} ++(1,0)
    -- node {4} ++(1,0)
    -- node {5} ++(1,0)
    -- ++ (0,1)
    -- node {4} ++(-1,0)
    -- node {2} ++(-1,0)
    -- node {5} ++(-1,0)
    -- node {3} ++(-1,0)
    -- node {1} ++(-1,0)
    -- cycle;
  \foreach \i in {0,...,5}
    {
    \fill (\i,0) circle (2pt);
    \fill (\i,1) circle (2pt);
    }
\end{tikzpicture}}\newline
12345-13524\newline
\end{minipage}
\hfill
\begin{minipage}{0.24\textwidth}
\centering
\begin{tikzpicture}[auto,scale=0.6]
  \pgfsetbaseline{0cm}
  \fill (0,0) circle (2pt);
  \foreach \a in {0,3}
    {
    \draw [-triangle 45] (0,0) -- (\a*72:0.8cm);
    \draw (\a*72:0.8cm) --
      (\a*72:0.9cm);
    \draw (\a*72:0.9cm)
      .. controls +(\a*72:0.75cm) and +(\a*72+54-90:1cm) ..
      (\a*72+54:2cm)
      .. controls +(\a*72+54+90:1cm) and +(\a*72+108:0.75cm) ..
      (\a*72+108:0.9cm);
     \draw [-open triangle 45] (\a*72+108:0.9cm) --
       (\a*72+108:0.6cm);
    \draw (\a*72+108:0.6cm) -- (0,0);
    }
  \foreach \a in {2,4}
    {
    \draw [-triangle 45] (0,0) -- (\a*72:0.8cm);
    \draw (\a*72:0.8cm) --
      (\a*72:0.9cm);
    \draw (\a*72:0.9cm)
      .. controls +(\a*72:0.75cm) and +(\a*72-54+90:1cm) ..
      (\a*72-54:2cm)
      .. controls +(\a*72-54-90:1cm) and +(\a*72-108:0.75cm) ..
      (\a*72-108:0.9cm);
     \draw [-open triangle 45] (\a*72-108:0.9cm) --
       (\a*72-108:0.6cm);
    \draw (\a*72-108:0.6cm) -- (0,0);
    }
  \foreach \a in {1}
    {
    \draw [-triangle 45] (0,0) -- (\a*72:0.8cm);
    \draw [-open triangle 45] (\a*72:0.8cm)
      .. controls +(\a*72:1.25cm) and +(\a*72:1.75cm) ..
      (\a*72+90:2.25cm)
      .. controls +(\a*72:-1.75cm) and +(\a*72:-1.25cm) ..
      (\a*72+180:0.6cm);
    \draw (\a*72+180:0.6cm) -- (0,0);
    }
    \useasboundingbox (-1.75cm,-2cm) rectangle (1.5cm,1.85cm);
\end{tikzpicture}%
\par\vspace*{0.25cm}
$\sigma$: $(12453)$ \newline $\sigma'$: $(13254)$\newline
$b$: $12453$ \newline $t$: $13254$\newline
\end{minipage}
\hfill
\begin{minipage}{0.24\textwidth}
{\scriptsize
\begin{tikzpicture}[auto,scale=0.6]
  \node at (3,1.5) {\bfseries\sffamily O};
  \draw (0,0)
    -- node {1} ++(1,0)
    -- node {2} ++(1,0)
    -- node {4} ++(1,0)
    -- node {5} ++(1,0)
    -- node {3} ++(1,0)
    -- ++ (0,1)
    -- node {4} ++(-1,0)
    -- node {5} ++(-1,0)
    -- node {2} ++(-1,0)
    -- node {3} ++(-1,0)
    -- node {1} ++(-1,0)
    -- cycle;
  \foreach \i in {0,...,5}
    {
    \fill (\i,0) circle (2pt);
    \fill (\i,1) circle (2pt);
    }
\end{tikzpicture}}\newline
12453-13254\newline
\end{minipage}
\hfill
\begin{minipage}{0.28\textwidth}
\centering
\begin{tikzpicture}[auto,scale=0.7]
  \pgfsetbaseline{0cm}
  \fill (0,0) circle (2pt);
  \foreach \a in {0}
    {
    \draw [-triangle 45] (0,0) -- (\a*72:0.8cm);
    \draw (\a*72:0.8cm) --
      (\a*72:0.9cm);
    \draw (\a*72:0.9cm)
      .. controls +(\a*72:0.75cm) and +(\a*72+54-90:1cm) ..
      (\a*72+54:2cm)
      .. controls +(\a*72+54+90:1cm) and +(\a*72+108:0.75cm) ..
      (\a*72+108:0.9cm);
     \draw [-open triangle 45] (\a*72+108:0.9cm) --
       (\a*72+108:0.6cm);
    \draw (\a*72+108:0.6cm) -- (0,0);
    }
  \foreach \a in {4}
    {
    \draw [-triangle 45] (0,0) -- (\a*72:0.8cm);
    \draw (\a*72:0.8cm) --
      (\a*72:0.9cm);
    \draw (\a*72:0.9cm)
      .. controls +(\a*72:0.75cm) and +(\a*72-54+90:1cm) ..
      (\a*72-54:2cm)
      .. controls +(\a*72-54-90:1cm) and +(\a*72-108:0.75cm) ..
      (\a*72-108:0.9cm);
     \draw [-open triangle 45] (\a*72-108:0.9cm) --
       (\a*72-108:0.6cm);
    \draw (\a*72-108:0.6cm) -- (0,0);
    }
  \foreach \a in {1,3}
    {
    \draw [-triangle 45] (0,0) -- (\a*72:0.8cm);
    \draw [-open triangle 45] (\a*72:0.8cm)
      .. controls +(\a*72:1.25cm) and +(\a*72:1.75cm) ..
      (\a*72+90:2.25cm)
      .. controls +(\a*72:-1.75cm) and +(\a*72:-1.25cm) ..
      (\a*72+180:0.6cm);
    \draw (\a*72+180:0.6cm) -- (0,0);
    }
  \foreach \a in {2}
    {
    \draw [-triangle 45] (0,0) -- (\a*72:0.8cm);
    \draw [-open triangle 45] (\a*72:0.8cm)
      .. controls +(\a*72:1.25cm) and +(\a*72:1.75cm) ..
      (\a*72+90:2.5cm)
      .. controls +(\a*72:-1.75cm) and +(\a*72:-1.25cm) ..
      (\a*72+180:0.6cm);
    \draw (\a*72+180:0.6cm) -- (0,0);
    }
    \useasboundingbox (-1.75cm,-2cm) rectangle (1.5cm,1.85cm);
\end{tikzpicture}%
\newline
$\sigma$: $(124)(35)$ \newline $\sigma'$: $(13)(254)$\newline
$b$: $124, 35$ \newline $t$: $13, 254$\newline
\end{minipage}
\hfill
\begin{minipage}{0.20\textwidth}
{\scriptsize
\begin{tikzpicture}[auto,scale=0.6]
  \node at (4,0.5) {\bfseries\sffamily O};
  \draw (0,0)
    -- node {1} ++(1,0)
    -- node {2} ++(1,0)
    -- node {4} ++(1,0)
    -- ++ (0,1)
    -- node {3} node [swap] {3} ++(-2,0)
    -- node {1} ++(-1,0)
    -- cycle;
  \draw (1,1)
    -- ++(0,1)
    -- node [swap] {2} ++(1,0)
    -- node [swap] {5} ++(1,0)
    -- node [swap] {4} ++(1,0)
    -- ++(0,-1)
    -- node [swap] {5} ++(-1,0);
  \fill (0,0) circle (2pt);
  \fill (1,0) circle (2pt);
  \fill (2,0) circle (2pt);
  \fill (3,0) circle (2pt);
  \fill (0,1) circle (2pt);
  \fill (1,1) circle (2pt);
  \fill (3,1) circle (2pt);
  \fill (4,1) circle (2pt);
  \fill (1,2) circle (2pt);
  \fill (2,2) circle (2pt);
  \fill (3,2) circle (2pt);
  \fill (4,2) circle (2pt);
\end{tikzpicture}}\newline
124-13, 35-254\par\medskip
{\scriptsize
\begin{tikzpicture}[auto,scale=0.6]
  \node at (3,0.5) {\bfseries\sffamily H};
  \draw (0,0)
    -- node {3} ++(1,0)
    -- node {5} ++(1,0)
    -- ++ (0,1)
    -- node {1} node [swap] {1} ++(-1,0)
    -- node {3} ++(-1,0)
    -- cycle;
  \draw (1,1)
    -- ++(0,1)
    -- node [swap] {2} ++(1,0)
    -- node [swap] {5} ++(1,0)
    -- node [swap] {4} ++(1,0)
    -- ++(0,-1)
    -- node [swap] {4} ++(-1,0)
    -- node [swap] {2} ++(-1,0);
  \fill (0,0) circle (2pt);
  \fill (1,0) circle (2pt);
  \fill (2,0) circle (2pt);
  \fill (0,1) circle (2pt);
  \fill (1,1) circle (2pt);
  \fill (2,1) circle (2pt);
  \fill (3,1) circle (2pt);
  \fill (4,1) circle (2pt);
  \fill (1,2) circle (2pt);
  \fill (2,2) circle (2pt);
  \fill (3,2) circle (2pt);
  \fill (4,2) circle (2pt);
\end{tikzpicture}}\newline
124-254, 35-13\newline
\end{minipage}
\hfill
\begin{minipage}{0.24\textwidth}
\centering
\begin{tikzpicture}[auto,scale=0.6]
  \pgfsetbaseline{0cm}
  \fill (0,0) circle (2pt);
  \foreach \a in {0,1,2,3,4}
    {
    \draw [-triangle 45] (0,0) -- (\a*72:0.8cm);
    \draw [-open triangle 45] (\a*72:0.8cm)
      .. controls +(\a*72:1.25cm) and +(\a*72:1.75cm) ..
      (\a*72+90:2.5cm)
      .. controls +(\a*72:-1.75cm) and +(\a*72:-1.25cm) ..
      (\a*72+180:0.6cm);
    \draw (\a*72+180:0.6cm) -- (0,0);
    }
    \useasboundingbox (-1.75cm,-2cm) rectangle (1.5cm,1.85cm);
\end{tikzpicture}%
\newline
$\sigma$: $(13524)$ \newline $\sigma'$: $(14253)$\newline
$b$: $13524$ \newline $t$: $14253$\newline
\end{minipage}
\hfill
\begin{minipage}{0.24\textwidth}
{\scriptsize
\begin{tikzpicture}[auto,scale=0.6]
  \node at (3,1.5) {\bfseries\sffamily H};
  \draw (0,0)
    -- node {1} ++(1,0)
    -- node {3} ++(1,0)
    -- node {5} ++(1,0)
    -- node {2} ++(1,0)
    -- node {4} ++(1,0)
    -- ++ (0,1)
    -- node {3} ++(-1,0)
    -- node {5} ++(-1,0)
    -- node {2} ++(-1,0)
    -- node {4} ++(-1,0)
    -- node {1} ++(-1,0)
    -- cycle;
  \foreach \i in {0,...,5}
    {
    \fill (\i,0) circle (2pt);
    \fill (\i,1) circle (2pt);
    }
\end{tikzpicture}}\newline
13524-14253\newline
\end{minipage}
\hfill\strut

\vspace*{0.25cm}

\strut\hfill
\begin{minipage}{0.24\textwidth}
\centering
\begin{tikzpicture}[auto,scale=0.7]
  \pgfsetbaseline{0cm}
  \fill (0,0) circle (2pt);
  \foreach \a in {4}
    {
    \draw [-triangle 45] (0,0) -- (\a*72:0.8cm);
    \draw (\a*72:0.8cm) --
      (\a*72:1cm);
    \draw (\a*72:1cm)
      .. controls +(\a*72:0.75cm) and +(\a*72-36:0.75cm) ..
      (\a*72-36:1cm);
    \draw [-open triangle 45] (\a*72-36:1cm) --
      (\a*72-36:0.6cm);
    \draw (\a*72-36:0.6cm) -- (0,0);
    }
  \foreach \a in {1,2,3}
    {
    \draw [-triangle 45] (0,0) -- (\a*72:0.8cm);
    \draw (\a*72:0.8cm) --
      (\a*72:0.9cm);
    \draw (\a*72:0.9cm)
      .. controls +(\a*72:0.75cm) and +(\a*72-54+90:1cm) ..
      (\a*72-54:2cm)
      .. controls +(\a*72-54-90:1cm) and +(\a*72-108:0.75cm) ..
      (\a*72-108:0.9cm);
     \draw [-open triangle 45] (\a*72-108:0.9cm) --
       (\a*72-108:0.6cm);
    \draw (\a*72-108:0.6cm) -- (0,0);
    }
  \foreach \a in {0}
    {
    \draw [-triangle 45] (0,0) -- (\a*72:0.8cm);
    \draw [-open triangle 45] (\a*72:0.8cm)
      .. controls +(\a*72:1.25cm) and +(\a*72:1.75cm) ..
      (\a*72+90:2.5cm)
      .. controls +(\a*72:-1.75cm) and +(\a*72:-1.25cm) ..
      (\a*72+180:0.6cm);
    \draw (\a*72+180:0.6cm) -- (0,0);
    }
    \useasboundingbox (-1.75cm,-2cm) rectangle (1.5cm,1.85cm);
\end{tikzpicture}%
\newline
$\sigma$: $(13)(254)$ \newline $\sigma'$: $(1432)$\newline
$b$: $13, 254$ \newline $t$: $1432, 5$\newline
\end{minipage}
\hfill
\begin{minipage}{0.24\textwidth}
{\scriptsize
\begin{tikzpicture}[auto,scale=0.7]
  \node at (4,0.5) {\bfseries\sffamily O};
  \draw (1,0)
    -- node {1} ++(1,0)
    -- node {3} ++(1,0)
    -- ++ (0,1)
    -- node {5} node [swap] {5} ++(-2,0)
    -- cycle;
  \draw (1,1)
    -- node [swap] {2} ++(-1,0)
    -- ++(0,1)
    -- node [swap] {1} ++(1,0)
    -- node [swap] {4} ++(1,0)
    -- node [swap] {3} ++(1,0)
    -- node [swap] {2} ++(1,0)
    -- ++(0,-1)
    -- node [swap] {4} ++(-1,0);
  \fill (1,0) circle (2pt);
  \fill (2,0) circle (2pt);
  \fill (3,0) circle (2pt);
  \fill (0,1) circle (2pt);
  \fill (1,1) circle (2pt);
  \fill (3,1) circle (2pt);
  \fill (4,1) circle (2pt);
  \fill (0,2) circle (2pt);
  \fill (1,2) circle (2pt);
  \fill (2,2) circle (2pt);
  \fill (3,2) circle (2pt);
  \fill (4,2) circle (2pt);
\end{tikzpicture}}\newline
13-5, 254-1432\newline
\end{minipage}
\hfill
\begin{minipage}{0.24\textwidth}
\centering
\begin{tikzpicture}[auto,scale=0.7]
  \pgfsetbaseline{0cm}
  \fill (0,0) circle (2pt);
  \foreach \a in {0,1,2,3,4}
    {
    \draw [-triangle 45] (0,0) -- (\a*72:0.8cm);
    \draw (\a*72:0.8cm) --
      (\a*72:0.9cm);
    \draw (\a*72:0.9cm)
      .. controls +(\a*72:0.75cm) and +(\a*72-54+90:1cm) ..
      (\a*72-54:2cm)
      .. controls +(\a*72-54-90:1cm) and +(\a*72-108:0.75cm) ..
      (\a*72-108:0.9cm);
     \draw [-open triangle 45] (\a*72-108:0.9cm) --
       (\a*72-108:0.6cm);
    \draw (\a*72-108:0.6cm) -- (0,0);
    }
    \useasboundingbox (-1.75cm,-2cm) rectangle (1.5cm,1.85cm);
\end{tikzpicture}%
\newline
$\sigma$: $(14253)$ \newline $\sigma'$: $(15432)$\newline
$b$: $14253$ \newline $t$: $15432$\newline
\end{minipage}
\hfill
\begin{minipage}{0.24\textwidth}
{\scriptsize
\begin{tikzpicture}[auto,scale=0.7]
  \node at (3,1.5) {\bfseries\sffamily O};
  \draw (0,0)
    -- node {1} ++(1,0)
    -- node {4} ++(1,0)
    -- node {2} ++(1,0)
    -- node {5} ++(1,0)
    -- node {3} ++(1,0)
    -- ++ (0,1)
    -- node {2} ++(-1,0)
    -- node {3} ++(-1,0)
    -- node {4} ++(-1,0)
    -- node {5} ++(-1,0)
    -- node {1} ++(-1,0)
    -- cycle;
  \foreach \i in {0,...,5}
    {
    \fill (\i,0) circle (2pt);
    \fill (\i,1) circle (2pt);
    }
\end{tikzpicture}}\newline
14253-15432\newline
\end{minipage}
\hfill\strut

\setlength{\parindent}{12pt}


\begin{thebibliography}{ABCD99}

\bibitem[AS]{AllSha} J.-P.~Allouche and J.~Shallit, 
\emph{Automatic sequences. Theory, applications, generalizations}, Cambridge University Press, Cambridge, 2003. xvi+571 pp. ISBN: 0-521-82332-3  

\bibitem[Ar]{Arnoux} P.~Arnoux, 
\emph{Le codage du flot g\'eod\'esique sur la surface modulaire}, Enseign. Math. 40 (1994), 29--48. 

\bibitem[AMY]{AMY} A. Avila, C. Matheus and J.-C. Yoccoz, {\em On the Kontsevich-Zorich cocycle over McMullen's family of symmetric translation surfaces}, in preparation. 

\bibitem[AV]{AVKZ} A. Avila and M. Viana, {\em Simplicity of Lyapunov spectra: proof of the Zorich-Kontsevich conjecture}, 
Acta Math. \textbf{198} (2007), 1--56.

\bibitem[AV2]{AVPort} A. Avila and M. Viana, {\em Simplicity
of Lyapunov spectra: a sufficient criterion}, Port. Math. \textbf{64} (2007), 311--376. 

\bibitem[Ba]{Bainbridge} M. Bainbridge, \emph{Euler characteristics of \Teichmuller curves in genus two},
Geom. Topol. 11 (2007), 1887--2073.

\bibitem[Bi]{Bilu} Y. Bilu, \emph{Quantitative Siegel's theorem for Galois coverings},
Compositio Math. 106 (1997), 125--158. 


\bibitem[Bo]{Bourbaki-CA} N. Bourbaki, \emph{Elements of mathematics. Commutative algebra}, 
Translated from the French. Hermann, Paris; Addison-Wesley Publishing Co., Reading, Mass., 1972. xxiv+625 pp. 

\bibitem[Bo1]{Bourbaki-A9} N. Bourbaki, \emph{\'El\'ements de math\'ematique. Alg\`ebre. Chapitre 9}, Reprint of the 1959 original. Springer-Verlag, Berlin, 2007. 211 pp. 

\bibitem[BL]{BiLa} C.~Birkenhake, H.~Lange, {\em Complex abelian
varieties}, Springer Grundlehren 302, 2nd.~ed.\ (2003).


\bibitem[Ch]{chencovers} D. Chen, {\em Square-tiled surfaces and rigid curves on moduli spaces}, Adv. Math. \textbf{228} (2011), 1135--1162. 

\bibitem[CM]{chenmoeller} D. Chen and M. M\"oller, {\em Non-varying sums of Lyapunov exponents of Abelian differentials in low genus}, Geometry and Topology \textbf{16} (2012), 2427--2479.


\bibitem[Da]{Da} F. Dalbo, {\em Geodesic and horocyclic trajectories}, 
Translated from the 2007 French original. Universitext. Springer-Verlag London, Ltd., London; EDP Sciences, Les Ulis (2011) xii+176 pp.

\bibitem[DM]{DeMa} V. Delecroix and C. Matheus, {\em Un contre-exemple \`a la r\'eciproque du crit\`ere de 
Forni pour la positivit\'e des exposants de Lyapunov du cocycle de Kontsevich-Zorich}, preprint arXiv:1103.1560 (2011), 1--7, to appear in Math. Res. Lett.


\bibitem[EKZ]{ekz} A. Eskin, M. Kontsevich and A. Zorich, {\em Sum of Lyapunov exponents of the Hodge bundle with respect to the \Teichmuller geodesic flow}, preprint arxiv: 1112.5872 (2011), 1--109, to appear in  Publ. Math. Inst. Hautes \'Etudes Sci.

\bibitem[EMas]{EsMas} A. Eskin and H. Masur, {\em Asymptotic formulas on flat surfaces}, Ergod. Theory Dynam. Systems \textbf{21} (2001), 443--478.

\bibitem[EMat]{EsMat} A. Eskin and C. Matheus, {\em A coding-free simplicity criterion for the Lyapunov exponents of \Teichmuller curves}, preprint arXiv:1210.2157 (2012), 1--25, to appear in  Geom. Dedicata

\bibitem[EM]{em} A. Eskin and M. Mirzakhani, {\em Invariant and stationary measures for the $\SL(2,\mathbb{R})$ action on moduli space}, preprint arXiv:1302.3320 (2013), 1--162.

\bibitem[EMM]{emm} A. Eskin, M. Mirzakhani and A. Mohammadi, {\em Isolation, equidistribution, and orbit closures for the $\SL(2,\Rset)$ action on moduli space}, preprint arXiv:1305.3015 (2013), 1--45.


\bibitem[Fo1]{Fo1} G. Forni, {\em Deviation of ergodic averages for area-preserving flows on surfaces of higher genus},
Ann. of Math. \textbf{155} (2002), 1--103.

\bibitem[Fo2]{Fo2} G. Forni, {\em On the Lyapunov exponents
of the Kontsevich-Zorich cocycle}, Handbook of dynamical systems. Vol. 1B, 549--580, Elsevier B. V., Amsterdam, 2006.

\bibitem[Fo3]{Fo3} G. Forni, {\em A geometric criterion for the nonuniform hyperbolicity of the Kontsevich-Zorich cocycle}. With an appendix by C. Matheus. 
J. Mod. Dyn. \textbf{5} (2011), 355--395.

\bibitem[FMZ1]{FMZ} G. Forni, C. Matheus and A. Zorich, {\em Square-tiled cyclic covers}, J. Mod. Dyn. \textbf{5} (2011), 285--318.

\bibitem[FMZ2]{FMZ2} G. Forni, C. Matheus and A. Zorich, {\em Zero Lyapunov exponents of the Hodge bundle}, Comment. Math. Helv. \textbf{89} (2014), 489--535.

\bibitem[Ga]{Gauss} C. F. Gauss, \emph{Disquisitiones arithmeticae}, 
Translated and with a preface by Arthur A. Clarke. Revised by William C. Waterhouse, Cornelius Greither and A. W. Grootendorst and with a preface by Waterhouse. Springer-Verlag, New York, 1986. xx+472 pp. ISBN: 0-387-96254-9 

\bibitem[Ha]{Hartshorne} R. Hartshorne, \emph{Algebraic geometry}, 
Graduate Texts in Mathematics, \textbf{52}, Springer-Verlag, New York-Heidelberg, 1977. xvi+496 pp. ISBN: 0-387-90244-9

\bibitem[HS]{HiSi} M. Hindry and J. Silverman, \emph{Diophantine geometry}, Graduate Texts in Mathematics, 201. Springer-Verlag, New York, 2000. xiv+558 pp. ISBN: 0-387-98975-7; 0-387-98981-1

\bibitem[HL]{hubertlelievre} P. Hubert and S. Leli\`evre, {\em Prime arithmetic \Teichmuller discs in $\mathcal{H}(2)$}, Israel J. Math. \textbf{151} (2006), 281--321.

\bibitem[Ka]{kani} E. Kani, {\em The number of genus 2 covers of an elliptic curve}, Manuscripta Math. \textbf{121} (2006), 51--80.

\bibitem[Ko]{K} M. Kontsevich, {\em Lyapunov
exponents and Hodge theory}, in {\em The Mathematical Beauty of Physics} (Saclay 1996),
{\em Adv. Ser. Math. Phys.} \textbf{24}, 318--332, World Sci. Publ., River Edge, NJ, 1997.

\bibitem[KZ]{KoZo03} M. Kontsevich and A. Zorich, {\em Connected
components of the moduli space of Abelian differentials with
prescribed singularities},  Invent. Math. \textbf{153} (2003), 631--678.

\bibitem[LN]{manhlann} E. Lanneau and D.-M. Nguyen, {\em \Teichmuller curves generated by Weierstrass Prym eigenforms in 
genus $3$ and genus $4$}, J. Topol. \textbf{7} (2014), 475--522.


\bibitem[MY]{MY} C. Matheus and J.-C. Yoccoz, {\em The action of the affine diffeomorphisms on the relative homology group of certain exceptionally symmetric origamis}, 
J. Mod. Dyn. \textbf{4} (2010), 453--486.

\bibitem[MYZ]{MYZ} C. Matheus, J.-C. Yoccoz and D. Zmiaikou, {\em Homology of origamis with symmetries}, preprint arXiv:1207.2423 (2012), 1--35, to appear in Ann. Inst. Fourier (Grenoble).

\bibitem[Ma]{Ma82} H. Masur, {\em Interval exchange transformations and measured foliations}, 
Ann. of Math. \textbf{115} (1982), no. 1, 169--200. 

\bibitem[Mc]{mcmullenprym} C. McMullen, {\em Prym varieties and \Teichmuller curves}, Duke Math. J. \textbf{133} (2006), 569--590.

\bibitem[Mc2]{McM} C. McMullen, {\em \Teichmuller curves in genus two: discriminant and spin}, 
Math. Ann. \textbf{333} (2005), 87--130.

\bibitem[Mo]{moelprym} M. M\"oller, {\em Prym covers, theta functions and Kobayashi curves in Hilbert modular surfaces}, Amer. J. Math. \textbf{136} (2014), 995--1021.



\bibitem[Se]{S} C. Series, {\em Geometrical methods of symbolic coding}. Ergodic theory, symbolic dynamics, and hyperbolic spaces (Trieste, 1989), 125--151,
Oxford Sci. Publ., Oxford Univ. Press, New York, 1991. 

\bibitem[SW]{SW} J. Smillie and B. Weiss, \emph{Characterizations of lattice surfaces}, Invent. Math. 180 (2010), 535--557.

\bibitem[Ve1]{V82} W. Veech, {\em Gauss measures for transformations on the space of interval exchange maps}, Ann. of Math. \textbf{115} (1982), 201--242.

\bibitem[Ve2]{V86} W. Veech, {\em The \Teichmuller geodesic flow}, 
Ann. of Math. \textbf{124} (1986), 441--530.

\bibitem[Ve3]{Ve89} W. Veech, {\em \Teichmuller curves in moduli
space, Eisenstein series and an application to triangular billiards},
Invent.~Math. \textbf{97} (1989), 533--583.

\bibitem[Ve4]{V90} W. Veech, \emph{Moduli spaces of quadratic differentials}, J. Anal. Math. \textbf{55} (1990), 117--171.

\bibitem[Zm]{zmia} D. Zmiaikou, {\em Origamis and permutation groups}, Ph.D thesis, Orsay, 2011, available at \texttt{http://www.zmiaikou.com/research}.

\bibitem[Zo1]{Zo1} A. Zorich, {\em Asymptotic flag of an orientable measured foliation on a surface}, in {\em Geometric Study of Foliations} (Tokyo, 1993), 
479--498, World Sci. Pub., River Edge, NJ, 1994.

\bibitem[Zo2]{Zo2} A. Zorich, {\em Finite Gauss measure on the space of interval exchange transformations. Lyapunov exponents}, 
Ann. Inst. Fourier (Grenoble) \textbf{46} (1996), 325--370.

\bibitem[Zo3]{Zo3} A. Zorich, {\em Deviation for interval exchange transformations}, Ergod. Theory Dynam. Systems \textbf{17} (1997), 1477--1499.

\bibitem[Zo4]{Zo4} A. Zorich, {\em On hyperplane sections of periodic surfaces}, in {\em Solitons, Geometry and Topology: on the Crossroad}, A. M. S. Transl. \textbf{179}, 173--189, A. M. S., Providence, RI, 1997.

\bibitem[Zo5]{Zo5} A. Zorich, {\em How do the leaves of a closed $1$-form wind around a surface?}, in {\em Pseudoperiodic Topology}, A. M. S. Transl. \textbf{197}, 
135--178, A. M. S., Providence, RI, 1999. 

\bibitem[Zo6]{Zo6} A. Zorich, {\em Flat surfaces}, ``Frontiers in Number Theory, Physics, and Geometry'', v. \textbf{I}, Springer (2006), 437--583.

\bibitem[Zo7]{Zo7} A. Zorich, {\em Square tiled surfaces and \Teichmuller volumes of the moduli spaces of abelian differentials}, (English summary) Rigidity in dynamics and geometry (Cambridge, 2000), 459--471, Springer, Berlin, 2002.

\end{thebibliography}


\end{document}